\documentclass[11pt]{amsart}

\usepackage[english]{babel}

\usepackage{amsmath} 
\usepackage{graphicx} 
\usepackage[colorlinks=true, allcolors=blue]{hyperref} 
\usepackage{amsthm} 
\usepackage{amssymb}
\usepackage{amsfonts,amscd} 
\usepackage[final]{showkeys} 
\usepackage{fullpage}
\usepackage{subfig}
\usepackage{enumitem}

\usepackage{tikz} 
\usetikzlibrary{positioning,automata,arrows,shapes,matrix,
arrows, decorations.pathmorphing, calc} 
\tikzset{main node/.style={circle,draw,minimum
size=0.3em,inner sep=0.5pt}} 
\tikzset{state node/.style={circle,draw,minimum
size=2em,fill=blue!20,inner sep=0pt}} 
\tikzset{small node/.style={circle,draw,minimum size=0.5em,inner
sep=2pt,font=\sffamily\bfseries}}

\usepackage{xspace} 
\usepackage{adjustbox} 

\theoremstyle{plain} 
\newtheorem{theorem}{Theorem}[section]
\newtheorem{lemma}[theorem]{Lemma} 
\newtheorem{prop}[theorem]{Proposition}
 
\newtheorem{cor}[theorem]{Corollary}
\theoremstyle{definition} 
\newtheorem{defn}[theorem]{Definition}
\newtheorem{exa}[theorem]{Example} 

\theoremstyle{remark} 
\newtheorem{rmk}[theorem]{Remark}
\numberwithin{equation}{subsection}
\newcommand{\al}{\alpha}
\newcommand{\be}{\beta} 
\newcommand{\ep}{\varepsilon}
\newcommand{\Q}{{\mathbb Q}} 
\newcommand{\R}{{\mathbb R}}
\newcommand{\Z}{{\mathbb Z}} 
\newcommand{\C}{{\mathbb C}}
\newcommand{\ra}{\rightarrow} 

\newcommand{\bfw}{\mathbf{w}} 
\newcommand{\Be}{\mathcal B}

\newcommand{\fh}{V}

\newcommand{\Aut}{\text{\rm Aut}}

\newcommand{\coplanar}{coplanar\xspace} 
\newcommand{\simproots}{\Pi}
\newcommand{\posroots}{\Phi^+} 
\newcommand{\allroots}{\Phi}

\newcommand{\Sym}{\text{\rm Sym}}

\newcommand{\rootht}{\text{\rm ht}}

\newcommand{\macd}[2]{j^{#1}_{#2}({\rm{sgn}})}

\newcommand{\se}{\subseteq}

\newcommand{\inverse}{^{-1}}
 
\newcommand{\abs}[1]{\lvert #1 \rvert}

\makeatletter 
\newcommand{\exterior}[1]{\mathop{\mathpalette\exterior@{#1}}}
\newcommand{\exterior@}[2]{%
\raisebox{\depth}{%
\fontsize{\sf@size}{0}%
\m@th
$\ifx#1\displaystyle\textstyle\else#1\fi\bigwedge$}%
^{\mspace{-2mu}#2}%
\kern-\scriptspace } 
\makeatother

\title{Orthogonal roots, Macdonald representations, and quasiparabolic
sets} 
\author{R.M. Green}
\address{Department of Mathematics, University of Colorado Boulder, Campus Box
395, Boulder, Colorado, USA, 80309}
\email{rmg@colorado.edu}

\author{Tianyuan Xu}
\address{Department of Mathematics and Statistics, University of Richmond,
Richmond, Virginia, USA, 23173}
\email{tianyuan.xu@richmond.edu}


\begin{document} \maketitle

\begin{abstract}
Let $W$ be a simply laced Weyl group of finite type and rank $n$.  If $W$ has
type $E_7$, $E_8$, or $D_n$ for $n$ even, then the root system of $W$ has
subsystems of type $nA_1$.  This gives rise to an irreducible Macdonald
representation of $W$ spanned by $n$-roots, which are products of $n$ orthogonal
roots in the symmetric algebra of the reflection representation.  We prove that
in these cases, the set of all maximal sets of orthogonal positive roots has the
structure of a quasiparabolic set in the sense of Rains--Vazirani.  The
quasiparabolic structure can be described in terms of certain quadruples of
orthogonal positive roots which we call crossings, nestings, and alignments.
This leads to nonnesting and noncrossing bases for the Macdonald representation,
as well as some highly structured partially ordered sets. We use the $8$-roots
in type $E_8$ to give a concise description of a graph that is known to be
non-isomorphic but quantum isomorphic to the orthogonality graph of the $E_8$
root system.
\end{abstract}


\section{Introduction}

If $W$ is a finite simply laced Weyl group, then it is possible to find a basis
of the reflection representation $\fh$ of $W$ that consists of orthogonal
positive roots when $W$ has type $E_7$, $E_8$, or $D_n$ for $n$ even. The goal
of this paper is to demonstrate that the set $X=X(W)$ of such bases has a rich
combinatorial structure, both by identifying $X$ with a subset of a Macdonald
representation of $W$, and by regarding $X$ as a quasiparabolic set in the sense
of Rains and Vazirani \cite{rains13}.  A {quasiparabolic set} for a Weyl group
$W$ is a $W$-set equipped with an integer-valued height function satisfying two
axioms that specify how the action of a reflection changes the height. The
axioms generalize properties satisfied by quotients of arbitrary Coxeter groups
by their parabolic subgroups that allow one to deform the action of Coxeter
group on the quotient to create a corresponding module for the Iwahori--Hecke
algebra of the group. 

Let $W$ be a finite Weyl group with root system $\allroots$ and Dynkin diagram
$\Gamma$. Let $\fh$ be the reflection representation of $W$ defined over $\Q$
(see Section \ref{sec:basics}), and let $\fh^*$ be the dual of $V$ over $\Q$.
The space of all rational-valued polynomial functions on $V$ is then the
symmetric algebra $\Sym(\fh^*)$, which is a $W$-module via the contragredient
action $(w\cdot\phi)(x) = \phi(w^{-1}(x))$.   Let $\Psi$ be a {\it subsystem} of
$\allroots$, meaning that $\emptyset \ne \Psi \subset \allroots$ and that $\Psi$
is also a root system, and let $\posroots$ and $\Psi^+$ be the set of positive
roots in $\allroots$ and $\Psi$, respectively. Following \cite{macdonald72}, we
use the (positive definite) inner product on $\fh$  to identify $\fh$ with
$\fh^*$, and we define the {\it Macdonald representation} $\macd{\Gamma}{\Psi} =
\macd{\allroots}{\Psi}$ of $W$ given by $\Psi$ to be the cyclic $\Q W$-submodule
of $\Sym(\fh^*)\cong \Sym(\fh)$ generated by $\pi_\Psi$, where $$ \pi_\Psi =
\prod_{\al \in \Psi^+} \al .$$ A short argument shows that the Macdonald
representation is an absolutely irreducible $W$-module.

When $W$ has type $E_7, E_8$, or $D_n$ for $n$ even, $\Phi$ contains subsystems
of type $nA_1$.  The set $\Psi^+$ for any such subsystem $\Psi$ consists of $n$
orthogonal positive roots, and we call an element of the form $w.\pi_\Psi$ (for
$w \in W$) an {\it $n$-root}. Thus, an $n$-root has the form $\al = \prod_{i =
1}^n \be_i$ where the elements $\be_i$ are orthogonal roots. Conversely, since
$W$ acts transitively on the set of maximal  sets of orthogonal roots (Lemma
\ref{lem:transitive}), every product of $n$ orthogonal roots is an $n$-root.
The transitivity of this $W$-action also implies that any two subsystems of type
$nA_1$ in $\Phi$ give rise to the same Macdonald representation, which we will
thus simply denote as $\macd{\Phi}{nA_1}$. The  representation
$\macd{\Phi}{nA_1}$ and the $n$-roots within it are the central objects of study
in this paper, and we summarize their definition below.

\begin{defn}
    \label{def:mac}
    Let $\Phi$ be a root system of type $E_7, E_8$, or $D_n$ for $n$ even. Let
    $W$ be the Weyl group of $\Phi$, and let $V$ be the reflection
    representation of $W$. We denote the Macdonald representation
    $\macd{\allroots}{\Psi}\subset \Sym(\fh^*)\cong \Sym(\fh)$ arising from any
    subsystem $\Psi$ of type $nA_1$ in $\Phi$ by $\macd{\Phi}{nA_1}$. We call
    each element of the form $\al=\prod_{i=1}^n \be_i\in \macd{\Phi}{nA_1}$
    where $\be_1,\cdots,\be_n$ are orthogonal roots of $\Phi$ an \emph{$n$-root}
    of $W$. 
\end{defn}

Given an $n$-root $\al=\prod_{i=1}^n\be_i$, the factors $\be_i$ are unique up to
reordering and multiplication by nonzero scalars, because they are the
irreducible factors of $\al$ in the unique factorization domain $\Q[\al_1,
\al_2, \ldots, \al_n]$, where the $\al_i$ correspond to the simple roots of
$\allroots$.  We say $\al$ is {\it positive} if all the factors $\be_i$ may be
taken to be positive or, equivalently, if evenly many of the components are
negative. If $\al=\prod_{i=1}^n \be_i$ is a positive root with all the $\be_i$
positive, we call the roots $\be_i$ the \emph{components} of $\al$. An $n$-root
$\al$ is {\it negative} if $-\al$ is positive. It is immediate from the
definitions that if $\al$ is an $n$-root then $-\al$ is also an $n$-root, and
that precisely one of $\al$ and $-\al$ is positive, similarly to how roots
appear in positive-negative pairs in ordinary root systems. If $\al$ is either a
root or an $n$-root, we define the \emph{absolute value} of $\al$, denoted
$\abs{\al}$, to be the positive element in the pair $\{\al,-\al\}$.  (We may
    view both ordinary roots and $n$-roots as special cases of \emph{$k$-roots},
    by which we mean products in $\Sym(\fh^*)$ of $k$ orthogonal roots of $W$
    for any fixed integer $1\le k\le n$.  The notion of $k$-roots plays an
    important role in our previous papers \cite{gx3} and \cite{green23}, and we
    will occasionally speak of $4$-roots, even when $n\neq 4$, in this paper.)

The set $\Phi^+_n$ of all positive $n$-roots admits a natural $W$-action given
by $w(\al)=\abs{w(\al)}$. Similarly, the set $X$ of sets of $n$ orthogonal roots
admits a natural $W$-action given by
$w(\{\be_1,\cdots,\be_n\})=\{\abs{w(\be_1)}, \cdots, \abs{w(\be_n)}\}$.  The map
sending each set $\{\be_1,\cdots,\be_n\}\in X$ to the product $\prod_{i=1}^n
\be_i\in \Phi^+_n$ respects these two $W$-actions, and we use it to identify $X$
with $\Phi^+_n$. In other words, we identify each positive $n$-root with its set of components.

We show that $X$ has the structure of a quasiparabolic set under a suitable
height function $\lambda$ (Theorem \ref{thm:qp}).  As we explain in sections
\ref{sec:crossnest} and \ref{sec:quasi}, to understand this structure it is
useful to consider quadruples $Q = \{\be_1, \be_2, \be_3, \be_4\}$ of four
orthogonal roots with the property that $(\be_1 + \be_2 + \be_3 + \be_4)/2$ is
also a root. We call such quadruples \emph{coplanar quadruples}, and show that
they fall into three distinct types, called {\it crossings}, {\it nestings}, and
{\it alignments}.  The height function $\lambda$ is given by
$\lambda(\gamma)=C(\gamma)+2N(\gamma)$, where $C(\gamma)$ and $N(\gamma)$ are
the numbers of crossings and nestings in $\gamma$, respectively, for each
$\gamma\in X$. The terms ``crossing'', ``nesting'', and ``alignment'' are
motivated by the theory of perfect matchings (Remark \ref{rmk:cna}).

As a quasiparabolic set, the set $X$ is equipped with a partial order $\le_Q$,
which is the weakest partial order such that $x\le_Q rx$ whenever $r$ is a
reflection such that $\lambda(x)\le\lambda(rx)$. We use the theory of
quasiparabolic sets to prove that $X$ has a unique maximally aligned $n$-root,
$\theta_A$, and a unique maximally nesting $n$-root, $\theta_N$; these two
elements are the unique minimal and maximal element of $X$ with respect to
$\le_Q$, respectively (Proposition \ref{prop:minmax}).  The $n$-roots that avoid
alignments, or the \emph{alignment-free} $n$-roots, form a quasiparabolic set
$X_I \subset X$ of a certain maximal parabolic subgroup $W_I$ of $W$. The
corresponding partial order on $X_I$ allows us to show that $X$ also has a
maximally crossing element, $\theta_C$, and that it is the unique minimal
element of $X_I$ (Proposition \ref{prop:neststab} (iii)). The set $X_I$ has a
natural bipartite structure, with the $n$-roots in $X_I$ with even levels and
those with odd levels partitioning $X_I$ into two equal-sized components that
are interchanged by every reflection in $W_I$ (Remark \ref{rmk:xht}).

The alignment-free $n$-roots in $X$ are one of three families that avoid a
particular type of coplanar quadruple, the other two being those that avoid
crossings and avoid nestings. Section \ref{sec:non} studies these three families
together. We use a version of Bergman's diamond lemma to show that the
noncrossing elements of $X$ form a basis for the Macdonald representation, as do
the nonnesting positive $n$-roots (Theorem \ref{thm:indep}). In addition, the
noncrossing basis behaves somewhat like a
simple system in a root system (Theorem \ref{thm:positive}) and may be viewed as
a canonical basis (Remark \ref{rmk:othersimple}).  The nonnesting
basis is naturally parametrized by a particular interval $[1, w_N]$ in the weak
Bruhat order of $W$ and has the structure of a distributive lattice (Theorem
\ref{thm:nonnest}). The element $w_N$, which we call the \emph{nonnesting
element} of $W$, is the unique shortest element taking the maximally crossing
$n$-root $\theta_C$ to the maximally aligned element $\theta_A$.  We note that
the {noncrossing basis} is essentially the same as a basis that appears in the
work of Fan \cite[Section 6]{fan97} and others, although the realization of the
basis as polynomials in roots seems to be new.

We say that two positive $n$-roots are \emph{sum equivalent} or {\it
$\sigma$-equivalent} if their sets of components have the same sum. We show that
the $\sigma$-equivalence classes are in canonical bijection with the nonnesting
and the noncrossing elements of $X$ in the following way: each
$\sigma$-equivalence class is an interval $[\be_1,\be_2]_Q=\{\gamma\in
X:\beta_1\le_Q \gamma\le_Q \beta_2\}$ in the quasiparabolic
order $\le_Q$, where the minimal element $\be_1$ and maximal element $\be_2$ are
the unique nonnesting and noncrossing $n$-roots in the class, respectively. The alignment-free
elements in $X$ form a single $\sigma$-equivalence class that is maximal with
respect to a natural order (Proposition \ref{prop:sumclass}).  Any set of
$\sigma$-equivalence class representatives forms a basis for the Macdonald
representation, and the change of basis matrix between any two such bases, such
as the one between the nonnesting and noncrossing basis, is always unitriangular
with integer entries once the bases are suitably ordered (Theorem \ref{thm:nesttri}).

As we explain in Section \ref{sec:abstract_fa}, the feature-avoiding $n$-roots
and $\sigma$-equivalence classes in the set $X$ can all be characterized
abstractly using the quasiparabolic structure of $X$, without using the
combinatorics of sets of roots.  This raises
the possibility of extending the notions of alignment-free, noncrossing, and
nonnesting elements to more general quasiparabolic sets.

The theory developed in this paper has various natural connections to many
previous works. In type $D_n$ for an even integer $n=2k$, the $n$-roots
correspond naturally to perfect matchings of the set $[n]=\{1,2,\dots,n\}$, and
the crossings, nestings, and alignments in $n$-roots recover the corresponding
notions in the theory of matchings.  Besides matchings, the quasiparabolic set
$X$ can be identified with the set of fixed-point free involutions in $S_n$,
which is one of the original motivating examples of a quasiparabolic set
\cite[Section 4]{rains13}.  The level function $\lambda=C+2N$ appears as a
useful statistic on matchings in \cite{simion96,diaz09,cheon15}, and has a
natural interpretation in the context of combinatorial game theory
\cite{irie21}; see Section \ref{sec:D}.

The Macdonald representation $\macd{\Phi}{nA_1}$ in type $D_n$ for $n$ even
recovers a Specht module in a very natural way: the action of the Weyl group $W$
factors through an obvious sign-forgetting map (Equation \eqref{eq:phimap}) to
induce an $S_n$-module structure on the Macdonald representation for the
symmetric group $S_n$, and the resulting module is isomorphic to a realization
of the Specht module corresponding to the two-row partition $(k,k)$ due to
Rhoades \cite{rhoades17} (Proposition \ref{prop:specht}). The noncrossing bases
and nonnesting bases have been studied extensively as the web basis and the
Specht basis, respectively, of the Specht module
\cite{russell19,im22,hwang23,kujawa24}; see Section \ref{sec:D}.

In type $E_7$, the Macdonald representation contains 135 positive 7-roots and
has degree 15 \cite[Proposition 4.12]{colombo10}. This representation has a long
history, going back the work of Coble in 1916 \cite[(65)]{coble1916} on the
G\"opel variety. There are also applications of $7$-roots to quantum information
theory and supergravity \cite[Section IV G]{cerchiai10}, \cite{duff07}. In this
case, the elements of the quasiparabolic set $X_I$ are in canonical bijection
with the 30 distinct labellings of the Fano plane, and the maximal and minimal
elements are given by $\{136, 145, 127, 235, 246, 347, 567\}$ and $\{123, 145,
246, 257, 347, 356, 167\}$, respectively (Proposition \ref{prop:e7top}).

In type $E_8$, there are 2025 positive $8$-roots. The bases of orthogonal roots
have applications to physics, where they can be used to prove the
Kochen--Specker theorem in quantum mechanics \cite{waegell15} (Section
\ref{sec:coxeter_elements}).  The Macdonald representation in this case has
degree $50$, but seems not to have been studied much before. The quasiparabolic
set $X_I$ in this case is a bipartite structure with 240 elements. As we explain
in Section \ref{sec:e8}, either partite component can be used to define a graph
that has an interesting relationship with two
strongly regular graphs studied recently by Schmidt \cite{schmidt24} (Remark
\ref{rmk:quantumiso}). Those two
graphs each have 120 vertices, and they have the remarkable property of being
quantum isomorphic (in the sense of \cite{atserias19}) but not isomorphic.

The properties of $n$-roots summarized in the last three paragraphs are
explained in more detail in Section \ref{sec:examples}. It is worth noting that
while these properties are type-specific, we have attempted to develop the
theory of $n$-roots in a type-independent way in the other parts of the paper in
general. In particular, we give a uniform proof for the fact that 
the positive $n$-roots form a quasiparabolic set in types $E_7, E_8,$ and $D_{2k}$ (Theorem \ref{thm:qp}).
While it is possible to
verify the theorem for types $E_7$ and $E_8$ using direct computation (which we
did, using the software SageMath \cite{sagemath}) and then separately deduce the
theorem for type $D_{2k}$ by considering the $S_{2k}$-action on its fixed-point
free involutions, our uniform proof of the theorem relying on Proposition
\ref{prop:miracle} has the advantage of being more conceptual and revealing more
details about the action of reflections on $n$-roots. Some of these details will
be further used in a forthcoming paper \cite{gx6}, where we will generalize
aspects of this paper and study quasiparabolic sets arising from $k$-roots for
more general values of $k$.

The rest of the paper is organized as follows. We recall the basics of root
systems in Section \ref{sec:roots}.  Section \ref{sec:crossnest} introduces the
key notions of crossings, nestings, and alignments in an $n$-root, and we
connect them to the theory of quasiparabolic sets in Section \ref{sec:quasi}.
 Section \ref{sec:non}
studies the alignment-free, noncrossing, and nonnesting $n$-roots.  Section
\ref{sec:examples} discusses the details of $n$-roots in the types $D_n$ with
$n$ even, $E_7$ and $E_8$. Section \ref{sec:conclusions} concludes the paper,
and includes discussions of the Poincar\'e polynomial of the set $X$ and of
orbits of $n$-roots under the action of Coxeter elements.

\section{Review of root systems}\label{sec:roots}

In this section, we recall the basic properties of simply laced root
systems of finite type. We will mostly follow the notation of the first two
chapters of \cite{humphreys90}, except in the case of type $E_7$, where we
follow \cite[Section 4]{green08}.

\subsection{Weyl groups, root systems, and reflection representations} 
\label{sec:basics}
The root systems in this paper will be irreducible simply laced root systems of finite type,
whose Dynkin diagrams are shown in Figure \ref{fig:ade}.  The vertices of the
Dynkin diagram $\Gamma$ index the {\it simple roots} $\simproots = \{\alpha_i :
    i \in \Gamma\}$. The {\it root lattice} $\Z\simproots$ is the free
    $\Z$-module on $\simproots$. We define a $\Z$-bilinear form $B$ on
    $\Z\simproots \times \Z\simproots$ by $$ B(\alpha_i, \alpha_j) =
    \begin{cases} 2 & \text{\ if \ } i = j;\\ -1 & \text{\ if \ $i$\ and\ $j$\
        are\ adjacent\ in\ } \Gamma;\\ 0 & \text{\ otherwise.}\\
    \end{cases} $$ If $\al_i \in \simproots$, then we define the {\it simple
    reflection} $s_i=s_{\al_i}$ to be the $\Z$-linear operator $\Z\simproots
    \rightarrow \Z\simproots$ given by $$ s_i(\be) = \be - B(\al_i, \be)\al_i
    .$$ The {\it Weyl group} $W = W(\Gamma)$ is the finite group generated by
    the simple reflections.


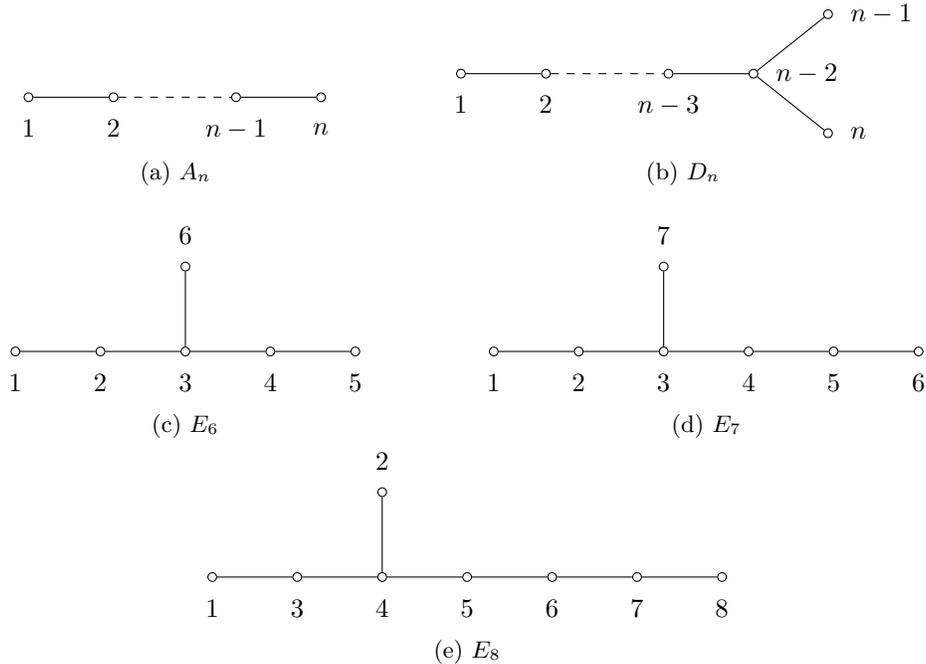
\begin{figure}[h!]
    \centering
    \subfloat[{$A_n$}]
    {
\begin{tikzpicture}

            \node[main node] (1) {};
            \node[main node] (2) [right=1cm of 1] {};
            \node[main node] (3) [right=1.5cm of 2] {};
            \node[main node] (4) [right=1cm of 3] {};

            \node (11) [below=0.1cm of 1] {\small{$1$}};
            \node (22) [below=0.1cm of 2] {\small{$2$}};
            \node (33) [below=0.1cm of 3] {\small{$n-1$}};
            \node (44) [below=0.15cm of 4] {\small{$n$}};

            \path[draw]
            (1)--(2)
            (3)--(4);

            \path[draw,dashed]
            (2)--(3);
\end{tikzpicture}
}
\quad\quad\quad
\subfloat[{$D_n$}]
{
\begin{tikzpicture}

            \node[main node] (1) {};
            \node[main node] (2) [right=1cm of 1] {};
            \node[main node] (3) [right=1.5cm of 2] {};
            \node[main node] (4) [right=1cm of 3] {};
            \node[main node] (5) [above right=0.7cm and 0.9cm of 4] {};
            \node[main node] (6) [below right=0.7cm and 0.9cm of 4] {};

            \node (11) [below=0.1cm of 1] {\small{$1$}};
            \node (22) [below=0.1cm of 2] {\small{$2$}};
            \node (33) [below=0.1cm of 3] {\small{$n-3$}};
            \node (44) [right=0.1cm of 4] {\small{$n-2$}};
            \node (55) [right=0.1cm of 5] {\small{$n-1$}};
            \node (66) [right=0.1cm of 6] {\small{$n$}};

            \path[draw]
            (1)--(2)
            (3)--(4)--(5)
            (4)--(6);

            \path[draw,dashed]
            (2)--(3);
\end{tikzpicture}
}\\
\subfloat[{$E_6$}]
{        
\begin{tikzpicture}
    \node[main node] (1) {};
    \node[main node] (2) [right=1cm of 1] {};
            \node[main node] (3) [right=1cm of 2] {};
            \node[main node] (4) [right=1cm of 3] {};
            \node[main node] (5) [right=1cm of 4] {};
            \node[main node] (6) [above=1cm of 3] {};

            \path[draw]
            (1)--(2)--(3)--(4)--(5)
            (3)--(6);

            \node (11) [below=0.1cm of 1] {\small{$1$}};
            \node (22) [below=0.1cm of 2] {\small{$2$}};
            \node (33) [below=0.1cm of 3] {\small{$3$}};
            \node (44) [below=0.1cm of 4] {\small{$4$}};
            \node (55) [below=0.1cm of 5] {\small{$5$}};
            \node (66) [above=0.1cm of 6] {\small{$6$}};
\end{tikzpicture}
}
\quad\quad\quad
\subfloat[{$E_7$}]
{
\begin{tikzpicture}
    \node[main node] (1) {};
    \node[main node] (2) [right=1cm of 1] {};
            \node[main node] (3) [right=1cm of 2] {};
            \node[main node] (4) [right=1cm of 3] {};
            \node[main node] (5) [right=1cm of 4] {};
            \node[main node] (6) [right=1cm of 5] {};
            \node[main node] (7) [above=1cm of 3] {};

            \path[draw]
            (1)--(2)--(3)--(4)--(5)--(6)
            (3)--(7);

            \node (11) [below=0.1cm of 1] {\small{$1$}};
            \node (22) [below=0.1cm of 2] {\small{$2$}};
            \node (33) [below=0.1cm of 3] {\small{$3$}};
            \node (44) [below=0.1cm of 4] {\small{$4$}};
            \node (55) [below=0.1cm of 5] {\small{$5$}};
            \node (66) [below=0.1cm of 6] {\small{$6$}};
            \node (77) [above=0.1cm of 7] {\small{$7$}};
\end{tikzpicture}
}
\quad
\subfloat[{$E_8$}]
{
\begin{tikzpicture}
    \node[main node] (1) {};
            \node[main node] (3) [right=1cm of 1] {};
            \node[main node] (4) [right=1cm of 3] {};
            \node[main node] (5) [right=1cm of 4] {};
            \node[main node] (6) [right=1cm of 5] {};
            \node[main node] (7) [right=1cm of 6] {};
            \node[main node] (8) [right=1cm of 7] {};
            \node[main node] (2) [above=1cm of 4] {};

            \path[draw]
            (1)--(3)--(4)--(5)--(6)--(7)--(8)
            (2)--(4);

            \node (11) [below=0.1cm of 1] {\small{$1$}};
            \node (22) [above=0.1cm of 2] {\small{$2$}};
            \node (33) [below=0.1cm of 3] {\small{$3$}};
            \node (44) [below=0.1cm of 4] {\small{$4$}};
            \node (55) [below=0.1cm of 5] {\small{$5$}};
            \node (66) [below=0.1cm of 6] {\small{$6$}};
            \node (77) [below=0.1cm of 7] {\small{$7$}};
            \node (88) [below=0.1cm of 8] {\small{$8$}};

\end{tikzpicture}
}
\caption{Dynkin diagrams of irreducible simply laced Weyl groups}
\label{fig:ade}
\end{figure}

The {\it root system} of $W$  is the set $\allroots = \{w(\al_i) : \al_i \in
\Pi, w \in W\}$. Each element of $\allroots$ is called a \emph{root}. The group
$W$ acts transitively on $\allroots$, and the form $B$ is $W$-invariant in the
sense that $B(\al, \be) = B(w(\al), w(\be))$ for all $w \in W$ and all $\al, \be
\in \allroots$. We say two roots are $\al,\beta\in \allroots$ are
\emph{orthogonal} if $B(\al,\be)=0$.

Each root $\alpha\in \Phi$ gives rise to a \emph{reflection} in $W$, which is
the self-inverse $\Z$-linear operator $s_\al:\Z\Pi\ra\Z\Pi$ generalizing simple
reflections and given by the formula 
\begin{equation}
    \label{eq:ref} s_\al(\be) = \be - B(\al, \be)\al.
\end{equation} The reflections in $W$ form a single conjugacy class. The
    $\Q$-vector space $\fh := \Q \otimes_\Z \Z\simproots$ affords the {\it
    reflection representation} of $W$, where each reflection $s_\al$ acts by
    Equation (\ref{eq:ref}). 

A subset $\Psi$ of $\Phi$ is called a \emph{subsystem} if $\Psi$ is itself a
root system (in the sense of \cite[Section 1.2]{humphreys90}). For each root
$\al\in \Phi$, the set $\Phi_\al:=\{\beta\in \Phi:B(\al,\beta)=0\}$ is
automatically a subsystem.

\subsection{Positive and simple systems}

\label{sec:basics2}
A subset $\Delta$ of a root system $\allroots$ is called a \emph{simple system}
if $\Delta$ is a vector space basis for $V$ and every root is a linear
combination of $\Delta$ with coefficients of like sign.  Given such a system
$\Delta$, we say a root $\al = \sum_{i \in \Gamma} c_i \al_i$ is \emph{positive}
(with respect to $\Delta$) if $c_i\ge 0$ for all $i$, and we call $\al$
\emph{negative} if $c_i\le 0$ for all $i$. The sets of positive and negative
roots are denoted by $\Phi_\Delta^+$ and $\Phi_\Delta^-$, and they are setwise
negations of each other.  For each root $\al$, the integer $\rootht(\al) =
\sum_{i \in \Gamma} c_i$ is called the {\it height} of $\al$. The set
$\{\al_i:c_i\neq 0\}$ is called the \emph{support} of $\al$ (with respect to
$\Delta$).

Each simple system $\Delta$ also gives rise to a partial order $\le_\Delta$ on
$\Phi$, which is defined by the condition that $\al \leq \be$ if and only if
$\be - \al$ is a nonnegative linear combination of $\Delta$.  With respect to
this partial order, $\allroots$ has a unique maximal element, $\theta_\Delta$,
called the {\it highest root}.

The set of simple roots $\Pi$ is an example of a simple system, and the
corresponding set of positive roots is an example of a \emph{positive system}.
Recall that each positive system $P$ contains a unique simple system, which is
the set $\Delta_P$ of all elements in $P$ that cannot be expressed as positive
linear combinations of other elements of $P$. The maps $\Delta\mapsto
\Phi_\Delta^+$ and $P\mapsto \Delta_P$ are mutually inverse bijections between
the sets of simple systems and positive systems in $\Phi$ \cite[Theorem
1.3]{humphreys90}. The simple systems of $\Phi$ are conjugate to each other
under the action of $W$, as are the positive systems.

From now on, we choose $\Pi$ as the default simple system of $\Phi$ and choose
$\Phi^+_\Pi$ as the default positive system. For each notion defined above
relative to a general simple system $\Delta$, an omission of the subscript in
the corresponding notation will indicate that $\Pi$ is chosen as $\Delta$.  For
example, the set of positive roots with respect to $\Pi$ will be denoted by
$\Phi^+$.

For any subsystem $\Psi$ of $\Phi$, the set $\Psi^+:=\Psi\cap\Phi^+$ is
automatically a positive system of $\Psi$. We call $\Psi^+$ the \emph{induced
positive system} of $\Psi$ with respect to $\Phi^+$. We call the corresponding
simple system of $\Psi$ the \emph{induced simple system} of $\Psi$. 

If $\Psi$ is a subsystem of the form $\Phi_\al$, i.e., if $\Psi$ is the
subsystem of roots orthogonal to a root $\al$, then we denote the induced simple
system by $\Pi_\al$. The elements of $\Pi_\al$ are thus the positive roots
orthogonal to $\al$ that cannot be expressed as a nonnegative linear combination
of other positive roots orthogonal to $\al$. Note that the elements of $\Pi_\al$
may not all lie in $\Pi$, but every simple root $\al_i\in \Pi$ that lies in $\Psi$ and is
orthogonal to $\al$ will lie in $\Pi_\al$. We denote the Weyl group corresponding to
$\Pi_\al$ by $W_\al$, so that $\Phi_\al$ is the root system of $W_\al$. It is
known that $W_\al$ is the stabilizer of $\al$ under the action of $W$, and
$W_\al$ is a direct product of irreducible simply laced Weyl groups
\cite{brink96,allcock13}. We will recall the known Dynkin type of the groups
$W_\al$ in the next subsection.

\begin{exa}\label{exa:d5}
Let $W$ be the Weyl group of type $D_5$ and let $\al_2$ be the simple root of
$W$ under the labelling used in Figure \ref{fig:ade}. The induced simple system
of $\Phi_{\al_2}$ is given by the disjoint union $$ \Pi_{\al_2}=\{\be_1=\al_2 + 2\al_3 +
\al_4 + \al_5\} \ \sqcup \ \{\be_2=\al_1 + \al_2 + \al_3, \be_3=\al_4,
\be_4=\al_5\},$$ 
where each root in one part of the union is orthogonal to every root in the other part. 
The Weyl group $W_{\al_2}$ corresponding to $\Phi_{\al_2}$ is the direct product
$W(A_1)\times W(A_3)$ of the Weyl groups of types $A_1$ and $A_3$, generated
respectively by the sets $\{s_{\be_1}\}$ and $\{s_{\be_2}, s_{\be_3},
s_{\be_4}\}$.
\end{exa}

\subsection{Explicit constructions}
\label{sec:constructions}

We now recall well-known explicit realizations of the root systems of types
$A,D$ and $E$ in coordinate systems. Let $\ep_1, \ep_2, \ldots, \ep_n$ be the
usual standard basis of the Euclidean space $\R^n$.  The vectors $\{\ep_i - \ep_j : 1 \leq i \ne j
\leq n\}$ form a root system of type $A_{n-1}$.  The simple roots $\Pi=\{\al_1,
\al_2, \ldots, \al_{n-1}\}$ are given by $\al_i = \ep_i - \ep_{i+1}$.  A root
$\ep_i - \ep_j$ is positive if $i < j$ and negative if $i > j$.  The highest
root (with respect to $\Pi$) is $\ep_1 - \ep_n$.  The bilinear form $B$ is the
Euclidean inner product on $\R^n$, and two roots are orthogonal if and only if they have
disjoint support. The Weyl group is isomorphic to $S_n$ and acts by permuting
the standard basis $\ep_1,..,\ep_n$.  The stabilizer $W_\al$ of each root
$\al$ is a Weyl group of type $A_{n-3}$, which is trivial if $n \leq 3$.

The vectors $\{\pm\ep_i \pm \ep_j : 1 \leq i < j \leq n\}$ form a root system of
type $D_n$.  The simple roots $\{\al_1, \al_2, \ldots, \al_n\}$ are given by
$\al_i = \ep_i - \ep_{i+1}$ for $i < n$, and $\al_n = \ep_{n-1} + \ep_{n}$. If $i
< j$, then the roots $\ep_i \pm \ep_j$ are positive, and the roots $-\ep_i \pm
\ep_j$ are negative. The highest root is $\ep_1 + \ep_2$.  The bilinear form $B$
is the Euclidean inner product, and two roots $\al$ and $\be$ are orthogonal if and
only if either (a) $\al$ and $\be$ have disjoint support or (b) $\al$ and $\be$
have the same support and $\al \ne \pm \be$.  The Weyl group has order
$2^{n-1}n!$, and acts by signed permutations of the standard basis, with the
restriction that each element effects an even number of sign changes
\cite[Section 2.10]{humphreys90}.  The stabilizer of a root is a Weyl group of
type $A_1 + D_{n-2}$, meaning $W(A_1) \times W(D_{n-2})$, where we interpret
$D_3$ as $A_3$ and $D_2$ as $A_1 + A_1$.  There is a well-known homomorphism
$\phi$ from $W(D_{2n})$ to the symmetric group $S_{2n}$ resulting from
forgetting the signs in a signed permutation; it is given by the following
formula:
\begin{equation}
    \label{eq:phimap} \phi(s_i) = \begin{cases} (i, i+1) & \text{if\ } i < 2n;\\
    (2n-1, 2n) & \text{if } i=2n.
      \end{cases}
  \end{equation}

Let $\ep_0, \ep_1, \ldots, \ep_7$ be the standard basis of $\R^8$. The root
system of type $E_7$ may be regarded as a subset of $\R^8$ as follows.  There
are 56 roots of the form $\pm 2(\ep_i - \ep_j)$ where $0 \leq i \ne j \leq 7$,
and there are 70 roots of the form $\sum_{i = 0}^7 \pm \ep_i$, where the signs
are chosen so that there are four $+$ and four $-$. The simple roots are $\al_1,
\al_2, \ldots, \al_7$, where $\al_i = 2(\ep_i - \ep_{i+1})$ for $1 \leq i \leq
6$, and $$ \al_7 = - \ep_0 - \ep_1 - \ep_2 - \ep_3 + \ep_4 + \ep_5 + \ep_6 +
\ep_7 .$$ A root of the form $2(\ep_i - \ep_j)$ is positive if $0 < i < j$ or $j
= 0$, and negative otherwise. A root of the form $\sum_{i = 0}^7 \pm \ep_i$ is
positive if and only if $\ep_0$ occurs with negative coefficient. The highest
root is $2(\ep_1 - \ep_0)$.  The bilinear form $B$ is $1/4$ of the Euclidean
inner 
product.  The stabilizer of a root is a Weyl group of type $D_6$. We call the
coordinates $\ep_i$ {\it Fano coordinates} because they are particularly
compatible with the combinatorics of the Fano plane; this will be important in
Section \ref{sec:e7}.

Let $\ep_1, \ep_2, \ldots, \ep_8$ be the standard basis of $\R^8$. The root
system of type $E_8$ may be regarded as a subset of $\R^8$ as follows.  There
are 112 roots of the form $\pm 2(\ep_i \pm \ep_j)$ where $1 \leq i \ne j \leq
8$, and there are 128 roots of the form $\sum_{i = 1}^8 \pm \ep_i$, where the
signs are chosen so that the total number of $-$ is even. The simple roots are
$\al_1, \al_2, \ldots, \al_8$, where $$ \al_1 = \ep_1 - \ep_2 - \ep_3 - \ep_4 -
\ep_5 - \ep_6 - \ep_7 + \ep_8 ,$$ $\al_2 = \ep_1 + \ep_2$, and $\al_i =
\ep_{i-1} - \ep_{i-2}$ for all $3 \leq i \leq 8$. If $k$ is the largest integer
such that $\ep_k$ appears in $\al$ with nonzero coefficient $c$, then $\al$ is
positive if and only if $c > 0$. The highest root is $2(\ep_7 + \ep_8)$.  The
bilinear form $B$ is $1/4$ of the Euclidean inner product.  The stabilizer of a root
is a Weyl group of type $E_7$. We call the coordinates $\ep_i$ the {\it standard
coordinates} for $E_8$.

\section{Combinatorics of coplanar quadruples}\label{sec:crossnest}

A {\it matching} of $[2n] := \{1, 2, \ldots, 2n\}$ is a collection of pairwise
disjoint size-2 subsets, or \emph{2-blocks}, of $[2n]$. The matching is {\it
perfect} if the union of the 2-blocks is the whole of $[2n]$.  If $1 \leq a < b
< c < d \leq 2n$, then a {\it crossing} is a subset of the matching of the form
$\{\{a, c\}, \ \{b, d\}\}$, a {\it nesting} is a subset of the form $\{\{a, d\},
\ \{b, c\}\}$, and an {\it alignment} is a subset of the form $\{\{a, b\}, \
\{c, d\}\}$.  For convenience, we will often denote each 2-block  $\{a,b\}$ in a
matching simply by $ab$ from now on.

In this section, we generalize crossings, nestings, and alignments to the notion
of \coplanar quadruples in the 
context of orthogonal sets of roots (Definition \ref{def:nestcross}, Remark
\ref{rmk:cna}). 
As explained in the introduction, we can naturally identify each $n$-root $\al=\prod_{i=1}^n\beta_i$ in the
Macdonald representation $\macd{\Phi}{nA_1}$ (Definition \ref{def:mac}) with
the set of its orthogonal components, and
it turns out that \coplanar
quadruples are very useful for understanding the action of $W$ on the 
orthogonal sets arising this way. 
We develop some key properties of  \coplanar quadruples in Theorem
\ref{thm:cna}, which is the main result of Section \ref{sec:crossnest}.  We also
show that crossings, nestings, and alignments can be distinguished from each
other based on the heights of the roots that they contain (Proposition
\ref{prop:htseq}), and we give a precise description of the ways in which two
coplanar quadruples can overlap (Proposition \ref{prop:overlap}).

\subsection{Coplanar quadruples}
We gather a few facts about $n$-roots and define coplanar quadruples in this
subsection.  The following two results on maximal orthogonal sets of roots are
well known, but we include proofs for ease of reference.

\begin{lemma}\label{lem:dmax}
    Let $W$ be a Weyl group of type $D_n$ for $n$ even. Suppose that $n=2k\ge
    4$.
    \begin{enumerate}
        \item[{\rm (i)}]{Every maximal orthogonal set of roots is of the form
     \begin{align*} \{& \pm(\ep_{i_1} + \ep_{j_1}),\
             \pm(\ep_{i_1} - \ep_{j_1}),\   \pm(\ep_{i_2} +
             \ep_{j_2}),\  \pm(\ep_{i_2} - \ep_{j_2}),\
             \ldots, & \pm(\ep_{i_k} + \ep_{j_k}),\
 \pm(\ep_{i_k} - \ep_{j_k}) \}\end{align*} where we have 
 $\{i_1,j_1,\dots,i_k,j_k\}=\{1,2,\dots, 2k-1,2k\}$ as sets and
  the signs are chosen independently.}
\item[{\rm (ii)}]{Every maximal orthogonal set of positive roots is of the form
     $$ \{ \ep_{i_1} + \ep_{j_1}, \ \ep_{i_1} - \ep_{j_1}, \ \ep_{i_2} +
     \ep_{j_2}, \ \ep_{i_2} - \ep_{j_2}, \ \ldots, \ep_{i_k} + \ep_{j_k}, \
 \ep_{i_k} - \ep_{j_k} \} ,$$ and these sets are in bijection with perfect
 matchings $\{ \{i_1, j_1\}, \ \{i_2, j_2\},\ \ldots, \{i_k, j_k\}\}$ of the set
 $[n]$ that satisfy $i_r < j_r$ for all $1 \leq r \leq k$.}
\end{enumerate}
\end{lemma}

\begin{proof}
Let $R$ be a maximal orthogonal set of roots. By symmetry, we may reduce to the
case where all the roots in $R$ are positive.  If $R$ contains the root $\ep_i
\pm \ep_j$, then $R$ must also contain the root $\ep_i \mp \ep_j$, because
otherwise $R \cup \{\ep_i \mp \ep_j\}$ would be a set of orthogonal roots that
was larger than $R$. It follows that $R$ consists of $n/2$ pairs of roots such
that each pair has the same support, and roots from distinct pairs have disjoint
supports. This completes the proof of (i).

Part (ii) follows from (i) and the fact that if $1 \leq i < j \leq n$, then the
roots $\ep_i \pm \ep_j$ are positive.
\end{proof}

It follows from Lemma \ref{lem:dmax} (ii) that a maximal orthogonal set of
positive roots in type $D_n$ (for $n$ even) contains the root $\ep_i - \ep_j$ if
and only if it contains the root $\ep_i + \ep_j$. We will call such a pair of
roots $\{\ep_i \pm \ep_j\}$ a {\it collinear pair} of roots.

\begin{lemma}\label{lem:transitive}
If $W$ is a Weyl group of type $E_7$, $E_8$, or $D_n$ with $n \geq 4$ even, then
$W$ acts transitively on the set $\mathcal{M}(W)$ of unordered maximal sets of
orthogonal roots of $W$.
\end{lemma}

\begin{proof}
Recall from Section \ref{sec:roots} that the group $W(D_n)$ can be regarded as
the group of signed permutations of $n$ objects in which there is an even number
of sign changes. Such a group acts transitively on the set described in Lemma
\ref{lem:dmax} (i).

Now suppose $W$ has type $E_7$, and let $\al$ be a root of $W$. Then by Section
2, the stabilizer $W_\al$ is a Weyl group of type $D_6$ whose root system is the
set $\Phi_\alpha$ of roots that are orthogonal to $\alpha$.  Since $W_\alpha$
acts transitively on $\Phi_\alpha$, it follows that there is a well-defined
bijection $[R]\mapsto [R\cup\{\al\}]$ from the set of $W_\alpha$-orbits on
$\mathcal{M}(W_\alpha)$ to the set of $W$-orbits on $\mathcal{M}(W)$, where the
orbit $[R]$ of every 6-tuple $R\in \mathcal{M}(W_\alpha)$ is sent to the orbit
$[R\cup\{\al\}]$ of the 7-tuple $R\cup\{\al\}$. It then follows that $W$ acts
transitively on $\mathcal{M}(W)$, as desired.

Finally, if $W$ has type $E_8$, then for each root $\alpha$ of $W$ the stablizer
$W_\alpha$ is of type $E_7$. A similar argument to the one above shows that $W$
acts transitively on $\mathcal{M}(W)$ because $W_\alpha$ acts transitively on
$\mathcal{M}(W_\alpha)$.
\end{proof}

We are ready to define coplanar quadruples. The following proposition offers
multiple equivalent characterizations of them.

\begin{prop}\label{prop:coll4}
Let $Q = \{\be_1, \be_2, \be_3, \be_4\}$ be a set of four mutually orthogonal
roots for a simply laced Weyl group $W$ with root system $\allroots$, and let
$\gamma = (\be_1 + \be_2 + \be_3 + \be_4)/2$. The following are equivalent:
\begin{enumerate}
    \item[{\rm (i)}]{$\gamma$ is a root (i.e., the elements of $Q$ sum to twice
     a root);}
 \item[{\rm (ii)}]{$Q$ is contained in a subsystem $\Psi$ of type $D_4$;}
 \item[{\rm (iii)}]{there is a unique subsystem $\Psi$ of type $D_4$ such that
     $(Q\cup\{\gamma\}) \subset \Psi \subseteq \allroots$, and we have $$ \Psi =
     \left\{\pm \be_1, \ \pm \be_2, \ \pm \be_3, \ \pm \be_4, \ (\pm \be_1 \pm
     \be_2 \pm \be_3 \pm \be_4)/2\right\} ,$$ where all the signs are chosen
 independently.}
\end{enumerate}
\end{prop}

\begin{proof}
We first prove that (i) implies (iii).  Assume that $\gamma$ is a root. Any root
subsystem containing $Q\cup \{\gamma\}$ also contains $s_{\be_i}(\be_i) =
-\be_i$ for each $i$, as well as all roots of the form $$ s_{\be_1}^{\epsilon_1}
s_{\be_2}^{\epsilon_2} s_{\be_3}^{\epsilon_3} s_{\be_4}^{\epsilon_4}(\gamma) ,$$
where we have $\epsilon_i \in \{0, 1\}$ for all $i$. The 16 roots listed above
can also be expressed as $$ (\pm \be_1 \pm \be_2 \pm \be_3 \pm \be_4)/2 .$$ We
have constructed all 24 roots in the set $\Psi$ listed in the statement, and
this is the cardinality of a root system of type $D_4$. To prove (iii) it now
suffices to show that $\Psi$ is a root system of type $D_4$. Because the
elements of $Q$ are orthogonal vectors of the same length, we may choose
Euclidean coordinates $\be_1 = \ep_1 - \ep_2$, $\be_2 = \ep_1 + \ep_2$, $\be_3 =
\ep_3 - \ep_4$, and $\be_4 = \ep_3 + \ep_4$. With respect to these coordinates,
we have $$ \Psi = \{\pm \ep_i \pm \ep_j : 1 \leq i < j \leq 4\} ,$$ which indeed
forms a root system of type $D_4$, as desired.

It is immediate that (iii) implies (ii).

In the usual notation for the simple roots of type $D_4$, the orthogonal roots
$\al_1$, $\al_3$, $\al_4$, and $\al_1 + 2\al_2 + \al_3 + \al_4$ sum to $2 \al$,
where $\al$ is the root $\al_1 + \al_2 + \al_3 + \al_4$. Lemma
\ref{lem:transitive} applied to a root system of type $D_4$ then implies that
the sum of every orthogonal quadruple of roots in a root system of type $D_4$ is
equal to $2 \al'$ for some root $\al'$. It follows that (ii) implies (i), which
completes the proof.
\end{proof}

\begin{defn}\label{def:coll4}
A set $Q$ of four mutually orthogonal roots for a simply laced Weyl group is
called a {\it \coplanar quadruple} if it satisfies the equivalent conditions of
Proposition \ref{prop:coll4}. In this case, we call the set $\Psi$ from
Proposition \ref{prop:coll4} \emph{the $D_4$-subsystem associated to $Q$}.
\end{defn}

Coplanar quadruples can be described explicitly in coordinates in type $D$:
\begin{lemma}
    \label{lem:d4coplanar} Let $W$ be a Weyl group of type $D_n$ for $n$ even
      and $n\ge 4$. Then four positive roots of $W$ form a coplanar quadruple if
      and only if they consist of two collinear pairs of roots, i.e., if and
      only if they are of the form $\ep_i+\ep_j, \ep_i-\ep_j,
      \ep_k+\ep_l,\ep_k-\ep_l$ for four distinct indices $i,j,k,l$ where $i<j$
      and $k<l$.
  \end{lemma}

\begin{rmk}
\label{rmk:d4matching}
In the setting of Lemma \ref{lem:d4coplanar}, we may naturally identify the
coplanar quadruple $\{\ep_i\pm\ep_j,\ep_k\pm\ep_l\}$ with the matching
$\{ij,kl\}$ of the set $\{i,j,k,l\}$.
\end{rmk}

\begin{rmk}
\label{rmk:Daction}
Recall that reflections in $W(D_n)$ act on the reflection representation as
signed permutations, with $s_{\al}(\ep_i)=\ep_j$ if $\al=\ep_i-\ep_j$ and
$s_\al(\ep_i)=-\ep_j$ if $\al=\ep_i+\ep_j$.  It follows that $W(D_n)$ acts on
the set $\{\ep_1^2,\cdots,\ep_n^2\}$ as ordinary permutations, with
$s_\al(\ep_i^2)=\ep_j^2$ for all distinct $i, j$.  When $n$ is even, it then
follows from Lemmas \ref{lem:dmax} and \ref{lem:d4coplanar} that the action of
$W(D_n)$ on $n$-roots factors through the map $\phi:W(D_n)\ra S_n$ from Equation
\eqref{eq:phimap} to induce an action of $S_n$ on $n$-roots. In particular, each
reflection $r\in W(D_n)$ acts in the same way as $\phi(r)$ on every $n$-root of
$W(D_n)$.
\end{rmk}

\begin{proof}[Proof of Lemma \ref{lem:d4coplanar}]
The ``if'' implication holds since the four roots in the given form sum to twice
the positive root $\ep_i+\ep_k$. To prove the ``only if'' implication, let
$Q=\{\be_1,\be_2,\be_3,\be_4\}$ be a coplanar quadruple. Recall from Section 2
that two roots in type $D_n$ are orthogonal if and only if they have the same or
disjoint support. It follows that if no two roots in $Q$ have the same support,
then the supports of $\be_1,\be_2,\be_3,\be_4$ contain a total of eight distinct
coordinates $\ep_i$, in which case the sum $\gamma=\be_1+\be_2+\be_3+\be_4$
cannot be twice a root. We may therefore assume, without loss of generality,
that $\{\be_1,\be_2\}=\{\ep_i\pm\ep_j\}$ for some $i<j$. This implies
$\beta_1+\beta_2=2\ep_i$. The condition that $Q$ is an orthogonal set summing to
twice a root then forces us to have $\{\be_3,\be_4\}=\{\ep_k\pm\ep_l\}$ for some
elements $k,l$ distinct from $i$ and $j$ with $k<l$.
\end{proof}

The next proposition shows that the action of $W$ on $n$-roots is local to
coplanar quadruples in the following sense: whenever a reflection in $W$ does
not fix a maximal orthogonal set $R$ of roots, it must change exactly four
elements of $R$ that form a \coplanar quadruple, and it changes these four
elements to another coplanar quadruple with the same associated $D_4$-subsystem.

\begin{prop}\label{prop:four}
Let $W$ be a Weyl group of type $E_7$, $E_8$, or $D_n$ for $n$ even, let $\al$
be a root, and let $R$ be a maximal set of orthogonal positive roots. Suppose
that neither $\al$ nor $-\al$ is an element of $R$.
\begin{enumerate}
    \item[{\rm (i)}] The root $\al$ is orthogonal to all but precisely four
     elements $Q = \{\be_1, \be_2, \be_3, \be_4\}$ of $R$.  The elements of $Q$
     form a \coplanar quadruple, and we have $ 2\al = \pm \be_1 \pm \be_2 \pm
     \be_3 \pm \be_4$ for suitable choices of signs.
 \item[{\rm (ii)}] Let $\Psi$ be the $D_4$ subsystem associated to $Q$. Then we
     have $\alpha\in \Psi$, and the set $s_\al(Q)=\{s_\al(\be_i):1\le i\le 4\}$
     is also a \coplanar quadruple whose associated $D_4$-subsystem is $\Psi$.
\end{enumerate}
\end{prop}

\begin{proof}
To prove (i), it suffices by Lemma \ref{lem:transitive} to do so for a fixed
$R$.  Suppose first that $W$ is of type $D_n$, and choose $$ R = \{ \ep_1 +
    \ep_2, \ \ep_1 - \ep_2,\ \ep_3 + \ep_4, \ \ep_3 - \ep_4, \ \ldots, \
    \ep_{n-1} + \ep_n, \ \ep_{n-1} - \ep_n \}.$$ The root $\al$ must be of the
    form $\pm \ep_i \pm \ep_j$, where $i$ and $j$ come from different parts of
    the partition $\{\{1, 2\}, \{3, 4\}, \ldots, \{n-1, n\}\}$. It follows that
    the support of $\al$ has one element in common with the support of each of
    precisely four elements of $R$ making up two collinear pairs, and that $\al$
    is orthogonal to all the other elements of $R$.  Furthermore, the roots
    $\be_1$, $\be_2$, $\be_3$, and $\be_4$ that are not orthogonal to $\al$ are
    of the form $\pm \ep_h \pm \ep_i$ and $\pm \ep_j \pm \ep_k$, where $|\{h, i,
    j, k\}| = 4$. It follows that $2\al$ can be written in the form $\pm
    \beta_1\pm\be_2\pm\beta_3\pm\beta_4$ for suitable choices of signs.

Next, suppose that $W$ has type $E_8$, and choose $$ R = \{ 2(\ep_1 + \ep_2), \
2(\ep_1 - \ep_2), \ 2(\ep_3 + \ep_4), \ 2(\ep_3 - \ep_4), \ 2(\ep_5 + \ep_6), \
2(\ep_5 - \ep_6), \ 2(\ep_7 + \ep_8), \ 2(\ep_7 - \ep_8) \}.$$ If $\al$ has the
form $2(\pm \ep_i \pm \ep_j)$, then the proof is completed using the same
argument as in type $D_8$.  The other possibility is that we have $\al = \sum_{i
= 1}^8 \pm \ep_i$, where the signs are chosen so that there is an even number of
minus signs. In this case, $\al$ is orthogonal to precisely one of the roots
$\{2(\ep_j - \ep_{j+1}), \ 2(\ep_j + \ep_{j+1})\}$, according as $\ep_j$ and
$\ep_{j+1}$ occur in $\al$ with the same or with opposite coefficients. It
follows that $\al$ is orthogonal to precisely four elements of $R$, and that $2
\al$ can be expressed in the required form.

Now suppose that $W$ has type $E_7$. By Section \ref{sec:constructions}, we may
identify the root system of $W$ with the set of roots orthogonal to the highest
root $\theta$ in the root system of type $E_8$, so that $R \cup \{\theta\}$ is a
maximal set of orthogonal roots in type $E_8$. By the above paragraph, $\al$ is orthogonal
to four of the roots in $R \cup \{\theta\}$, but one of these roots is $\theta$
itself. It follows that $\al$ is orthogonal to three elements of $R$, and that
$2 \al$ can be expressed as a signed sum of the other four elements of $R$, as
required. This completes the proof of (i).

It follows from (i) that we have $\alpha\in \Psi$, because condition (i) implies
condition (iii) in Proposition \ref{prop:coll4}. The element
$\gamma=(\be_1+\be_2+\be_3+\be_4)/2$ is a root because $Q$ is a coplanar
quadruple, and the set $s_\al(Q)$ is a \coplanar quadruple because its elements
sum to twice $s_\al(\gamma)$, which is a root because $\gamma$ is.  We also have
$s_\al(Q)\se \Psi$ because both $\al$ and $Q$ are in $\Psi$.  This implies that
$\Psi$ must be the $D_4$-subsystem associated to $s_\al(Q)$, proving (ii).
\end{proof}

\subsection{Crossings, nestings, and alignments}

We examine coplanar quadruples more closely in this subsection and classify them
into three types, namely, crossings, nestings, and alignments. As we will
explain in Remark \ref{rmk:cna}, our terminology comes from the theory of
matchings, but the following definition makes sense for all the simply laced
root systems considered in this paper.

\begin{defn}\label{def:nestcross}
Let $Q = \{\be_1, \be_2, \be_3, \be_4\}$ be a \coplanar quadruple of positive
orthogonal roots, let $\Psi$ be the $D_4$-subsystem associated to $Q$, let
$\leq$ be the partial order on $\Psi$ relative to the induced simple roots of
$\Psi$, and let $\gamma$ be the root $(\be_1 + \be_2 + \be_3 + \be_4)/2$. We say
that $Q$ is
\begin{enumerate}
    \item[{\rm (i)}]{a {\it crossing} if $\be_i \leq \gamma$ for all $i$ and $Q$
     contains the unique $\leq$-maximal element of $Q \cup (-s_\gamma(Q))$;}
 \item[{\rm (ii)}]{a {\it nesting} if $\be_i \leq \gamma$ for all $i$ and $Q$
     contains the unique $\leq$-minimal element of $Q \cup (-s_\gamma(Q))$;}
 \item[{\rm (iii)}]{an {\it alignment} otherwise.}
\end{enumerate} We also call each crossing, nesting, and alignment a
    \emph{feature} of \emph{type C}, \emph{type N}, and \emph{type A},
    respectively.
\end{defn}

\begin{theorem}\label{thm:cna}
Let $\allroots$ be a root system for a Weyl group of type $E_7$, $E_8$, or $D_n$
for $n$ even.  Let $Q$ be a \coplanar quadruple of positive roots of
$\allroots$, let $\Psi$ be the associated $D_4$-subsystem, and let $\Psi^+ =
\Psi \cap \posroots$ be the induced positive system of $\Psi$.
\begin{enumerate}
\item[{\rm (i)}]
The set $\Psi^+$ contains precisely three distinct quadruples of mutually
orthogonal roots. These quadruples are pairwise disjoint and partition $\Psi^+$.

\item[{\rm (ii)}]
The three quadruples of orthogonal roots in $\Psi^+$  are all coplanar. Among
them there is exactly one crossing, $\Psi^+_C$, exactly one nesting, $\Psi^+_N$,
and exactly one alignment, $\Psi^+_A$. In particular, the quadruple $Q$ cannot
be both a crossing and a nesting, and the three conditions in Definition
\ref{def:nestcross} are mutually exclusive.

\item[{\rm (iii)}]
Each quadruple in $\{\Psi^+_C,\Psi^+_N, \Psi^+_A\}$ uniquely determines both of
the other two.

\item[{\rm (iv)}]
{If $R$ is a set of mutually orthogonal roots that is disjoint from $\Psi$, then
either each of the three sets $\{R \cup \Psi^+_C, \ R \cup \Psi^+_N, \ R \cup
\Psi^+_A\}$ consists of mutually orthogonal roots, or none of them does.}

\item[{\rm (v)}]
The crossing $\Psi^+_C$ contains no root from the induced simple system of
$\Psi$.

\item[{\rm (vi)}]
For each $x\in \{C,N,A\}$, let $\gamma_x$ be the product of the roots in
$\Psi^+_x$, and let $\sigma(\gamma_x)$ be the sum of the components of $\gamma_x$. Then we have $\sigma(\gamma_A) < \sigma(\gamma_N) =
\sigma(\gamma_C)$ and $\gamma_C = \gamma_N + \gamma_A.$ Moreover, if $\al$ is
any component in one of the three $4$-roots $\gamma_C,\gamma_N$ and $\gamma_A$,
then the reflection $s_\al$ sends the other two $4$-roots to each other; for
example, if $\al\in \Psi^+_C$ then $s_\al$ sends $\gamma_N$ and $\gamma_A$ to
each other.
\end{enumerate}
\end{theorem}

\begin{proof}
By Section \ref{sec:constructions}, the roots orthogonal to a given root in type
$D_4$ form a subsystem of type $3A_1$. Therefore each positive root lies in a
unique quadruple of mutually orthogonal positive roots, which proves (i).

Let $\{\al_1,\al_2,\al_3,\al_4\}$ be the induced simple roots of $\Psi$, with
$\al_2$ corresponding to the branch node in the Dynkin diagram.  Then the three
quadruples from (i) are given by \begin{align*} \Psi^+_1 &= \{\al_1 + \al_2, \
    \al_2 + \al_3, \ \al_2 + \al_4, \ \al_1 + \al_2 + \al_3 + \al_4\}, \\
    \Psi^+_2 &= \{\al_2, \ \al_1 + \al_2 + \al_3, \ \al_1 + \al_2 + \al_4, \
    \al_2 + \al_3 + \al_4\}, \quad \text{\ and\ }\\ \Psi^+_3 &= \{\al_1, \
    \al_3, \ \al_4, \ \al_1 + 2\al_2 + \al_3 + \al_4\} .
\end{align*} The roots in $\Psi^+_3$ add up to twice the root $\al = \al_1 +
    \al_2 + \al_3 + \al_4$. The root $\al$ is strictly lower in the $\le$ order
    than one of the roots in $\Psi^+_3$; therefore $\Psi^+_3$ is an alignment by
    Definition \ref{def:nestcross}.  The roots in $\Psi^+_1$ and $\Psi^+_2$ both
    add up to $2 \theta$, where $\theta = \al_1 + 2\al_2 + \al_3 + \al_4$.  The
    root $\theta$ is strictly higher in the $\le$ order than each element of
    $\Psi^+_1$ and $\Psi^+_2$.  Note that we have $s_\theta(\Psi^+_1) =
    -\Psi^+_2$ and $s_\theta(\Psi^+_2) = -\Psi^+_1$.  Furthermore, the roots
    $\al_2$ and $\al_1 + \al_2 + \al_3 + \al_4$ are the unique $\leq$-minimal
    and unique $\leq$-maximal elements of $\Psi^+_1 \ \cup\ \Psi^+_2$,
    respectively. This implies that $\Psi^+_1$ is a crossing and that $\Psi^+_2$
    is a nesting, and it also follows that none of $\Psi^+_1, \Psi^+_2$ and
    $\Psi^+_3$ is both a crossing and a nesting. The quadruple $Q$ must be one
    of $\Psi^+_1, \Psi^+_2$ and $\Psi^+_3$, and (ii) follows.

Part (iii) follows from (ii), since each of  $\Psi^+_C, \Psi^+_N$ and $\Psi^+_A$
uniquely determines $\Psi^+$ as its associated $D_4$-subsystem by Proposition
\ref{prop:coll4} (iii).

Part (iv) holds as each of the quadruples in $\{\Psi^+_C, \Psi^+_N, \Psi^+_A\}$
is a basis for the span of $\Psi^+$, so that any root that is orthogonal to
every element of one quadruple is also orthogonal to every element of each of
the other quadruples.

Finally, the claims in (v) and (vi) can all be verified by inspection or direct
computation based on the enumerate of $\Psi^+_C=\Psi^+_1, \Psi^+_N=\Psi^+_2$
and $\Psi^+_A=\Psi^+_3$. For the equation $\gamma_C=\gamma_N+\gamma_A$ and the
assertion about $s_\al$ in (vi), one can alternatively prove them using the
usual realizations of the root system and Weyl group of type $D_4$, where the
simple roots are $\al_1 = \ep_1 - \ep_2$, $\al_2 = \ep_2 - \ep_3$, $\al_3 =
\ep_3 - \ep_4$ and $\al_4 = \ep_3 + \ep_4$ and the group $W$ acts as signed
permutations.  Under this realization, we have \begin{equation}
\label{eq:D4} \gamma_C = (\ep_1^2 - \ep_3^2)(\ep_2^2 - \ep_4^2),\;
      \gamma_N=(\ep_1^2 - \ep_4^2)(\ep_2^2 - \ep_3^2),\;\text{and}\;
  \gamma_A=(\ep_1^2 - \ep_2^2)(\ep_3^2 - \ep_4^2), \end{equation} and the
  equation $\gamma_C=\gamma_N+\gamma_A$ follows as the terms expressed in
  coordinates satisfy the Ptolemy relation \[ (A-C)(B-D) = (A-D)(B-C) +
  (A-B)(C-D).  \qedhere \]
\end{proof}

\begin{rmk}\label{rmk:cna}
In the setting of Theorem \ref{thm:cna}, the coordinate forms of the $4$-roots
$\gamma_C, \gamma_N, \gamma_A$  given in Equation (\ref{eq:D4}) correspond  via
the bijection of Lemma \ref{lem:dmax} (ii) to the perfect matchings of the set
$[4]$ given by the crossing $m_C=\{\{1,3\}, \{2,4\}\}$, the nesting
$m_N=\{\{1,4\},\{2,3\}\},$ and the alignment $m_A=\{\{1,2\}, \{3,4\}\}$,
respectively.  Definition \ref{def:nestcross} generalizes the notion of
crossings, nestings, and alignments in the sense that a coplanar quadruple in
the sense of the definition is a crossing, nesting, or alignment if and only if
the corresponding perfect matching of the set $[4]$ is a crossing, nesting, or
alignment, respectively, in the context of matchings.
\end{rmk}

\begin{rmk}\label{rmk:moves}
Let $R$ be a maximal orthogonal set of positive roots and let $\alpha$ be a
positive root not in $R$. Proposition \ref{prop:four} shows that the reflection
$s_\alpha$ changes precisely four elements in $R$ which form a \coplanar
quadruple $Q$, and that $Q'=s_\al(Q)$ is another \coplanar quadruple with the
same associated $D_4$-subsystem as $Q$. Theorem \ref{thm:cna} (vi) reveals more
about $\al,Q$ and $Q'$: it shows that $Q$ and $Q'$ are two distinct features
from the set $\{\Psi^+_C,\Psi^+_N,\Psi^+_A\}$ inside the $D_4$-subsystem $\Psi$
associated to $Q$, while $\al$ is in the remaining feature.  We will say that
$s_\alpha$ \emph{moves} $Q$ in this case; we will also say that $s_\alpha$ (or
$\al$) \emph{moves an $X$ to a $Y$} and call $s_\al$ an \emph{XY move}, where
$X$ and $Y$ are the distinct types of $Q$ and $Q'$, respectively. Note that
knowledge of $Q$ and $Y$ is enough to determine $Q'$ by Theorem \ref{thm:cna}
(iii), even if $\al$ is not known. Note also that Theorem \ref{thm:cna} (vi)
guarantees that $XY$ moves exist for any distinct elements $X$, $Y$ in
$\{C,N,A\}$, so any two coplanar quadruples sharing the same $D_4$-subsystem can
be connected, up to sign, by a reflection.
\end{rmk}

The next proposition shows how to distinguish crossings, nestings, and
alignments from each other using only the heights of their components. Recall
from Section \ref{sec:basics2} that in a root system with simple system $\{\alpha_1,\al_2,
\cdots,\al_n\}$, the height of each root $\al=\sum_{i=1}^n c_i\al_i$ is the
integer $\rootht(\al)=\sum_{i=1}^n c_i$.

\begin{prop}\label{prop:htseq}
Let $W$ be a Weyl group of type $E_7$, $E_8$, or $D_n$ for $n$ even. Let $Q =
\{\be_1, \be_2, \be_3, \be_4\}$ be a \coplanar quadruple of orthogonal positive
roots, ordered so that we have $h_1 \leq h_2 \leq h_3 \leq h_4$ where $h_i=\rootht(\be_i)$ for each $i$.
\begin{enumerate}
    \item[{\rm (i)}]{We have $h_1 + h_2 + h_3 \ne h_4$ and $h_2 + h_3 \ne h_1 +
     h_4$.}
 \item[{\rm (ii)}]{If $h_1 + h_2 + h_3 < h_4$, then $Q$ is an alignment.}
 \item[{\rm (iii)}]{If $h_1 + h_2 + h_3 > h_4$ and $h_2 + h_3 > h_1 + h_4$, then
     $Q$ is a nesting. In this case, we also have $h_1 < h_2$.}
 \item[{\rm (iv)}]{If $h_1 + h_2 + h_3 > h_4$ and $h_2 + h_3 < h_1 + h_4$, then
     $Q$ is a crossing. In this case, we also have $h_1 > 1$ and $h_3 < h_4$.}
\end{enumerate}
\end{prop}

\begin{proof}
Let $\Psi$ be the type $D_4$ subsystem associated to $Q$, and let $\{\al_1,
\al_2, \al_3, \al_4\}$ be the simple system of $\Psi$ induced by $\posroots$,
with $\al_2$ corresponding to the branch node in the Dynkin diagram.  Then as in
the proof of Theorem \ref{thm:cna}, the set $\Psi^+$ decomposes into the union
of the crossing $\Psi^+_C$, nesting $\Psi^+_N$, and alignment $\Psi^+_A$ given
below:
\begin{align*} \Psi^+_C &= \{\al_1 + \al_2, \ \al_2 + \al_3, \ \al_2 + \al_4, \
      \al_1 + \al_2 + \al_3 + \al_4\}, \\ \Psi^+_N &= \{\al_2, \ \al_1 + \al_2 +
      \al_3, \ \al_1 + \al_2 + \al_4, \ \al_2 + \al_3 + \al_4\}, \quad \text{\
      and\ }\\ \Psi^+_A &= \{\al_1, \ \al_3, \ \al_4, \ \al_1 + 2\al_2 + \al_3 +
      \al_4\} .\end{align*}

If $Q=\Psi^+_A$, then we have $\be_4 - \be_1 - \be_2 - \be_3 = 2\al_2$, and
$\be_4 + \be_1 - \be_2 - \be_3 = 2\al_2 + 2\al_i$ for some $i \in \{1, 3, 4\}$.

If $Q=\Psi^+_N$, then we have $\be_1 < \be_j$ for all $j \in \{2, 3, 4\}$. We
also have $\be_4 - \be_1 - \be_2 - \be_3 = -2\al_i - 2\al_2$ for some $i \in
\{1, 3, 4\}$, and $\be_4 + \be_1 - \be_2 - \be_3 = -2\al_i$ for some $i \in \{1,
3, 4\}$.

If $Q=\Psi^+_C$, then we have $\be_4 > \be_j$ for all $j \in \{1, 2, 3\}$. We
also have $\be_4 - \be_1 - \be_2 - \be_3 = -2\al_2$ and $\be_4 + \be_1 - \be_2 -
\be_3 = 2\al_i$ for some $i \in \{1, 3, 4\}$.  Furthermore, none of the $\be_i$
is a simple root of $\Phi$, because none of the $\be_i$ is even a simple root of
$\Psi$.

All the claims in the proposition follow from the above observations. Note that
the conditions in (ii), (iii), and (iv) are exclusive and exhaustive because of
part (i).
\end{proof}

\begin{rmk}\label{rmk:squaresum}
    It can be shown that if $W$ is a Weyl group of type $E_7, E_8,$ or $D_n$ for
    $n$ even and the components of a positive $n$-root $\beta$ have heights
    $h_1, h_2, \ldots, h_n$, then the number $\sum_{i = 1}^n h_i^2$ depends only on $W$
  and is independent of the choice of $\be$.
\end{rmk}

\subsection{Intersections of coplanar quadruples}
Let $R$ be a maximal orthogonal set of roots of $W$.  In this subsection, we
first focus on type $E_8$ and show that coplanar quadruples in $R$ gives rise to
a Steiner quadruple system in this type. We will then use this result to count
coplanar quadruples in $R$ and deduce how coplanar quadruples in $R$ can overlap
with each other, in the general case.

\begin{defn}
\label{def:steiner}
A {\it Steiner system} $S(t,k,N)$ is a collection $B$ of $k$-element subsets of
the set $[N]=\{1,2,\dots, N\}$ with the property that every $t$-element subset
is contained in a unique element of $B$. The elements of $B$ are called
\emph{blocks}, and we write each block $\{a,b,c,\dots\}$ where $a<b<c<\dots$ as
$abc\dots$. We call $B$ a \emph{Steiner triple system} if $k=3$ and a
\emph{Steiner quadruple system} if $k=4$.
\end{defn}

\begin{rmk}
    \label{rmk:steiner_unique}
    A Steiner system $S(t,k,N)$ is also known as a $t$-$(N,k,1)$ \emph{design},
    which is a special kind of $t$-designs \cite[Section
    4.1]{design07}.
    It is well known that, up to isomorphism (by permutations), there is a
    unique Steiner triple system $S(2,3,7)$ and a unique Steiner quadruple
    system $S(3,4,8)$ \cite{barrau1908}. The following 14 quadruples form an
    example of a Steiner quadruple system, and removing the element 8 from all
    the quadruples on the left results in a Steiner triple system.  $$
    \begin{tabular}{c  c} 1  2  4  8  \quad& \quad  3  5  6  7 \\ 2 3  5  8
          \quad &\quad  1  4  6  7 \\ 3  4  6  8  \quad& \quad  1 2  5  7 \\ 4 5
          7 8  \quad& \quad  1  2  3  6 \\ 1  5  6 8  \quad& \quad  2  3  4 7 \\
          2  6 7  8  \quad& \quad  1  3  4 5 \\ 1  3  7  8  \quad& \quad  2 4  5
          6 \\
      \end{tabular} $$
  \end{rmk}

\begin{lemma}\label{lem:e8coplanar}
Let $\be_1, \be_2, \be_3$ be three mutually orthogonal positive roots of type
$E_8$.
\begin{enumerate}
    \item[{\rm (i)}]{There exists a unique positive root $\be_4$ such that
     $\{\be_1, \be_2, \be_3, \be_4\}$ is a \coplanar quadruple.}
 \item[{\rm (ii)}]{If $R$ is a maximal orthogonal set of positive roots of type
     $E_8$, and $\{\be_1, \be_2, \be_3\} \subset R$, then we have $\be_4 \in
 R$.}
\end{enumerate}
\end{lemma}

\begin{proof}
By \cite[Lemma 11 (iii)]{carter72} and its proof, if $W = W(E_8)$, then $W$ acts
transitively on ordered triples of orthogonal roots, and the set of roots
orthogonal to three given mutually orthogonal roots is a root system of type
$A_1 + D_4$. Since the action of orthogonal triples of roots is transitive, it
suffices to prove (i) for a fixed choice of $\be_1$, $\be_2$, and $\be_3$. If we
choose $\be_1 = \al_3 = 2(\ep_2 - \ep_1)$, $\be_2 = \al_5 = 2(\ep_4 - \ep_3)$,
and $\be_3 = \al_7 = 2(\ep_6 - \ep_5)$, then it follows from the explicit
description of the root system in Section \ref{sec:constructions} that the only
positive root that forms a \coplanar quadruple with $\be_1,\be_2$ and $\be_3$ is
the root $\be_4 = 2(\ep_8 - \ep_7)$. This proves (i).

The uniqueness property of (i) proves that the set of \coplanar quadruples
corresponds to the $A_1$ summand of the $A_1 + D_4$ subsystem. This implies that
if $\be$ is any positive root that is orthogonal to all of $\be_1$, $\be_2$, and
$\be_3$, then either $\be_4 = \be$, or $\be_4$ is orthogonal to $\be$. The
maximality of $R$ in the statement of (ii) implies that $\be_4$ cannot be
orthogonal to all elements of $R$. It follows that $\be_4 \in R$, which proves
(ii).
\end{proof}

\begin{lemma}\label{lem:steiner}
Let $W$ be a Weyl group of type $E_8$, and let $R$ be a maximal orthogonal set
of roots.
\begin{enumerate}
    \item[{\rm (i)}]{The \coplanar quadruples of $R$ endow $R$ with the
     structure of a Steiner quadruple system $S(3,4,8)$.}
 \item[{\rm (ii)}]{Any two \coplanar quadruples of $R$ intersect in either $0$,
     $2$, or $4$ elements.}
\end{enumerate}
\end{lemma}

\begin{proof}
Part (i) is immediate from Lemma \ref{lem:e8coplanar}.  To prove part (ii), we
need to show that any two distinct quadruples from the Steiner quadruple system
are either disjoint or overlap in precisely two elements. This can be proved by
an exhaustive check, or by arguing as follows.

The quadruples in the left column of the table consist of the element 8 together
with three points forming a line in the Fano plane (see Section \ref{sec:e7}).
Any two such quadruples intersect in two elements: the element 8 and the unique
point on the intersection of the two lines of the Fano plane. The general case
follows by combining this observation with the fact that each quadruple in the
right column is the complement of the corresponding quadruple in the left
column.
\end{proof}

\begin{cor}\label{cor:mcount}
Let $W$ be a Weyl group of type $E_7$, $E_8$, or $D_n$ for $n$ even and $n=2k\ge
4$. Let $R$ be a maximal orthogonal set of positive roots. Then the number $M$
of  \coplanar quadruples contained in $R$  does not depend on $R$. We have $M =
\binom{k}{2}$ if $W$ has type $D_{2k}$, $M = 7$ if $W$ has type $E_7$, and $M =
14$ if $W$ has type $E_8$. \end{cor}

\begin{proof}
If $W$ has type $D_{2k}$, then $R$ determines a perfect matching of the set
$\{1, 2, \ldots, 2k\}$ with $k$ blocks by Lemma \ref{lem:dmax}, and the
\coplanar quadruples in $R$ correspond bijectively to pairs of these blocks by
Remark \ref{rmk:d4matching}. It follows that $M$ does not depend on $R$ and
equals $\binom{k}{2}$.

In type $E_8$, the result follows from Lemma \ref{lem:steiner} (i).

If $W$ has type $E_7$, then as in the proof of Proposition \ref{prop:four}, we
may again identify the root system of $W$ with the set  of roots orthogonal to
the  the highest root $\theta$ in the root system of type $E_8$. The set $R \cup
\{\theta\}$ is then a maximal orthogonal set of roots in type $E_8$. The
\coplanar quadruples of $R$ are in bijection with the quadruples of $S(3,4,8)$
that exclude a fixed element, and there are $7$ such quadruples, so we are done.
\end{proof}

\begin{prop}\label{prop:overlap}
    Let $W$ be a Weyl group of type $E_7$, $E_8$, or $D_n$ for $n$ even. Let
    $R$ be a maximal orthogonal subset of positive roots, and suppose $Q_1$ and $Q_2$
    are \coplanar quadruples of roots contained in $R$. 
    \begin{enumerate}
        \item[{\rm (i)}]{The intersection $Q_1 \cap Q_2$ has size 0, 2, or 4.}
        \item[{\rm (ii)}]{If $|Q_1 \cap Q_2| = 2$, then there is a root
     subsystem $\Psi \subseteq \allroots$ of type $D_6$ that contains both $Q_1$
     and $Q_2$. In this case, each of the sets $Q_1$, $Q_2$, $Q_1 \cap Q_2$, and
     $Q_1 \cup Q_2$ consists of collinear pairs of roots with respect to $\Psi$,
     and the symmetric difference $Q_1 \ \Delta\ Q_2$ is also a \coplanar
 quadruple.}
\end{enumerate}
\end{prop}

\begin{proof}
If $W$ has type $D_n$, then the assertions follow from Lemma
\ref{lem:d4coplanar}.

Suppose that $W$ has type $E_8$. In this case, part (i) follows from Lemma
\ref{lem:steiner} (ii).  If $|Q_1 \cap Q_2| = 2$, then $|Q_1 \cup Q_2| = 6$, and
there are precisely two elements $\al, \be \in R$ that are orthogonal to every
root in $Q_1 \cup Q_2$. Let $\Psi$ be the set of roots in $\Phi$ that are
orthogonal to $\al$ and $\be$. Then $\Psi$ forms a root system of type $D_6$ by
Section \ref{sec:constructions}, proving the first assertion of (ii), and the
sets $Q_1, Q_2, Q_1\cap Q_2$ and $Q_1\cup Q_2$ all lie in $\Psi$.  The other
assertions of (ii) now follow by applying the type $D_n$ case of the result to
$\Psi$.

Finally, suppose that $W$ has type $E_7$. We identify the root system with the
set of roots orthogonal to the  the highest root $\theta$ in the root system of
type $E_8$ as usual.  Then $R \cup \{\theta\}$ is a maximal orthogonal set of
roots in type $E_8$. The assertions in this case follow by applying the argument
of the previous paragraph with $\al = \theta$.
\end{proof}

\section{Quasiparabolic structure}\label{sec:quasi}
Let $X=X(W)$ be the set of all maximal orthogonal sets of positive roots of $W$,
and recall from the introduction that $W$ acts on $X$ naturally via the action
$w(\{\be_1,\cdots,\be_n\})=\{\abs{w(\be_1)}, \cdots, \abs{w(\be_n)}\}$.  In
this section, we recall the notion of a quasiparabolic set as defined
by Rains and Vazirani \cite{rains13}, and we use the concepts of crossings and
nestings to endow the $W$-set $X$ with a quasiparabolic structure.

\subsection{Quasiparabolic sets}
Quasiparabolic sets were introduced by Rains and Vazirani for a general Coxeter
system as follows.

\begin{defn}
\cite[Section 2, Section 5]{rains13}
\label{def:qpset}
Let $W$ be a Coxeter group with generating set $S$ and set of reflections $T$. A
{\it scaled $W$-set} is a pair $(\mathcal{X}, \lambda)$, where $\mathcal{X}$ is a $W$-set and
$\lambda : \mathcal{X} \rightarrow \Z$ is a function satisfying $|\lambda(sx) -
\lambda(x)| \leq 1$ for all $s \in S$.  An element $x \in \mathcal{X}$ is {\it
$W$-minimal}  if $\lambda(sx) \geq \lambda(x)$ and is {\it $W$-maximal} if
$\lambda(sx) \leq \lambda(x)$ for all $s \in S$.

A {\it quasiparabolic set} for $W$ is a scaled $W$-set $\mathcal{X}$ satisfying the
following two properties:
\begin{enumerate}
    \item[(QP1)] for any $r \in T$ and $x \in \mathcal{X}$, if $\lambda(rx) = \lambda(x)$,
     then $rx = x$;
 \item[(QP2)] for any $r \in T$, $x \in \mathcal{X}$, and $s \in S$, if $\lambda(rx) >
     \lambda(x)$ and $\lambda(srx) < \lambda(sx)$, then $rx = sx$.
\end{enumerate} For a quasiparabolic set $\mathcal{X}$, we define $\le_Q$ to be the weakest
partial order such that $x \leq_Q rx$ whenever $x \in \mathcal{X}$, $r \in T$, and
    $\lambda(x) \leq \lambda(rx)$.
\end{defn}

Rains and Vazirani call $\lambda(x)$ the {\it height} of $x$, and $\leq_Q$ the
{\it Bruhat order}, but we will refer to them as the {\it level} of $x$ and the
{\it quasiparabolic order} because of the potential for confusion in the context
of this paper. It follows from \cite[Proposition 5.16]{rains13} that $\lambda$
is a rank function with respect to the partial order $\leq_Q$, so that every
covering relation $x <_Q y$ satisfies $\lambda(y) = \lambda(x) + 1$.

We will show that the set $X=X(W)$ forms a quasiparabolic set for the Weyl group $W$ in type $E_7,
E_8$ or $D_{n}$ with $n$ even under a suitable level function defined in terms
of coplanar quadruples. We define the level function and some other useful
statistics below.
\begin{defn}\label{def:stats} Let $W$ be a Weyl group of type $E_7, E_8$ or
      $D_{n}$ for $n$ even.  Let $R$ be a set of mutually orthogonal roots of
      $W$, and let $\beta$ be a positive $n$-root of $W$.
      \begin{enumerate}
          \item[(i)] We define the {\it crossing number} $C(R)$, the
               \emph{nesting number} $N(R)$, and the \emph{alignment number}
               $A(R)$ of $R$ to be the numbers of crossings, nestings, and
               alignments contained in $R^+$, respectively.
           \item[(ii)] We define the \emph{type} of $R$ to be the monomial
      $A^{A(R)}C^{C(R)}N^{N(R)}$, and define the {\it level} $\lambda(R)$ of $R$
      to be $C(R) + 2N(R)$.
  \item[(iii)] If $R$ is a maximal orthogonal set of roots, then we say that $R$
      is \emph{noncrossing}, \emph{nonnesting}, and \emph{alignment-free} if
      $C(R)=0$, $N(R)=0$ and $A(R)=0$, respectively; we also call $R$ {\it
      maximally crossing}, {\it maximally nesting}, or {\it maximally aligned}
      if we have $N(R)=A(R)=0$, $C(R)=A(R)=0$, or $C(R)=N(R)=0$, respectively.
\end{enumerate} We also apply all the above definitions to $\beta$ by applying them to
    the set of components of $\beta$. (Note that the
    definitions of type in (ii) and in Definition \ref{def:nestcross} are
consistent.)
\end{defn}

\begin{rmk}
\label{rmk:sum}
By Corollary \ref{cor:mcount}, when $R$ is a maximal orthogonal set of positive
roots, the sum of the numbers $C(R), N(R)$ and $A(R)$ is a constant depending
only on $W$ and not on $R$; therefore each of these numbers achieves the maximal
possible value when the other two equal zero. This justifies the terms
``maximally crossing'', ``maximally nesting'', and ``maximally aligned'' in
Definition \ref{def:stats}.(iii).
\end{rmk}

\begin{exa}\label{exa:type}
Suppose that $W$ has type $D_6$, and that $R=\{\ep_1 \pm
\ep_2,\ep_3\pm\ep_6,\ep_4 \pm \ep_5\}$ is the maximal orthogonal set of positive
roots corresponding to the matching $\{12,36,45\}$ (via the natural bijection of
Lemma \ref{lem:dmax} (ii)). In this case, $R$ contains two alignments,
corresponding to the pairs $\{12,45\}$ and $\{12, 36\}$, and one nesting,
corresponding to the pair $\{36, 45\}$. The type of $R$ is therefore $A^2 C^0
N^1$ (or $AAN$).
\end{exa}

We now state the main theorem of this section. Its proof will occupy the next
subsection.

\begin{theorem}\label{thm:qp}
Let $W$ be a Weyl group of type $E_7$, $E_8$, or $D_n$ for $n$ even, and let $X$
be the set of maximal orthogonal sets of positive roots of $W$, regarded as a
$W$-set under the action \[ w(\{\be_1,\cdots,\be_n\})=\{\abs{w(\be_1)}, \cdots,
\abs{w(\be_n)}\}.\] Then the set $(X, \lambda)$ is a quasiparabolic set for $W$,
where $\lambda : X \rightarrow \Z$ is the level function $\lambda(x) = C(x) +
2N(x)$.
\end{theorem}

\subsection{Proof of Theorem \ref{thm:qp}}

\label{sec:proof}
We will prove Theorem \ref{thm:qp} by showing that the set $X$ is a scaled
$W$-set satisfying the axioms (QP1) and (QP2) of Definition \ref{def:qpset}. To
this end, we first study how the action of a reflection $s_\al$ corresponding to
a root $\alpha$ can affect the level of a maximal orthogonal set $R$ of roots.
Recall from Remark \ref{rmk:moves} that $s_\alpha$ must replace a coplanar
quadruple in $R$ with a feature of a different type whenever $R$ does not
contain $\pm\al$. We will therefore examine how such feature replacements affect
the level function $\lambda$. Also recall from Section \ref{sec:roots} that the
root system $\Phi$ is equipped with a natural partial order $\le$ defined by the
condition that $\al\le \beta$ if and only if $\beta-\al$ is a nonnegative linear
combination of simple roots. We will frequently use the order $\le$ throughout
the proofs.

\begin{exa}\label{exa:moves}
Let $W=W(D_6)$ and $R=\{\ep_1 \pm \ep_2,\ep_3\pm\ep_6,\ep_4 \pm \ep_5\}$ be as
in Example \ref{exa:type}. Let $\al = \ep_2 - \ep_4$, so that $s_\al$ acts as
the transposition $(2, 4)$.  In this case, the set $s_\al(R)$  corresponds to
the matching $\{14,25,36\}$ and has type $CCC$. The reflection $s_\al$ changes
the alignment $Q = \{\ep_1 \pm \ep_2, \ \ep_4 \pm \ep_5\}$ to the crossing
$Q'=s_\al(Q) = \{\ep_1 \pm \ep_4, \ \ep_2 \pm \ep_5\}$, so $\al$ moves an $A$ to
a $C$.  Note that while $s_\al$ changes the quadruple $Q$  from an $A$ to a $C$
locally, globally $s_\al$ does not change the type of $R$ from $AAN$ to $ACN$
but to $CCC$. This is because after the application of $s_\al$, each collinear
pair of roots in $Q$ becomes a new collinear pair that forms a new type of
coplanar quadruple with the collinear pair of roots $\ep_3\pm\ep_6$ outside $Q$.

Part (ii) of the next proposition, however, will imply that if $\al$ is minimal
among roots moving an $A$ to a $C$, then the global change in the type of $R$
will mirror this local change, so that if $R$ has type $A^pC^qN^r$ then
$s_\al(R)$ has type $A^{p-1}C^{q+1}N^r$. In our example, the root
$\al'=\ep_2-\ep_3$ satisfies the minimality condition since it is simple. The
reflection $s_{\al'}$ changes the coplanar quadruple
$\{\ep_1\pm\ep_2,\ep_3\pm\ep_6\}$ of type $A$ to the coplanar quadruple
$\{\ep_1\pm\ep_3,\ep_2\pm\ep_6\}$ of type $C$, and changes the set $R$ of type
$AAN$ to a maximal orthogonal set of type $ACN$.
\end{exa}

\begin{prop}\label{prop:miracle}
Let $W$ be a Weyl group of type $E_7$, $E_8$, or $D_{n}$ with $n$ even, and let
$R$ be a maximal set of  orthogonal positive roots of type $A^p C^q N^r$. Let
$\lambda=C+2N$ be the level function from Definition \ref{def:stats}.
\begin{enumerate}
    \item[{\rm (i)}]{If $\al_i$ is a simple root, then either $\al_i \in R$, or
         $\lambda(s_{\al_i}(R)) \ne \lambda(R)$.  If $\lambda(s_{\al_i}(R)) >
         \lambda(R)$, then we have $\lambda(s_{\al_i}(R)) = \lambda(R) + 1$, and
         either (1) $s_{\al_i}$ moves an $A$ to a $C$ and $s_{\al_i}(R)$ has
         type $A^{p-1}C^{q+1}N^r$, or (2) $s_{\al_i}$ moves a $C$ to an $N$ and
     $s_{\al_i}(R)$ has type $A^pC^{q-1}N^{r+1}$.}
 \item[{\rm (ii)}]If $Q$ is an alignment in $R$, $Q'$ is the corresponding
     crossing quadruple, and $R'=(R\setminus Q)\cup Q'$,  then $d = \lambda(R')
     - \lambda(R)$ is a positive odd number.  If there is a positive root $\al$
     such that $s_\alpha(Q)=Q'$ but no positive root $\al'<\al$ moves an $A$ in
     $R$ to
     a $C$, then $R'$ has type $A^{p-1}C^{q+1}N^{r}$.
 \item[{\rm (iii)}]If $Q$ is a crossing in $R$, $Q'$ is the corresponding
      nesting quadruple, and $R'=(R\setminus Q)\cup Q'$, then $A(R) = A((R
      \backslash Q) \cup Q')$ and $d = \lambda(R') - \lambda(R)$ is a positive
      odd number.  If there is a positive root $\al$ such that $s_\alpha(Q)=Q'$
      but no positive root $\al'<\al$ moves a $C$ in $R$ to an $N$, then $R'$
      has type $A^{p}C^{q-1}N^{r+1}$.
  \item[{\rm (iv)}] If $Q$ is an alignment in $R$, $Q'$ is the corresponding
      nesting quadruple, and $R'=(R\setminus Q)\cup Q'$, then $d = \lambda((R
      \backslash Q) \cup Q') - \lambda(R)$ is a strictly positive even number.
\end{enumerate}
\end{prop}

\begin{lemma}\label{lem:d6miracle}
Proposition \ref{prop:miracle} holds if $W$ has type $D_6$.
\end{lemma}

\begin{proof}[Proof of Lemma \ref{lem:d6miracle}]
Throughout the proof, we will identify both $R$ and the coplanar quadruples
within $R$ with their corresponding matchings (as in Lemma \ref{lem:dmax} and
Remark \ref{rmk:d4matching}). Recall that we will often write a 2-block
$\{a,b\}$ in a matching as $ab$.

Suppose that $\al_i$ is a simple root, so that the reflection $s_{\al_i}$ acts
as the transposition $(i, i+1)$. We will assume that $\al_i \not\in R$, so that
we have $R=\{ai,(i+1)b,ef\}$ and the coplanar quadruple moved by $s_{\al_i}$ is
$Q=\{ai,(i+1)b\}$.  Let $Q'$ be any other coplanar quadruple in $R$. Then $Q'$
is of the form $\{xy,ef\}$ with $x\in \{a,b\}$ and $y\in \{i,i+1\}$, and we have
$s_{\al_i}(Q')=\{\{x,s_{\al_i}(y)\},ef\}$. Since the numbers $i$ and $i+1$ are
only distance 1 apart, the four elements $x, s_{\al_i}(y), e,f$ appearing in
$Q'$ have the same relative order as the numbers $x,y,e,f$; therefore $Q'$ and
$s_{\al_i}(Q')$ have the same type. It follows that $Q$ is the only quadruple in
$R$ that is changed to a quadruple of another type by  $s_{\al_i}$.  Note that
$Q$ will be an alignment if $a < i < i+1 < b$; $Q$ will be a crossing if $a < b
< i$ or $i+1 < a < b$ or $b < i < i+1 < a$; and $Q$ will be a nesting if $b < a
< i$ or $i+1 < b < a$.  We have $\lambda(s_{\al_i}(R)) = \lambda(R)+1$ in the
first three of these six cases and $\lambda(s_{\al_i}(R))=\lambda(R)-1$ in the
last three cases. The first of the six cases corresponds to the situation in
(1), and the second and third cases correspond to the situation in (2). Part (i)
follows.

Suppose that $Q$ and $Q'$ are as in the statement of (ii), with $Q$ being the
alignment $\{a_1a_2,b_1b_2\}$ for some $a_1 < a_2 < b_1 < b_2$.  If $b_1 = a_2 +
1$, then the simple root $s_{a_2}$ moves an $A$ to a $C$ by moving $Q$, and we
have $R'=s_{\al_2}(R)$; therefore $R'$ has type $A^{p-1}C^{q+1}N^{r}$ by (i). If
$b_1>a_2+1$, then $Q$ must be one of the following five quadruples \[
    \{12,56\},\quad \{12,45\},\quad \{12, 46 \},\quad \{23,
    56\},\quad\text{and\quad} \{13,56\}.  \] Direct computation shows that $d$
    equals $5,1,3,1$, and $3$ in these cases, respectively. It follows that $d$
    is a positive odd number.

To prove the second assertion in (ii), we prove its contrapositive. If $R'$ does
not have type $A^{p-1}C^{q+1}N^r$, then $Q$ must be one of five quadruples
listed in the  last paragraph. We claim that in each case, for every positive
root $\al$ such that $s_\al(Q)=Q'$, there exists a positive root $\al'<\al$ that
moves another alignment in $R$ other than $Q$ to a crossing. Specifically, we may
always take $\al'$ to be $\ep_2-\ep_3$ in the first three cases and
$\ep_4-\ep_5$ in the last two cases. For example, the only possibilities for
$\al$ if $Q=\{12,56\}$ are $\ep_2\pm\ep_5$ and $\ep_1\pm\ep_6$, and for all
these possibilities the root $\al'=\ep_2-\ep_3$ is smaller than $\al$ and moves
an alignment in $R$ other than $Q$ to a crossing. This completes the proof of
the desired contrapositive.

A similar argument proves (iii). This time, we have $Q= \{a_1a_2,b_1b_2\}$ for
some $a_1 < b_1 < a_2 < b_2$. If we have $b_1=a_1+1$ or $b_2=a_2+1$, then the
simple root $\al_{a_1}$ or $\al_{a_2}$ moves $Q$ to a nesting $Q'$, so $R'$ has
type $A^pC^{q-1}N^{r+1}$ by (i). The only remaining possibility for $Q$ is
$\{14,36\}$. If $\al$ is a positive root such that $s_\al(Q)=Q'$, then $\al\in
\{\ep_1\pm\ep_3,\ep_4\pm\ep_6\}$. Direct computation shows that any simple root
$\al'<\al$ moves some crossing in $R$ other than $Q$ to a nesting.

Finally, (iv) follows by combining (ii) and (iii), since $R'$ can be obtained
from by first replacing $Q$ in $R$ with its corresponding crossing $Q''$ and
then replacing $Q''$ in the result with $Q'$.
\end{proof}

\begin{proof}[Proof of Proposition \ref{prop:miracle}]
If $W$ has type $D_4$ then there are only three possibilities for $R$, and all
the assertions follow by direct verification.  By Lemma \ref{lem:d6miracle}, we
may therefore assume that the rank of $W$ is at least $7$.

For each coplanar quadruple $Q\se R$, we define $H_Q$ to be the set of all
6-subsets $H$ of $R$ such that there exists a $D_6$-subsystem of $\allroots$
containing both $\Psi_Q$ and $H$, where $\Psi_Q$ is the $D_4$-subsystem
associated with $Q$.  By Proposition \ref{prop:overlap}, if $Q_1$ is any
coplanar quadruple in $R$, then either $Q_1=Q$, or $Q_1 \cap Q = \emptyset$, or
$\abs{Q_1\cap Q}=2$ and there is a unique element of $H_\Psi$ that contains both
$Q$ and $Q_1$.

We now prove (iii). By tracking the contributions towards the crossing number
made by the three types of coplanar quadruples $Q_1\se R$ just mentioned, we
note that
\begin{equation}
    \label{eq:Cdiff} C(R') - C(R) = -1 + \sum_{H \in H_Q} \big(C((H \backslash
      Q) \cup Q') - C(H) + 1 \big).
  \end{equation}
Here, the term $-1$ comes from the case $Q_1=Q$, since the crossing $Q$ in $R$
is replaced by the non-crossing feature $Q'$ as we change $R$ to $R'$. In the
second case where $Q_1\cap Q=\emptyset$, the quadruple $Q_1$ lies in both $R$
and $R'$ and thus does not contribute to the difference $C(R')-C(R)$.  Finally,
every $Q_1$ with $\abs{Q_1\cap Q}=2$ appears together with $Q$ in a unique
element $H$ of $H_Q$ and contributes a term in the sum over $H_Q$, where we have
added 1 to the difference $C((H\backslash Q)\cup Q')-C(H)$ to account for the
fact that the change from $Q$ to $Q'$ in $H$ has been recorded by the term $-1$
in the first case.

Similar arguments based on the facts that  $N(Q')=N(Q)+1$ and  $A(Q')=A(Q)$ show
that
\begin{equation}
    \label{eq:Ndiff} N(R') - N(R) = 1 + \sum_{H \in H_Q} \big(N((H \backslash Q)
      \cup Q') - N(H) - 1 \big)
  \end{equation} and
  \begin{equation}
      \label{eq:Adiff} A(R') - A(R) = \sum_{H \in H_Q} \big(A((H \backslash Q)
      \cup Q') - A(H) \big).
  \end{equation} Since $\lambda=C+2N$, it follows from Equations
    \eqref{eq:Cdiff} and \eqref{eq:Ndiff} that
    \begin{equation}
        \label{eq:leveldiff} \lambda(R') - \lambda(R) = 1 + \sum_{H \in H_Q}
      \big(\lambda((H \backslash Q) \cup Q') - \lambda(H) - 1 \big).
  \end{equation}
By Lemma \ref{lem:d6miracle}.(iii), each summand in the sum over $H_Q$ is zero
in Equation \eqref{eq:Adiff} and is a nonnegative even number in Equation
\eqref{eq:leveldiff}; therefore we have $A(R')=A(R)$ and the number
$d=\lambda(R')-\lambda(R)$ is a positive odd number.

To prove the last assertion in (ii), suppose that $s_\al(Q)=Q'$ for some
positive root $\al$, but no positive root $\al' < \al$ moves a $C$ to an $N$ in
$R$. The same minimality condition then applies if $R$ is replaced by an element
of $H_Q$, so every summand in the sums over $H_Q$ in Equations \eqref{eq:Cdiff}
and \eqref{eq:Ndiff} is zero by Lemma \ref{lem:d6miracle} (iii). It follows that
$R'$ has type $A^pC^{q-1}N^{r+1}$, which proves Proposition \ref{prop:miracle}
(iii).

The proof of Proposition \ref{prop:miracle} (ii) follows by a similar but
shorter argument, and the proof of Proposition \ref{prop:miracle} (iv) follows
by combining parts (ii) and (iii).

Finally, to prove Proposition \ref{prop:miracle} (i), assume that $\al_i$ is a
simple root such that $\al_i \not\in R$.  We have already proved part (i) if $W$
has type $D_4$, and this implies that either $\al_i$ moves an $A$ to a $C$ or a
$C$ to an $N$. In the former case, the conclusions follow from part (ii), and in
the latter case, they follow from part (iii), in each case because the simple
root $\al_i$ is minimal in the order $\le$.
\end{proof}

\begin{proof}[Proof of Theorem \ref{thm:qp}]
We first prove (i). Proposition \ref{prop:miracle} (i) proves that $(X,
\lambda)$ is a scaled $W$-set, so it suffices to show $(X,\lambda)$ satisfies
the axioms (QP1) and (QP2). We do so by induction on $n$.

If $n=4$ or $n=6$, then $W$ has type $D$, and the axioms (QP1) and (QP2) can be
proved by direct verification or as follows. Suppose $n=2k$. By Remark
\ref{rmk:Daction}, it suffices to show that $(X,\lambda)$ is a quasiparabolic
set for the symmetric group $S_{n}$.  We may identify the set $X$ with the set
$X'$ of fixed-point free involutions in $S_{n}=W(A_{n-1})$, with each collinear
pair $\{\ep_i\pm\ep_j\}$ in a maximal set $R\in X$ corresponding to a factor
$(i,j)$ in an involution $\iota\in X'$. Under this identification, the actions
of $S_n$ on $X'$ and $X$ coincide with each other, so it suffices to show that
$X'$ is a quasiparabolic set for $S_n$ under the level function $\lambda$.
Rains and Vazirani \cite[Section 4]{rains13} proved that $X'$ is a
quasiparabolic set for $S_n$ under the level function $h$ given by
$h(\iota)=(\ell(\iota) - k)/2$, where $\ell$ denotes Coxeter length, so it
further suffices to show that whenever an involution $\iota \in S_n$ corresponds
to a maximal set $R$ of positive orthogonal roots, we have $(\ell(\iota) - k)/2
= \lambda(R)$. This can be proved by an exhaustive check or by induction on
$\lambda(R)$  by using the first three cases in the second paragraph of the
proof of Lemma \ref{lem:d6miracle}.  For example, the involution $$ \iota =
(13)(26)(45) = s_2 s_1 s_3 s_5 s_4 s_3 s_5 s_4 s_2 \in S_6 $$ corresponds to the
set $R=\{13,26,45\}$ of type $ACN$. In this case, we have $(\ell(\iota) - k)/2 =
(9-3)/2$ and $\lambda(R)=3$, as required.

Now assume $n\ge 7$. Let $r=s_\alpha \in T$ be the reflection corresponding to a
root $\al$, and let $x \in X$. If $rx \ne x$, then $\pm \al$ are not in $x$, so
$r$ moves an $A$ to a $C$, or a $C$ to an $N$, or an $A$ to an $N$, or vice
versa by Remark \ref{rmk:moves}. It follows from Proposition \ref{prop:miracle}
(ii), (iii), and (iv) that $\lambda(rx)> \lambda(x)$ in all the first three
cases and $\lambda(rx)<\lambda(x)$ in the last three cases; therefore, axiom (QP1)
holds.

To prove axiom (QP2), assume that we have $r \in T$, $x \in X$, $s \in S$,
$\lambda(rx) > \lambda(x)$ and $\lambda(srx) < \lambda(sx)$.  Then the
definition of scaled $W$-sets forces $\lambda(rx) =\lambda(sx)= \lambda(x)+1$,
so each of $r$ and $s$ must be an $AC$ or $CN$ move by Proposition
\ref{prop:miracle} (ii), (iii), and (iv). Let $Q_1$ and $Q_2$ be the coplanar
quadruples of roots in $R$ moved by $r$ and $s$, respectively. Then $Q_1$ and
$Q_2$ are disjoint, or coincide with each other, or intersect in two elements by
Proposition \ref{prop:overlap}. If $Q_1$ and $Q_2$ were disjoint, we would have
$Q_2\se r(R)$, so that $s$ would move $Q_2$ in $r(R)$, and Proposition
\ref{prop:miracle} (ii) and (iii) would imply that $
\lambda(srx)=\lambda(rx)+1>\lambda(rx)=\lambda(sx), $ contradicting the
assumption that $\lambda(srx)<\lambda(sx)$.  If $Q_1=Q_2$, then the fact that
$\lambda(rx)=\lambda(sx)=\lambda(x)+1$ implies that $s$ and $r$ must both be
$AC$ moves or both be $CN$ moves, according as $Q_1=Q_2$ is an alignment or
crossing, by Proposition \ref{prop:miracle}. It follows from Remark
\ref{rmk:moves} that $rx=sx$. Finally, if $\abs{Q_1\cap Q_2}=2$, then
Proposition \ref{prop:overlap} implies that there is a subsystem $\Sigma$ of
type $D_6$ containing both $Q_1$ and $Q_2$.  Applying the inductive hypothesis
to $\Sigma$ proves that $rx = sx$, which completes the proof.
\end{proof}

We end this subsection by recording some useful consequences of Proposition
\ref{prop:miracle} concerning sums of $n$-roots:

\begin{cor}\label{cor:moveup}
Let $W$ be a Weyl group of type $E_7$, $E_8$, or $D_{n}$ with $n$ even.
\begin{enumerate}
\item[{\rm (i)}]{If $\be \leq_Q \gamma$ are two positive $n$-roots that are
     comparable in the quasiparabolic order, then we have $\sigma(\be) \leq
 \sigma(\gamma)$, with equality if and only if we have $A(\be) = A(\gamma)$.}
\item[{\rm (ii)}] {If $\al_i$ is a simple root and $R$ is a maximal orthogonal
     set of positive roots, then we have $$ \sigma(s_{i}(R)) =
     \begin{cases} \sigma(R) - 2\al_i & \text{\ if\ } \al_i \in R \text{\ or\ }
              \al_i \text{\ moves\ a\ } C \text{\ to\ an\ } A,\\ \sigma(R) +
              2\al_i & \text{\ if\ } \al_i \text{\ moves\ an\ } A \text{\ to\ a\
              } C,\\ \sigma(R) & \text{\ otherwise}.
      \end{cases}$$}
  \item[{\rm (iii)}]{If $\al_i$ is a simple root and $\beta$ is a positive
         $n$-root of type $A^pC^qN^r$ such that $\sigma(s_{\al_i}(\be)) >
         \sigma(\be)$, then $s_{\al_i}(\be)$ is a positive $n$-root of type
         $A^{p-1}C^{q+1}N^r$, and we have $\sigma(s_{\al_i}(\be)) - \sigma(\be)
     = 2\al_i$. If $\be$ is nonnesting, then so is $s_{\al_i}(\be)$.}
\end{enumerate}
\end{cor}

\begin{proof}
All the claims can be proved by examining the effects of various types of
(simple) reflections on sums and levels of $n$-roots recorded in Theorem
\ref{thm:cna} and Proposition \ref{prop:miracle}. We first prove (i). By the
definition of $\le_Q$, it is enough to consider the case where $r \in T$ is a
reflection and $\gamma = r(\beta)$ satisfies $\lambda(\beta)<\lambda(\gamma)$.
By Proposition \ref{prop:miracle}, this implies that  $r$ is an $AC$, $CN$, or
$AN$ move. The proof of (i) then follows from Theorem \ref{thm:cna} (vi).

We now prove (ii). If $\al_i \in R$, then $\sigma(s_i(R))=\sigma(R)-2\al_i$
because $s_{i}(\al_i) = -\al_i$. If $\al_i$ is a $CN$ or $NC$ move, Theorem
\ref{thm:cna} (vi) implies that $\sigma(s_{\al_i}(R)) = \sigma(R)$. It follows
from Proposition \ref{prop:miracle} (i) that the only other possibility is that
$\al_i$ moves an $A$ to a $C$ or vice versa. By Theorem \ref{thm:cna} (vi), this
can only happen if $\al_i$ is the unique simple root in the corresponding
nesting, i.e., the root $\al_2$ in the explicitly constructed sets
$\Psi_2^+=\Psi^+_N$ in the proof of that theorem. Explicit computation shows
that root sums in this case differ by $2\al_i$ in the precise manner described
in the statement, which completes the proof of (ii).

The first part of (iii) follows immediately from (ii) and Proposition
\ref{prop:miracle} (i).  The assertion about nonnesting $n$-roots follows from
the special case $r=0$.
\end{proof}

\begin{rmk}\label{rmk:burns}
Burns and Pfeiffer \cite[Theorem 1.2]{burns23} prove that if $T$ is a maximal
order abelian subgroup of one of the groups $W$ in Theorem \ref{thm:qp}, then
$T$ is elementary abelian of order $2^n$, where $n$ is the rank of $W$.  They
also prove that the set of all such subgroups forms a single conjugacy class
\cite[Theorem 3.1]{burns23}.  It follows that the stabilizers of the elements $x
\in X$ can be defined abstractly from the group structure of $W$: they are the
normalizers of the maximal order abelian subgroups of $W$.
\end{rmk}

\subsection{Extremal elements}
\label{sec:minmax}

In this subsection, we identify $X$ with the set $\Phi^+_n$ of positive
$n$-roots as usual and discuss an application of Theorem \ref{thm:qp} concerning
the maximally aligned and maximally nested $n$-roots, which turn out to be the
unique $W$-minimal and $W$-maximal elements of the set $X$.  The uniqueness of
the maximally aligned and maximally nested $n$-root is not a priori clear, but
it will follow conveniently from the general theory of quasiparabolic sets.

\begin{prop}
    \label{prop:minmax}
Let $W$ be a Weyl group of rank $n$ of types $E_7$, $E_8$, or $D_n$ for $n$
even, and let $M$ be the number of \coplanar quadruples in a positive $n$-root.
\begin{enumerate}
    \item[{\rm (i)}]{There is a unique positive $n$-root, $\theta_A$, of type
     $A^M$, and it corresponds to the unique $W$-minimal element of the
 quasiparabolic set $X$ of Theorem \ref{thm:qp}.}
\item[{\rm (ii)}]{There is a unique positive $n$-root, $\theta_N$,  of type
     $N^M$, and it corresponds to the unique $W$-maximal element of the
 quasiparabolic set $X$ of Theorem \ref{thm:qp}.}
    \end{enumerate}
\end{prop}

\begin{proof}
We recall that by Theorem 2.8, Remark 2.9, and Corollary 2.10 of \cite{rains13},
every orbit of a quasiparabolic set contains at most one $W$-minimal and at most
one $W$-maximal element.  If such a $W$-minimal or $W$-maximal element exists,
then it can be identified as the unique element in the orbit with the minimal or maximal
possible level, respectively.

The set $X$ is finite, so it has at least one $W$-maximal and one $W$-minimal
element.  Since $X$ consists of a single $W$-orbit by Lemma
\ref{lem:transitive}, it follows from the paragraph preceding this proposition
that $X$ has a unique $W$-maximal and a unique $W$-minimal element, and that
they are the unique elements with the minimal and maximal possible level.

Let $R$ be a maximal orthogonal set of positive roots. If $R$ contains any
coplanar quadruple $Q$ that is a crossing or a nesting, then by Remark
\ref{rmk:moves}  we can find a reflection $s_\al$ that moves $Q$ to an alignment
or a crossing, respectively, and in both cases Proposition \ref{prop:miracle}
implies that $\lambda(s_\al(R))<\lambda(R)$. Iterating this procedure proves the
existence of an $n$-root of type $A^M$, which achieves the lowest possible value
of $\lambda$.  The uniqueness property of the previous paragraph then completes
the proof of (i). Part (ii) can be proved similarly, by using the fact that any
alignment or crossing in $R$ would induce a level-increasing $AC$ or $CN$ move.
\end{proof}

We will prove shortly, in Proposition \ref{prop:neststab}, that $W$ also has a
unique maximally crossing element, $\theta_C$. The element $\theta_C$ will be
the unique minimal element in a quasiparabolic set of a parabolic subgroup of
$W$.

\begin{rmk}
    \label{rmk:monoidal}
In Theorem \ref{thm:positive} below, we will introduce another partial order,
$\leq_\Be$, on the positive $n$-roots. The argument of \cite[Section 6]{gx3} can
be adapted to show that, under suitable identifications, $\leq_\Be$ refines the
{\it monoidal order} introduced by Cohen, Gijsbers, and Wales \cite[Section
3]{cohen06}.  The quasiparabolic order has $\theta_N$ and $\theta_A$ as its
unique maximal and unique minimal elements, whereas the monoidal order (and
$\leq_\Be$) has $\theta_C$ as its unique maximal element and has multiple
minimal elements.
\end{rmk}


\section{Feature-avoiding elements}\label{sec:non}
In this section, we develop the properties of $n$-roots that avoid
features of a given type: the alignment-free, noncrossing and nonnesting
elements.  We show that the alignment-free elements form a quasiparabolic set
$X_I$ of a maximal standard parabolic subgroup $W_I$ of $W$, and that the unique
maximally crossing $n$-root $\theta_C$ is the unique $W_I$-minimal
element of $X_I$.  We also show that the sets of noncrossing elements and of
nonnesting elements both form bases of the Macdonald representation 
$\macd{\Phi}{nA_1}$ (Definition \ref{def:mac}).  Moreover,
the basis of noncrossing elements may be viewed as a canonical basis and behaves
in a way that is reminiscent of the set of simple roots of a root system
(Theorem \ref{thm:positive}).  The basis of nonnesting elements admits an
interesting combinatorial characterization: it is a distributive lattice induced
by a suitable Bruhat order (Theorem \ref{thm:nonnest}). Finally, we introduce
the notion of $\sigma$-equivalence classes to tie together the alignment-free,
noncrossing, and nonnesting elements. These equivalence classes turn out to be
intervals with respect to the quasiparabolic order on $X$, and the set $X_I$ of
alignment-free elements form the top class with respect to a natural partial
order. Any set of $\sigma$-equivalence class representatives forms a basis of
the Macdonald representation, and the change of basis matrices between any pair
of such bases, including the noncrossing and nonnesting bases, are unitriangular
(Theorem \ref{thm:nesttri}).

Throughout the rest of this section, we assume that we are working with
a Weyl group $W$ of rank $n$ and type $E_7, E_8$ or $D_n$ for $n$ even. All
results hold independently of the rank and type of $W$, and we shall omit the
statement of the above assumption except in the main theorems. We define the
\emph{sum} of each positive $n$-root $\gamma$ to be the sum of the components of
$\gamma$, and we denote it by $\sigma(\gamma)$.

\subsection{Alignment-free elements}
Recall from Proposition \ref{prop:minmax} that $W$ has a unique positive
$n$-root $\theta_N$ that avoids both alignments and crossings. We will use
$\theta_N$ to help study general alignment-free elements.

\begin{prop}\label{prop:ancomp}
    Let $\theta_N$ be the unique positive $n$-root of type $N^M$.
    \begin{enumerate}
        \item[{\rm (i)}]{A noncrossing $n$-root (i.e., one of type $A^pN^r$) has
     a simple component.}
 \item[{\rm (ii)}]{The $n$-root $\theta_N$ has a unique simple component,
     $\al_x$.}
 \item [{\rm (iii)}] If $\al_i$ is a simple root, then
     $B(\sigma(\theta_N),\al_i)\ge 0$, where equality holds if and only if
     $\al_i\neq\al_x$.
 \item[{\rm (iv)}] Let $W_I$ be the parabolic subgroup of $W$ generated by the
     set $S \backslash \{\al_x\}$. Then the stabilizer of $\sigma(\theta_N)$ is
     precisely $W_I$, and we have $B(\sigma(\theta_N), \al_i) \geq 0$ for all
     simple roots $\al_i$.
\end{enumerate}
\end{prop}

\begin{proof}
Let $\gamma$ be a noncrossing $n$-root of type $A^pN^r$, and let $R$ be the set
of components of $\gamma$. Assume for a contradiction that $R$ contains no
simple root, and let $\be$ be a root of minimal height in $R$. The bilinear form
$B$ has the property that $B(\al,\al')\in \{-2,-1,0,1,2\}$ for any roots
$\al,\al'$, with $B(\al,\al')=2$ if and only if $\al=\al'$.  Furthermore, for
any positive root $\al$, there exists a simple root $\al_i$ such that
$B(\al,\al_i)>0$ \cite[Theorem 1.5]{humphreys90}. It follows that there exists a
simple root $\al_i$ such that $B(\be, \al_i) = 1$.

Since $\al_i$ is not in $R$, it moves a coplanar quadruple $Q\se R$, and we have
$\be\in Q$ since $B(\beta,\al_i)\neq 0$. Let $\Psi$ be the $D_4$-subsystem
associated to $Q$.  By hypothesis, $Q$ is either an alignment or a nesting, and
$\be$ is an element of $Q$ of minimal height. It follows from the explicit
description of the sets $\Psi^+_A=\Psi^+_3$ and $\Psi^+_N=\Psi^+_2$ in the proof
of Theorem \ref{thm:cna} that $\be$ is a root in the induced simple system of
$\Psi$. However, since $\al_i$ is simple root of $W$, $\al_i$ is also
in this induced simple system.  This is a contradiction, because we cannot have
$B(\gamma_1, \gamma_2) > 0$ for two simple roots $\gamma_1$ and $\gamma_2$ in a
root system. This completes the proof of (i).

Now suppose further that $\gamma=\theta_N$. It follows from (i) that $R$
contains a simple root, so assume for a contradiction that $R$ contains two
simple roots, $\al_x$ and $\al_y$.  Let $P$ be a path from $x$ to $y$ in the
Dynkin diagram $\Gamma$, and let $\be$ be the root $\sum_{p \in P} \al_p$. Note
that $B(\be, \al_x) = B(\be, \al_y) = 1$, so that $\al_x$ and $\al_y$ are both
elements of the \coplanar quadruple $Q$ consisting of the roots moved by $\be$.
Let $\Psi$ be the $D_4$-subsystem associated with $Q$. Then $\al_i,\al_j$ are
both induced simple roots of $\Psi$ since they are simple roots of $W$. However,
the type of $R$ is $N^M$, so $Q$ is a nesting and thus contains a unique minimal
root by the description of the set $\Psi^+_N=\Psi^+_2$ in the proof of Theorem
\ref{thm:cna}.  This is a contradiction, and (ii) follows.

To prove (iii), let ${\al_i}$ be a simple root. If $\al_i = \al_x$, then $\al_x$
is a component of $\theta_N$ and we have $$ s_{i}(\sigma(\theta_N)) =
\sigma(\theta_N) - 2\al_i .$$ If $\al_i \ne \al_x$, then $\al_i$ is not a
component of $\theta_N$, and Proposition \ref{prop:miracle} (i) implies that
$\al_i$ moves an $N$ to a $C$. Corollary \ref{cor:moveup} (ii) then implies that
$$ s_{i}(\sigma(\theta_N) = \sigma(\theta_N) .$$ Equation \ref{eq:ref} implies
    that  $B(\al_i, \sigma(\theta_N))\ge 0$ for all $i$, with equality holding
    if and only if $\al_i\neq \al_x$.

Part (iv) follows from (iii) by \cite[Theorem 1.12 (a)]{humphreys90}, which says
that the stabilizer of $\sigma(\theta_N)$ in $W$ is generated by the simple
reflections it contains.
\end{proof}

\begin{prop}\label{prop:neststab}
Let $\theta_N$ be the unique positive $n$-root of type $N^M$, let $\al_x$ be the
unique simple component of $\theta_N$, and let $W_I$ be the parabolic subgroup
of $W$ generated by the set $S \backslash \{\al_x\}$.
\begin{enumerate}
    \item[{\rm (i)}]{The $W_I$-orbit of positive $n$-roots that contains
     $\theta_N$ is a quasiparabolic set $(X_I, \lambda_I)$ for $W_I$, where
 $\lambda_I$ is the restriction of $\lambda$ to $X_I$.}
\item[{\rm (ii)}]{The following are equivalent for a positive $n$-root $\be$:
        \begin{enumerate}
            \item[{\rm (1)}]{$\be$ has type $C^q N^r$ for some $q$ and $r$;}
            \item[{\rm (2)}]{$\sigma(\be) = \sigma(\theta_N)$;}
            \item[{\rm (3)}]{$\be$ is an element of $X_I$.}
    \end{enumerate}} \noindent In particular, the elements of the quasiparabolic set $X_I$
        are precisely the alignment-free positive $n$-roots.
    \item[{\rm (iii)}]{There is a unique positive $n$-root, $\theta_C$, of type
     $C^M$, and it corresponds to the unique $W$-minimal element of the
 quasiparabolic set $X_I$.}
\item[{\rm (iv)}]{If $\al$ is a root of $W_I$ and $\be$ is an $n$-root in $X_I$
     whose components do not contain $\pm\al$, then $\lambda(s_\al(\be)) =
 \lambda(\be) + 1 \mod 2$. }
\end{enumerate}
\end{prop}

\begin{proof}
Part (i) follows from Theorem \ref{thm:qp} by restriction.

We now prove the implication (1) $\Rightarrow$ (2) of part (ii).  Let $\be$ be
an $n$-root of type $C^q N^r$. If $q = 0$, then $\beta=\theta_N$ by Proposition
\ref{prop:minmax} (ii) and (2) follows immediately, so suppose that $q > 0$. By
Remark \ref{rmk:moves} and Proposition \ref{prop:miracle} (iii), there exists a
reflection $\al$ that moves a crossing to a nesting in such a way that the
$n$-root $\be' = s_{\al}(\be)$ has type $C^{q-1}N^{r+1}$.  Corollary
\ref{cor:moveup} (ii) proves that $\sigma(\be') = \sigma(\be)$.  It now follows
from induction on $q$ that $\sigma(\be) = \sigma(\theta_N)$, proving (2).

To prove (2) $\Rightarrow$ (3), assume that $\sigma(\be) = \sigma(\theta_N)$. By
Lemma \ref{lem:transitive}, there exists $w \in W$ such that $w(\theta_N) =
\be$, so we have $$
w\sigma(\theta_N)=\sigma(w(\theta_N))=\sigma(\be)=\sigma(\theta_N) .$$ It
follows that $w \in W_I$ and $\be \in X_I$, which proves (3).

To prove (3) $\Rightarrow$ (1), let $ \be = w(\theta_N)$ for some $w\in W_I$ and
let $w= s_{i_1}s_{i_2} \cdots s_{i_k}$ be a reduced word of $w$. Then each
simple reflection $s_{i_j}$ fixes $\sigma(\theta_N)$, so it follows from
Corollary \ref{cor:moveup} (ii)  that $s_{i_j}$ is a $CN$ or $NC$ move. It
follows from Proposition \ref{prop:miracle} (i) that $\be$ has type $C^q N^r$,
which implies (1) and completes the proof of (ii).

To prove (iii), note that the quasiparabolic set $X_I$ is finite and transitive,
so it follows, as in the proof of Proposition \ref{prop:minmax} (i), that $X_I$
has a unique minimal element (with respect to the quasiparabolic order), namely, the unique element having minimal
level in $X_I$.  The elements of $X_I$ are all of type $C^qN^r$ by (ii) where
$q+r=M$ and $M$ is as in Corollary \ref{cor:mcount}, so we have
$\lambda_I(\gamma)\ge M$ for any element $\gamma\in X_I$, with equality holding
if and only if $\gamma$ has type $C^M$. To prove (iii) it now remains to show
that such an element exists.

Let $\be$ be an $n$-root in $X_I$, and suppose $\beta$ has type $C^q N^r$ where
$r > 0$. Then $\sigma(\be)=\sigma(\theta_N)$, and $\beta$ admits an $NC$ move by
a reflection $s_\al$ corresponding to some root $\al$.  Theorem \ref{thm:cna}
(vi) implies that $\sigma(s_{\al}(\be)) = \sigma(\be)=\sigma(\theta_N)$, so
$s_\al(\be)$ is in $X_I$ and has type $C^a N^b$ by (ii).  Proposition
\ref{prop:miracle} (iii) implies that $a > q$ and $b < r$, and that $s_\al(\be)$
has a lower level than $\be$. It follows that $X_I$ has an element of type
$C^M$, which completes the proof of (iii).

Suppose that $\al$ and $\be$ are as in the statement of (iv). Since neither of
$\pm \al$ is a component of $\be$, the reflection $s_\al$ must move a $C$ to an
$N$ or vice versa by Remark \ref{rmk:moves} and (ii).  Proposition
\ref{prop:miracle} (iii) then implies that $\lambda(\be)$ and
$\lambda(s_\al(\be))$ have opposite parities.
\end{proof}

\begin{rmk}
\label{rmk:xht}
With some work, it can be shown that the positive $n$-roots in $X_I$ are also
exactly the positive $n$-roots $\beta$ with the property that every component
$\al$ of $\beta$ has $x$-height 1, in the sense that $\al$ expands into a linear
combination of simple roots where the simple root $\al_x$ appears with
coefficient $1$. Since the simple roots of $W_I$ do not include $\al_x$, it
follows that no root of $W_I$ divides any $n$-root in $X_I$. In other words, the
``if'' condition in Proposition \ref{prop:neststab} (iv) in fact holds for every
root $\al$ of $W_I$ and every $n$-root $\beta\in X_I$.  This implies that the
sets of all $n$-roots in $X_I$ with even levels and of all $n$-roots in $X_I$
with odd levels are interchanged by $s_\al$ for every $\al\in W_I$.  In
particular, these two sets have the same cardinality.
\end{rmk}

\subsection{Two bases} 
\label{sec:twobases}
The goal of this subsection is to prove that the noncrossing $n$-roots and the
nonnesting $n$-roots each form a basis for the Macdonald representation
$\macd{\Phi}{nA_1}$.  The
proof is based on a commutative version of Bergman's diamond lemma
\cite{bergman78}, which is a special case of Newman's diamond lemma
\cite{newman42}. We define the {\it crossing order}, $\leq_C$, on the set of
positive $n$-roots by declaring that $\be \leq_C \gamma$ if either $\sigma(\be)
< \sigma(\gamma)$, or both $\sigma(\be) = \sigma(\gamma)$ and $\be \geq_Q
\gamma$, where $\leq_Q$ is the quasiparabolic order.  Similarly, we define the
{\it nesting order}, $\leq_N$, on the set of positive $n$-roots by declaring
that $\be \leq_N \gamma$ if either $\sigma(\be) < \sigma(\gamma)$ (with respect
    to the order
$\le$ on roots), or both
$\sigma(\be) = \sigma(\gamma)$ and $\be \leq_Q \gamma$.

Given any relation of the form $\gamma \gamma_C = \gamma \gamma_N + \gamma
\gamma_A$ among three $n$-roots in the setting of Theorem \ref{thm:cna} (where
$\gamma_C, \gamma_N,$ and $\gamma_A$ are the crossing, nesting, and alignment
corresponding to the same type-$D_4$ subsystem of the root system of $W$,
respectively), 
we
have $\gamma\gamma_A<_C\gamma\gamma_C$ and
$\gamma\gamma_A<_N\gamma\gamma_N$ because
$\sigma(\gamma_A)<\sigma(\gamma_C)=\sigma(\gamma_N)$ by Theorem \ref{thm:cna}
(vi). We also have $\gamma_C<_N \gamma_N$ and $\gamma_N <_C \gamma_C$ by
Proposition \ref{prop:miracle} (iii) and the definition of $\le_Q$, because for
any component $\al$ of $\theta_A$, the reflection $s_\al$ moves the components
of $\gamma_C$ to those of $\gamma_N$ by Theorem \ref{thm:cna} (vi).  It also
follows that $\lambda(\gamma\gamma_C)< \lambda(\gamma\gamma_N)$ and
$\gamma\gamma_C \leq_Q \gamma\gamma_N$.  We may therefore regard the relations
$\gamma \gamma_C = \gamma \gamma_N + \gamma \gamma_A$ as directed \emph{reduction
rules}, each of which operates on a single term $\lambda_i \beta_i$ in a linear
combination $\sum_i \lambda_i \beta_i$, where the $\beta_i$ are positive
$n$-roots.  Each reduction rule can be used either (a) to express a positive
$n$-root $\gamma\gamma_C$ containing a crossing as the sum of two other positive
$n$-roots $\gamma\gamma_A, \gamma\gamma_N$ that are strictly lower than it in
the crossing order, or (b) to express a positive $n$-root $\gamma\gamma_N$
containing a nesting as a linear combination of two other positive $n$-roots
$\gamma\gamma_A,\gamma\gamma_C$ that are strictly lower than it in the nesting
order.

In order to apply the diamond lemma, we need to know (a) that it is never
possible to apply an infinite sequence of reduction rules to a linear
combination of $n$-roots, and (b) that the reduction rules are {\it confluent}.
The latter condition means that if $m$ is a linear combination of $n$-roots and
if $f_1$ and $f_2$ are two different reductions that can be applied to $m$, then
the linear combinations $f_1(m)$ and $f_2(m)$ themselves have a common
reduction, $m'$. In other words, it is possible to reduce $f_1(m)$ to $m'$ by
applying a suitable sequence of reductions, and it is possible to reduce
$f_2(m)$ to the same $m'$ by applying a possibly different sequence of
reductions. If these two conditions hold, the conclusion of the diamond lemma is
that every element of the module may be uniquely expressed as an element to
which no reduction rules may be applied; in other words, a unique linear
combination of noncrossing $n$-roots, or a unique linear combination of
nonnesting $n$-roots.

Conversely, the diamond lemma guarantees that if each element $m$ can be
uniquely expressed as a linear combination of nonnesting (or noncrossing)
$n$-roots, then the reduction relations are confluent.

\begin{lemma}\label{lem:d4d6}
There are $2$ nonnesting positive $4$-roots in type $D_4$, and $5$ nonnesting
positive $6$-roots in type $D_6$.  The nonnesting positive $n$-roots are
linearly independent in the Macdonald representation $\macd{\Phi}{nA_1}$ in each case.
\end{lemma}

\begin{proof}
In type $D_4$, the set in question is $\{(\ep_1^2-\ep_2^2)(\ep_3^2-\ep_4^2),\
(\ep_1^2-\ep_3^2)(\ep_2^2-\ep_4^2)\}$, which is clearly linearly independent.
In type $D_6$, the nonnesting positive $6$-roots correspond to the  matchings $$
\{14, 25, 36\}, \ \{13, 25, 46\}, \ \{13, 24, 56\}, \ \{12, 35, 46\},
\,\text{and}\; \{12, 34, 56\}.$$ One can check that this set is linearly
independent by comparing coefficients of $\ep_1^2 \ep_2^2 \ep_3^2$, $\ep_1^2
\ep_2^2 \ep_4^2$, $\ep_1^2 \ep_2^2 \ep_5^2$, $\ep_1^2 \ep_3^2 \ep_4^2$, and
$\ep_1^2 \ep_3^2 \ep_5^2$.
\end{proof}

\begin{theorem}\label{thm:indep}
Let $W$ be a Weyl group of type $E_7$, $E_8$, or $D_n$ for $n$ even. Let
$\macd{\Phi}{nA_1}$ be the Macdonald representation of $W$.
\begin{enumerate}
    \item[{\rm (i)}]{The nonnesting positive $n$-roots form a $\Q$-basis for
     $\macd{\Phi}{nA_1}$.}
 \item[{\rm (ii)}]{The noncrossing positive $n$-roots form a $\Q$-basis for
     $\macd{\Phi}{nA_1}$.}
 \item[{\rm (iii)}]{The alignment-free positive $n$-roots span
     $\macd{\Phi}{nA_1}$.}
\end{enumerate}
\end{theorem}

\begin{proof}
We first prove (i) by using the reduction rule $\gamma \gamma_N = \gamma
\gamma_C - \gamma \gamma_A$ of Theorem \ref{thm:cna} (vi) to express an $n$-root
that contains a nesting as a linear combination of $n$-roots that are strictly
lower in the nesting order.  There are no infinite descending chains in the
crossing order because there are only finitely many $n$-roots. It remains to
show that the reductions $f_i$ are confluent, by induction on the rank $n$.  By
Lemma \ref{lem:d4d6}, this is already known to be the case in types $D_4$ and
$D_6$, so we assume from now on that we have $n > 6$.

If two reductions, $f_i$ and $f_j$, affect different terms $\lambda_i \be_i$ in
the linear combination $m = \sum_i \lambda_i \be_i$, or if $f_i$ and $f_j$
affect disjoint components of the same term $\lambda_i \be_i$, then $f_i$ and
$f_j$ commute. It is then immediate that $f_i(m)$ and $f_j(m)$ have a common
reduction, namely, $f_if_j(m) = f_jf_i(m)$.  The proof of confluence now reduces
to proving that if $f_i$ and $f_j$ change at least one component in the same
$n$-root $\be$, then $f_i(\be)$ and $f_j(\be)$ have a common reduction.  In this
case, if $Q_i$ and $Q_j$ are the sets of components of $\be$ that are moved by
$f_i$ and $f_j$ respectively, then Proposition \ref{prop:overlap} (i) implies
that either $Q_i = Q_j$, or $|Q_i \cap Q_j| = 2$. In the first case, we have
$f_i = f_j$, and there is nothing to prove. In the latter case, Proposition
\ref{prop:overlap} (ii) implies that there is a root subsystem $\Psi$ of type
$D_6$ that contains $Q_i$ and $Q_j$ as coplanar quadruples. Confluence now
follows by applying the inductive hypothesis to $\Psi$, which completes the
proof of (i).

We now prove (ii) by using the reduction rule $\gamma \gamma_C = \gamma \gamma_N
+ \gamma \gamma_A$ of Theorem \ref{thm:cna} (vi) to express an $n$-root that
contains a crossing as a sum of $n$-roots that are strictly lower in the
crossing order $\leq_C$. It follows that the noncrossing positive $n$-roots
form a spanning set. There are $2$ noncrossing positive $4$-roots in type $D_4$,
corresponding to the matchings $\{12, 34\}$ and $\{14,23\}$, and $5$ noncrossing
positive $6$-roots in type $D_6$, corresponding to the matchings $$ \{16, 25,
34\}, \ \{16, 23, 45\}, \ \{14, 23, 56\}, \ \{12, 36, 45\}, \,\text{and}\; \{12,
34, 56\}.$$ These spanning sets are bases of $\macd{\Phi}{nA_1}$ by (i), and the
rest of the argument used to prove (i) now applies {\it mutatis mutandis}.

Part (iii) follows by expressing the reduction rule in the form $\gamma \gamma_A
= \gamma \gamma_C - \gamma \gamma_N.$ By Theorem \ref{thm:cna} (vi), we have
$\sigma(\gamma\gamma_A) < \sigma(\gamma\gamma_N) = \sigma(\gamma\gamma_C)$. This
implies that the relation can only be applied finitely many times before the
procedure terminates, and (iii) follows.
\end{proof}

We will refer to the bases of nonnesting and noncrossing positive $n$-roots as
the \emph{nonnesting basis} and \emph{noncrossing basis} of the Macdonald
representation.

\subsection{Properties of the noncrossing basis}
\label{sec:ncbasis}
In this subsection, we show that the noncrossing basis  behaves in the Macdonald
representation in many ways like a simple system in the reflection
representation. In particular, every $n$-root decomposes into the noncrossing
basis with coefficients of like sign, and the noncrossing $n$-roots are
precisely the minimal ones that are minimal in the sense that they are not
further decomposable.
This minimality property yields an elementary algebraic characterization.  We
also show that the maximally crossing $n$-root $\theta_C$ has a maximal
decomposition into the noncrossing basis in a natural sense, and that simple
reflections act on the noncrossing basis in a way reminiscent of the way they
act on a simple system in the reflection representation. In addition, as we
explain in Remark \ref{rmk:othersimple}, 
the noncrossing basis is a sign-coherence basis in the sense of cluster
algebras, and it also essentially agrees with an IC basis in the sense of Du
\cite{Du}.
For the above reasons, we
may think of the noncrossing basis as the canonical basis of the Macdonald
representation.

\begin{lemma}
    \label{lem:pm1}
{If $\be$ and $\lambda \be$ are both $n$-roots for some scalar $\lambda$, then we must have $\lambda = \pm 1$.}
\end{lemma}

\begin{proof}
Theorem \ref{thm:indep} implies that the scalar $\lambda$ in (ii) lies in $\Q$.
Lemma \ref{lem:transitive} implies that there exists $w \in W$ such that $w(\be)
= \lambda \be$. Because $w$ has finite order, it follows that $\lambda$ is a
root of unity, and this forces $\lambda=\pm 1$.
\end{proof}

\begin{theorem}\label{thm:positive}
Let $W$ be a Weyl group of type $E_7$, $E_8$, or $D_n$ for $n$ even, and let
$\Be$ be the set of noncrossing positive $n$-roots.
\begin{enumerate}
    \item[{\rm (i)}]{Every $n$-root is a $\Z$-linear combination of elements of
     $\Be$, with coefficients of like sign. This sign is positive if the
 $n$-root is positive, and is negative if the $n$-root is negative.}
\item[{\rm (ii)}]{A positive $n$-root is noncrossing if and only if it is not a
     positive linear combination of other positive $n$-roots.}
 \item[{\rm (iii)}] Define $\gamma\le_\Be\gamma'$ for positive $n$-roots
     $\gamma=\sum_{\be \in \Be} c_\be \be$ and $\gamma'=\sum_{\be\in \Be}c'_\be
     \be$ whenever $c_\be \leq d_\be$ for all $\be \in \Be$. Then $\le_\Be$ is a
     partial order on the set $X$ of positive $n$-roots. The maximally crossing
     element $\theta_C$ is the unique maximal element of $X$ with respect to
     $\le_\Be$.
 \item[{\rm (iv)}]{If $\gamma \in \Be$ and $\al_i$ is a simple root, then we
     have $$ s_{\al_i}(\gamma) = \begin{cases} -\gamma & \text{\ if\ } \al_i |
         \gamma;\\ \gamma + \gamma' & \text{\ otherwise,\ for\ some\ } \gamma'
         \in \Be {\text{\ such\ that\ }} \al_i | \gamma'.\\
 \end{cases} $$}
\end{enumerate}\end{theorem}

\begin{proof}
Let $\be$ be a positive $n$-root. By the proof of Theorem \ref{thm:indep}, the
result of applying reductions of the form $\gamma \gamma_C = \gamma \gamma_N +
\gamma \gamma_A$ to $\be$ until this is no longer possible has the effect of
expressing $\be$ as a positive integer linear combination of noncrossing
$n$-roots, and this procedure will always terminate after finitely many steps.
This proves (i) for positive $n$-roots, and the statement for negative $n$-roots
follows because $n$-roots occur in positive-negative pairs.

If $\be$ is a positive $n$-root that contains a crossing, then $\be$ is a
positive linear combination of other positive $n$-roots by applying the
reduction rule in the first paragraph. Conversely, suppose that $\be$ is a
noncrossing $n$-root and that $\be = \sum_i \lambda_i \be_i$, where $\lambda_i >
0$ and $\be_i$ is a positive $n$-root that is different from $\be$  for each
$i$. Part (i) implies that each of the $\be_i$ is a positive linear combination
of noncrossing $n$-roots. Because no cancellation can occur in the sum, Theorem
\ref{thm:indep} (ii) implies that this is only possible if each $\be_i$ is a
multiple of $\be$.  Collecting terms, we then have $\be = \lambda \be_i$. Lemma
\ref{lem:pm1} and the assumption that $\be_i$ is positive then imply that
$\lambda=1$ and $\be = \be_i$, which is a contradiction.

The relation $\le_\Be$ in (iii) is clearly a partial order on $X$.  Since $X$ is
finite, it contains at least one maximal positive $n$-root with respect to
$\le$. To prove (iii), it then suffices to show that for every  $n$-root
$\gamma\in X$ not equal to $\theta_C$, there is an element $\gamma'\in X$ such
that $\gamma<_\Be\gamma'$.  Let $\gamma\in X$ be an $n$-root not equal to
$\theta_C$, so that we can factorize $\gamma$ as $\gamma_1 \gamma'$, where
$\gamma'$ is either an alignment or a nesting. In either case, Theorem
\ref{thm:cna} (vi) implies that there exists a reflection $s_\al$ such that
$s_\al(\gamma')$ is a crossing, and that we have $s_\al(\gamma') = \gamma' +
\gamma''$, where $\gamma''$ is a nesting if $\gamma'$ is an alignment, or vice
versa. We then have $$ s_\al(\gamma) = \gamma + \gamma_1\gamma'' ,$$ where
$\gamma_1\gamma''$ is also a positive $n$-root. If we write  the $n$-root
$s_\al(\gamma)$ as $s_\al(\gamma) = \sum_{\be \in \Be} e_\be \be$, then it
follows from (i) that $c_\be \leq e_\be$ for all $\be \in \Be$. It follows that
$\gamma<_\Be s_\al(\gamma)$, proving (iii).

In the situation of (iv), it is immediate that if $\al_i$ is a component of
$\gamma$, then $s_{i}(\gamma) = -\gamma$.  Suppose from now on that this is not
the case, and let $A^pN^r$ be the type of $\gamma$.  Let $R$ be the set of
components of $\gamma$, let $Q\se R$ be the coplanar quadruple moved by
$\al_{i}$, and let $\Psi$ be the $D_4$-subsystem of $Q$. Then the sets $Q,
Q'=s_i(Q)$ and $Q''=\Psi^+\setminus (Q\cup Q')$ are the three distinct coplanar
quadruples partitioning the induced positive system by Proposition
\ref{prop:four} (ii) and Theorem \ref{thm:cna}. Let $\Psi^+_A, \Psi^+_C$ and
$\Psi^+_N$ be the alignment, crossing, and nesting in $\Psi^+$, respectively.
Then since $\al_i$ is a simple root, Proposition \ref{prop:miracle} (i) implies
that we must have one of the following two situations:
\begin{enumerate}
    \item[(1)] $Q=\Psi^+_A$, $Q'=\Psi^+_C$, $s_i(\gamma)$ has type
         $A^{p-1}CN^r$, and $Q''=\Psi^+_N$;
     \item[(2)] $Q=\Psi^+_N$, $Q'=\Psi^+_C$, $s_i(\gamma)$ has type
         $A^pCN^{r-1}$, and $Q''=\Psi^+_A$.
 \end{enumerate} Let $\gamma_x=\prod_{\beta\in x}\beta$ for all $x\in
    \{Q,Q',Q''\}$, and write $\gamma=\gamma_1\gamma_Q$. Then we have
    $\gamma_{Q'}=\gamma_Q+\gamma_{Q''}$, and thus,
    \[s_i(\gamma)=\gamma_1\gamma_{Q'}=\gamma+\gamma_1\gamma_{Q''} \] in both of
    the above cases. We have $\al_i\in \Psi$ by Proposition \ref{prop:four}
    (ii), and $\al_i$ must lie in the induced simple system of $\Psi$ since it
    is simple. Theorem \ref{thm:cna} (v) then implies that $\al\notin Q'$, and
    we have $\al_i\notin Q$ by assumption, so we have $\al\in Q''$. It follows
    that $\al$ divides the $n$-root $\gamma''=\gamma_1\gamma_{Q''}$.

It remains to prove that $\gamma''$ is noncrossing.  We treat case (1) first.
If $\al$ is any root that is minimal in $Q$, then $\al$ moves the crossing $Q'$
to the nesting $Q''$ by Theorem \ref{thm:cna} (vi), so that
$\gamma''=s_\al(s_i(\gamma))$.  Since $s_i(\gamma)$ has type $A^{p-1}CN^r$, $Q'$
is the unique crossing in $s_i(\gamma)$; therefore any root that moves a $C$ in
$s_i(\gamma)$ to an $N$ must move $Q'$, and it must move $Q'$ to $Q''$ by
Proposition \ref{prop:four} (ii). Together with Theorem \ref{thm:cna} (vi), this
further implies that any root moving $Q'$ to $Q''$ must come from $Q$, so it
follows that $\al$ is minimal among roots moving a $C$ in $s_i(\gamma)$ to an
$N$. It follows from Proposition \ref{prop:miracle} (iii) that
$\gamma''=s_\al(s_i(\gamma))$ has type $A^{p-1}N^{r+1}$; therefore $\gamma''$ is
noncrossing. A similar argument shows that $\gamma''$ has type $A^{p+1}N^{r-1}$
in case (2), so $\gamma''$ is noncrossing in both cases.
\end{proof}

\begin{rmk}\label{rmk:othersimple}
\begin{enumerate}
\item[(i)]{Since the Weyl group $W$ acts transitively on $n$-roots in types
     $E_7, E_8$, and $D_n$ for $n$ even, the first assertion of Theorem
     \ref{thm:positive} (i) is equivalent to the assertion that the noncrossing
     basis $\Be$ is a \emph{sign-coherent basis} of the Macdonald representation
     $\macd{\Phi}{nA_1}$ in the sense of cluster algebras (\cite[Definition 2.2
     (i)]{cao19}, \cite[Definition 6.12]{fomin07}); that is, with respect to
     $\Be$, every element of $W$ acts on  $\macd{\Phi}{nA_1}$ by a matrix where
     the entries in each column all have the same sign. It would be interesting
     to know whether these entries, i.e., the coefficients appearing in the expansion
     of arbitrary $n$-roots into the noncrossing elements, have any
 interpretation in terms of categorification.}
\item[(ii)]{There are other constructions of the basis of noncrossing $n$-roots.  For
        example, one can modify the monomial bases of \cite{fan97} by
        specializing the parameter to $1$ and twisting by the sign
        representation, where the monomial basis in turn agrees with a suitable
        IC basis in the sense of Du \cite{Du} by a result of the first named
        author and Losonczy \cite[Theorem 3.6]{gl99}.
        However, the $n$-root approach has the significant advantage that it is
        relatively easy, given an arbitrary group element $w$ and an arbitrary
        $n$-root $\al$, to express $w(\al)$ as a linear combination of basis
        elements. The bases in type $D_n$ may be constructed diagrammatically in
        terms of perfect matchings, as we explain at the end of Section
        \ref{sec:D}. There is also a diagrammatic construction in types $E_7$
        and $E_8$, as described in \cite{tomdieck97} and \cite{green09}, but we
        do not pursue this here because it is not easy to recover the components
    of a basis $n$-root by inspection of the corresponding diagram.}
 \item[(iii)]{With some more work, it can be shown that every component of a
     noncrossing $n$-root has odd height, and conversely, that every root of odd
 height occurs as a component of some noncrossing $n$-root.}
\end{enumerate}
\end{rmk}

\subsection{Properties of the nonnesting basis}
\label{sec:nnbasis}
In this subsection, we show that the nonnesting basis is naturally indexed by a
distributive lattice whose unique maximal and minimal elements are given by the
maximally crossing and aligned $n$-roots $\theta_C$ and $\theta_A$,
respectively. This lattice is induced by the left weak Bruhat order $\le_L$ of
$W$ and is isomorphic to a lattice consisting of certain fully commutative
elements. We recall that $\le_L$ is defined by the condition that $v\le_L w$ if
$w=uv$ for some $u\in W$ such that $\ell(w)=\ell(u)+\ell(v)$, or,  equivalently,
by the condition that $\ell(wv\inverse)+\ell(v)=\ell(w)$. An element $w$ in a
simply laced Weyl group is fully commutative precisely when no reduced word for
$w$ contains a factor of the form $s_i s_j s_i$ \cite{FC}.

\begin{defn}\label{def:nestchain}
Let $\theta,\theta'$ be two nonnesting positive $n$-roots.  A {\it nonnesting
sequence from $\theta$ to $\theta'$} is a (possibly trivial) sequence
$(\theta_i) = (\theta_0=\theta, \theta_1, \ldots, \theta_r=\theta')$ of positive
nonnesting $n$-roots such that for all $1 \leq j \leq r$ there exists a simple
root $\al_{i_j}$ such that
\begin{equation}
    \label{eq:link} s_{i_j}(\theta_{j-1})=\theta_j\;\text{ and
          }\;\sigma(\theta_j)>\sigma(\theta_{j-1}).
      \end{equation}  If $s_{i_1}, s_{i_2}, \cdots, s_{i_r}$ are simple
    reflections satisfying the condition in \eqref{eq:link}, we say that $\bfw =
    s_{i_1} s_{i_2} \cdots s_{i_r}$ is a \emph{$(\theta,\theta')$-word}, and we
    call the element $w$ expressed by $\bfw$ a
    \emph{$(\theta,\theta')$-element}. Note that we have $w(\theta')=\theta$.
\end{defn}

\begin{rmk}
    \label{rmk:link}
Let $\theta$ be a nonnesting positive $n$-root of type $A^pC^q$ and let $\al_i$
be a simple root. The condition that $\sigma(s_i(\theta))>\sigma(\theta)$ is
equivalent to the condition that $B(\sigma(\theta),\al_i)<0$ by Equation
\eqref{eq:ref}, because $\sigma(s_i(\theta))=s_i(\sigma(\theta))$. In addition,
by Corollary \ref{cor:moveup} (ii) and (iii), the condition that
$\sigma(s_i(\theta))>\sigma(\theta)$ is also equivalent to the condition that
$\sigma(s_i(\theta))=\sigma(\theta)+2\al_i$, or the condition that $s_i$ moves
$\theta$ to an $n$-root of type $A^{p-1}C^{q+1}$.  It follows that if
$(\theta_0,\theta_1,\cdots,\theta_r)$ is a nonnesting sequence, then we have
$\lambda(\theta_j)=\lambda(\theta_{j-1})+1$ for all $1\le j\le r$. In
particular, every nonnesting sequence is a saturated chain with respect to the
quasiparabolic order $\le_Q$.
\end{rmk}

\begin{rmk}
    \label{rmk:rigid}
Let $\theta$ and $\theta'$ be two positive $n$-roots with
$\lambda(\theta')>\lambda(\theta)$, and let $w\in W$ be an element such that
$w(\theta')=\theta$. Let $\bfw=s_{i_1}\cdots s_{i_r}$ be an arbitrary word for
$w$. By the definition of quasiparabolic sets, applying a simple reflection
decreases the level by at most 1, so any element taking $\theta'$ to $\theta$
has length at least $\lambda(\theta')-\lambda(\theta)$. It follows that $r\ge
\lambda(\theta')-\lambda(\theta)$. It also follows that if
$r=\lambda(\theta')-\lambda(\theta)$, then $\bfw$ is reduced and successively
applying the simple reflections $s_{i_r}, \cdots, s_{i_2}, s_{i_1}$ starting
from $\theta'$ must reduce the level by 1 at each step. In particular, if
$r=\lambda(\theta')-\lambda(\theta)$ and $\theta'$ is nonnesting, then it
follows from Proposition \ref{prop:miracle} (i) that each of these simple
reflections is a $CA$ move, so that conversely the sequence
$\theta_\bfw:=(\theta_0,\cdots, \theta_r)$ defined by $\theta_0=\theta,
\theta_j=s_{i_j}(\theta_{j-1})$ for $1\le j\le r$ must be a nonnesting sequence
by Remark \ref{rmk:link}.
\end{rmk}

\begin{prop}\label{prop:nestchain}
Let $W_I\subset W$ be the
parabolic subgroup of Proposition \ref{prop:neststab} (i), and let $\theta$ be a
nonnesting positive $n$-root of type $A^pC^q$.
\begin{enumerate}
    \item[{\rm (i)}] If $\theta$ is the maximally crossing element $\theta_C$,
             then  we have $B(\sigma(\theta), \al_i)\ge 0$ for every simple root
             $\al_i$. Otherwise, there exists a simple root $\al_i$ such that
             $B(\sigma(\theta), \al_i)<0$.
         \item[{\rm (ii)}] There exists a nonnesting sequence from $\theta$ to
          $\theta_C$, and we have
          $\rootht(\sigma(\theta))=\rootht(\sigma(\theta_C))-2p$.

      \item[{\rm (iii)}]{Every $(\theta,\theta_C)$-word is reduced. Every
      $(\theta,\theta_C)$-element is fully commutative and has length $p$. Every
  shortest element taking $\theta_C$ to $\theta$ has length $p$.}

\item[{\rm (iv)}]{There is a unique
$(\theta,\theta_C)$-element $w$. It is
        the unique shortest element in the coset $wW_I$, and is also
    the unique shortest element in $W$ taking $\theta_C$ to $\theta$.}


\item[{\rm (v)}] There exists a nonnesting sequence from the maximally aligned
        element $\theta_A$ to $\theta_C$ that includes $\theta$.
\end{enumerate}
\end{prop}

\begin{proof}
The first assertion of (i) follows from Proposition \ref{prop:ancomp} (iii) and
the fact that $\sigma(\theta_C)=\sigma(\theta_N)$ by Proposition
\ref{prop:neststab}.  Let $V$ be the reflection representation of $W$ and let $$
D = \{v \in V : B(v, \al_i) \ge 0 \text{\ for\ all\ simple\ roots\ } \al_i\}. $$
The set $D$ is a fundamental domain for the action of $W$ on $V$ by
\cite[Theorem 1.12 (a)]{humphreys90}, and we have $\theta_C\in D$. If $\theta$
is a nonnesting $n$-root different from $\theta_C$, then $\theta$ and $\theta_C$
are conjugate under the action of $W$ by Lemma \ref{lem:transitive}, and therefore
so are $\sigma(\theta)$ and $\sigma(\theta_C)$. It follows that $\theta\not\in
D$; therefore we have $B(\sigma(\theta),\al_i)<0$ for some simple root $\al_i$.

We prove (ii) by induction on $p$. If $p=0$, then $\theta=\theta_C$ by
Proposition \ref{prop:neststab} (iii) and the conclusion of (ii) holds
trivially. If $p>0$, then $\theta\neq \theta_C$ and there exists a simple root
$\al_i$ with $B(\sigma(\theta),\al_i)<0$ by (i). The simple reflection $s_i$
satisfies the condition \eqref{eq:link}, adds 2 to the height of the sum, and
sends $\theta$ to an $n$-root of type $A^{p-1}C^{q+1}$ by Remark \ref{rmk:link},
so (ii) follows by induction.

Let $\bfw=s_{i_1}\cdots s_{i_r}$ be a $(\theta,\theta_C)$-word expressing a
$(\theta,\theta_C)$ element $w$. Then $w$ takes $\theta_C$ to $\theta$, and we
have $r=\lambda(\theta_C)-\lambda(\theta)=p$ by by Remark \ref{rmk:link}. Remark
\ref{rmk:rigid} then implies that $\bfw$ is reduced.
It follows that $\ell(w)=r=p$. Since an element taking $\theta_C$ to $\theta$
has length at least $\lambda(\theta_C)-\lambda(\theta)=p$ by Remark
\ref{rmk:rigid} and $w$ is such a shortest element, it also follows that every
shortest element taking $\theta_C$ to $\theta$ has length $p$.

To prove (iii), it remains to show that $\bfw$ cannot contain a factor of the
form $s_{\al_i} s_{\al_j} s_{\al_i}$. By Remark \ref{rmk:link} and direct
computation, such a factor would imply the existence of a subsequence
$(\theta_a, \theta_{a+1}, \theta_{a+2}, \theta_{a+3})$ such that \[
s_{\al_i}(\theta_{a+2}) = s_{\al_i}(\theta_a + 2\al_i + 2\al_j) = \theta_a +
2\al_i + 2\al_j = \theta_{a+2}.  \] This contradicts the fact that
$s_{\al_i}(\theta_{a+2})=\theta_{a+3}$, which completes the proof of (iii).

Let $w$ be a shortest element taking $\theta_C$ to $\theta$. Then $w$ has length
$p=\lambda(\theta_C)-\lambda(\theta)$ by (iii). Let $\bfw=s_{i_1}\dots s_{i_p}$
be a reduced word for $w$. Remark \ref{rmk:rigid} implies that if we start from
$\theta_C$ and apply $s_{i_p}, \cdots, s_{i_2}, s_{i_1}$ successively, each
simple reflection must be a $CA$ move.  In particular, $s_{i_p}$ must perform a
$CA$ move on $\theta_C$, so we have
\[s_{i_p}(\sigma(\theta_C))=\sigma(s_{i_p}(\theta_C))=\sigma(\theta_C)-2\al_{i_p}\]
by Corollary \ref{cor:moveup} (ii). This implies that
$B(\sigma(\theta_N),\al_{i_p})=B(\sigma(\theta_C), \al_{i_p})>0$, so it follows
from Proposition \ref{prop:ancomp} (iii) that $\al_{i_p}=\al_x$ where $\al_x$ is
the unique Coxeter generator of $W$ not in $I$. In other words, every reduced
word for $w$ ends in $s_{\al_x}$.  It follows from \cite[Proposition 1.10
(c)]{humphreys90} that $w$ is the unique shortest element in $wW_I$.

If $w'$ is another shortest element taking $\theta_C$ to $\theta$, then
$w'w\inverse(\theta_C)=\theta_C$, so $w'w\inverse\in W_I$ by Proposition
\ref{prop:ancomp} (iv). It follows that the cosets $wW_I$ and $w'W_I$ are equal; therefore we have $w=w'$, because $w$ and $w'$ are both the unique shortest
element in the common coset $wW_I=w'W_I$ by the last paragraph. This proves the
uniqueness of the shortest element taking $\theta_C$ to $\theta$. Part (iii)
says that each $(\theta,\theta_C)$-element is such a shortest element, and (iv)
now follows.

Finally, to prove (v), we recall from Proposition \ref{prop:minmax} (i) that
$\theta_A$ is the unique minimal element of the set $X$. It follows from
\cite[Theorem 2.8]{rains13} that there exists an element $u\in W$ such that
$u(\theta_A)=\theta$ and $\ell(u)=\lambda(\theta)-\lambda(\theta_A)=q$. Let $v$
be the unique $(\theta,\theta_C)$-element, let $s_{i_1}\cdots s_{i_p}$ be a
reduced word for $v$, and let $s'_{i_1}\cdots s'_{i_q}$ be a reduced word for
$u\inverse$. Then $u\inverse v$ takes $\theta_C$ to $\theta_A$ and has length at
most $p+q=M=\lambda(\theta_C)-\lambda(\theta_A)$, so $u\inverse v$ must be the
unique $(\theta_A,\theta_C)$-element by (ii) and (iii), and the word
$\bfw=s'_{i_1}\cdots s'_{i_q}s_{i_1}\cdots s_{i_p}$ must be a reduced word for
$u\inverse v$. Remark \ref{rmk:rigid} now implies that starting from $\theta_A$
and applying $s'_{i_1}, \cdots, s'_{i_q}, s_{i_1},\cdots, s_{i_p}$ successively
yields a nonnesting sequence $\theta_\bfw$ from $\theta_A$ to $\theta_C$ that
reaches $\theta=v\theta_C$ after the first $q$ steps, and (v) follows.
\end{proof}

\begin{theorem}\label{thm:nonnest}
Let $W$ be a Weyl group of rank $n$ of types $E_7$, $E_8$, or $D_n$ for $n$
even. Let $M$ be the number of \coplanar quadruples in a maximal orthogonal
set, and let $W_I$ be the parabolic subgroup of Proposition \ref{prop:neststab}
(i).
\begin{enumerate}
    \item[{\rm (i)}]{There is a unique element $w_N \in W$ of minimal length
     such that $w_N(\theta_C) = \theta_A$.  The element $w_N$ is fully
 commutative and has length $\ell(w_N) = M$, and is the unique element of
 minimal length in the coset $w_N W_I$.}
\item[{\rm (ii)}]{The set $$ L= \{v(\theta_C) : v \leq_L w_N\} $$ is a complete,
     irredundantly described set of nonnesting positive $n$-roots. The set $L$
 has the structure of a distributive lattice, induced by the weak Bruhat order
 $\leq_L$.}
\item[{\rm (iii)}]{If $\gamma_1$ and $\gamma_2$ are positive $n$-roots
     satisfying $\rootht(\sigma(\gamma_1)) - \rootht(\sigma(\gamma_2)) = 2M$,
     and $w$ is an element expressed by a word $\bfw$ of length $M$ satisfying
     $w(\gamma_1) = \pm\gamma_2$, then $\bfw$ is reduced, and we must have
     $\gamma_1 = \theta_C$, $\gamma_2 = \theta_A$, $w = w_N$, and $w(\gamma_1) =
 \gamma_2$.}
\end{enumerate}
\end{theorem}

\begin{proof}
Part (i) follows from Proposition \ref{prop:nestchain} in the case where
$\theta_0 = \theta_A$.

Let $w_N \in W$ be as in (i) and let $v\le_L w_N$. We may complete a reduced
word $\mathbf{v}$ for $v$ to a reduced word of the form $\bfw=\mathbf{u}\cdot
\mathbf{v}$ for $w_N$. Remark \ref{rmk:rigid} implies that $\bfw$ gives rise to
a nonnesting sequence $\theta_\bfw$ from $\theta_A$ to $\theta_C$ that passes
$v(w_N)$. It follows that the elements of $L$ are indeed all nonnesting positive
$n$-roots.  Conversely, for every nonnesting positive root $\theta$, it follows
from Proposition \ref{prop:nestchain} (v) and its proof that
$\theta=v(\theta_C)$ for some element $v\le_L w_N$, so the list $L$ is complete.
Finally, if $v\le_L w_N$ and $v'\le_L w_N$ are elements such that
$v(\theta_C)=v'(\theta_C)$, then $v'v\inverse$ stabilizes $\theta_C$ and hence
$\sigma(\theta_C)$, so we have $v'v\in W_I$ and $v'W_I=vW_I$. Since $w_N$ is the
shortest element in $w_NW_I$, the elements $v'$ and $v$ must be the shortest
elements in $vW_I$ and $v'W_I$ as well, which implies that $v=v'$ as well. It
follows that the list $L$ irredundantly describes the positive nonnesting
$n$-roots, proving the first statement of (ii).

By \cite[Theorem 3.2]{FC}, the fact that $w_N$ is fully commutative implies that
the poset $\{x : x \leq_L w\}$ is a distributive lattice. This completes the
proof of (ii).

Suppose that the conditions of (iii) hold, and let $s_{i_1} s_{i_2} \cdots
s_{i_M}$ be a reduced expression for $w$. Since
$\rootht(\sigma(\gamma_1))-\rootht(\sigma(\gamma_2))=2M$, as we start from
$\theta_C$ and successively apply the simple reflection $s_{i_M}, \dots,
s_{i_2}, s_{i_1}$, the application of each simple reflection $s_{i_j}$ must
subtract $2$ from the height of the sum and change a $C$ to an $A$ by Corollary
\ref{cor:moveup} (ii). It is therefore not possible at any stage for a simple
reflection to negate a component of an $n$-root, which implies that we have
$w(\gamma_1)=\gamma_2$. Each simple reflection $s_{i_j}$ also causes no change
in the number of nestings by Proposition \ref{prop:miracle} (i), so the fact
that $\rootht(\sigma(\gamma_1))-\rootht(\sigma(\gamma_2))=2M$ implies that
$\gamma_1$ and $\gamma_2$ have types $C^M$ and $A^M$, respectively, so we have
$\gamma_1 = \theta_C$ and $\gamma_2 = \theta_A$. We then have $w = w_N$ by (i),
which completes the proof of (iii).
\end{proof}

\begin{defn}
    \label{def:wN}
    We call the element $w_N$ from Theorem \ref{thm:nonnest}, i.e., the unique
    element of minimal length that sends $\theta_C$ to $\theta_A$, the
    \emph{nonnesting element} of $W$.
\end{defn}

In Section \ref{sec:examples}, we will compute the nonnesting element
explicitly with the help of Theorem \ref{thm:nonnest} (iii).

\subsection{Sum equivalence}
We say that two positive $n$-roots $\beta$ and $\gamma$ of $W$  are \emph{sum
equivalent}, or {\it $\sigma$-equivalent}, if $\sigma(\beta) = \sigma(\gamma)$.
If $C$ and $C'$ are two $\sigma$-equivalence classes, then we write $C\le_\sigma
C'$ if $\sigma(\beta)\le \sigma(\gamma)$ for any $\beta\in C$ and $\gamma\in
C'$ in the usual order $\le$ on roots (Section \ref{sec:basics2}). The goal of this subsection is to show that the $\sigma$-equivalence
classes of $X$ are highly compatible with the quasiparabolic order $\le_Q$ and
the feature-avoiding $n$-roots.

\begin{prop}\label{prop:sumclass}
Let $\Be$ be the set of nonnesting positive $n$-roots of $W$.
\begin{enumerate}
    \item[{\rm (i)}]{If $\be, \be'\in \Be$ are nonnesting positive $n$-roots
             with $\sigma(\be) = \sigma(\be')$, then we have $\be = \be'$.}
         \item[{\rm (ii)}]{Each positive $n$-root $\gamma$ is
          $\sigma$-equivalent to a unique nonnesting $n$-root $f(\gamma)$ and a
          unique noncrossing $n$-root $g(\gamma)$. We have $f(\gamma)\le_Q
          \gamma$, and
          \begin{equation*} \gamma=f(\gamma)+\sum_{\beta\in \Be:
           \sigma(\beta)<\sigma(\gamma)} \lambda_{\beta,\gamma}\beta
   \end{equation*} for suitable integers $\lambda_{\be,\gamma}$.  }
  \item[{\rm (iii)}] Every $\sigma$-equivalence class contains a unique
       nonnesting $n$-root, $\be_1$, and a unique noncrossing $n$-root, $\be_2$.
       The $\sigma$-equivalence class containing $\be_1$ and $\be_2$ is equal to
       the
       interval $$ [\be_1, \be_2]_{Q} = \{ \gamma \in X : \be_1 \leq_Q \gamma \leq_Q
       \beta_2 \} $$ in the quasiparabolic set $X$.
   \item[{\rm (iv)}] The set of alignment-free positive $n$-roots is a
     $\sigma$-equivalence class and is equal to the interval
     $[\theta_C,\theta_N]_Q$ in the quasiparabolic set $X$. It is the unique
     maximal
     $\sigma$-equivalence class with respect to the partial order $\le_\sigma$.
\end{enumerate}
\end{prop}

\begin{proof}
Suppose that $\be$ and $\be'$ are nonnesting positive $n$-roots with
$\sigma(\be') = \sigma(\be)$. It follows from Proposition \ref{prop:nestchain}
(ii) that $\be$ and $\be'$ have the same number of alignments, namely, the
number $p=(\rootht(\sigma(\theta_C))-\rootht(\sigma(\beta)))/2$. If $p=0$, then
we have $\be=\theta_C=\be'$ by Proposition \ref{prop:neststab} (iii). If $p>0$,
then neither $\be$ nor $\be'$ equals $\theta_C$, so there is a simple root
$\al_i$ satisfying $B(\sigma(\be),\al_i)=B(\sigma(\be'),\al_i)<0$ by Proposition
\ref{prop:nestchain} (i). By Remark \ref{rmk:link}, both $s_i(\be)$ and
$s_i(\be')$ are nonnesting positive $n$-roots with $p-1$ alignments, and we have
\[
    \sigma(s_i(\beta))=s_i(\sigma(\beta))=\sigma(\beta)+2\al_i=\sigma(\beta')+2\al_i=s_i(\sigma(\be'))=\sigma(s_i(\be')),
\] 
so (i) follows by induction on $p$.

Let $\gamma$ be a positive $n$-root, and let $\le_N$ be the nesting order
defined in Section \ref{sec:twobases}.
If $\gamma$ contains no nesting, we can simply take $f(\gamma)=\gamma$.
Otherwise, we can factorize $\gamma=\gamma'\gamma_N$ where $\gamma_N$ is a
nesting. By the second paragraph of Section \ref{sec:twobases}, we can write
$\gamma=\gamma'\gamma_C-\gamma'\gamma_A$ where we have (a) $\gamma'\gamma_C\le
_N\gamma$, because $\sigma(\gamma'\gamma_C)=\sigma(\gamma)$ and $
\gamma'\gamma_C<_Q \gamma$, and (b) $\gamma'\gamma_A\le _N\gamma$, because
$\sigma(\gamma'\gamma_A)<\sigma(\gamma)$. Taking $f(\gamma)=f(\gamma'\gamma_C)$
proves the existence of $f(\gamma)$ and the required expression for $\gamma$ by
induction on the order $\le_N$.  The uniqueness of $f(\gamma)$ follows from (i).
We can use a similar induction using the crossing order $\le_C$ to show that
$\gamma$ is $\sigma$-equivalent to a noncrossing $n$-root $g(\gamma)$ such that
$\gamma \leq_Q g(\gamma)$, and this completes the proof of (ii).

It follows from (ii) that every $\sigma$-equivalence class contains a unique
nonnesting $n$-root, and that the number of $\sigma$-equivalence classes equals
the number of nonnesting $n$-roots. The latter number is the dimension of the
Macdonald representation and also the number of noncrossing roots by Theorem
\ref{thm:indep} (i) and (ii).  Since each $\sigma$-equivalence class contains at
least one noncrossing $n$-root by (ii), it follows that each
$\sigma$-equivalence class must contain exactly one nonnesting element, and
exactly one noncrossing element. This proves the first sentence of (iii).

Let $E$ be a $\sigma$-equivalence class with unique nonnesting element $\beta_1$
and unique noncrossing element $\beta_2$. Then we have $[\be_1,\be_2]_Q\se E$ by
Corollary \ref{cor:moveup} (i). Conversely, if $\gamma$ is an $n$-root in $E$,
then (ii) and its proof imply that we may find a nonnesting $n$-root
$f(\gamma)\in E$ and a noncrossing $n$-root $g(\gamma)\in E$ such that
$f(\gamma)\le_Q \gamma\le_Q g(\gamma)$. We must have $f(\gamma)=\beta_1$ and
$g(\gamma_2)=\beta_2$ by the uniqueness of the nonnesting and noncrossing elements
in $E$; therefore we have $\gamma\in [\be_1,\be_2]_Q$. It follows that
$E=[\be_1,\be_2]_Q$.

For every nonnesting root $\beta$ not equal to $\theta_C$, there is a
nontrivial nonnesting sequence from $\beta$ to $\theta_C$ by by Proposition
\ref{prop:nestchain} (ii), so $\sigma(\beta)<\sigma(\theta_C)$ by Definition
\ref{def:nestchain}.  Part (iv) now follows from (iii) and Proposition
\ref{prop:neststab} (ii)--(iii).
\end{proof}

\begin{theorem}\label{thm:nesttri}
Let $W$ be a Weyl group of type $E_7$, $E_8$, or $D_n$ for $n$ even. If $\Be$ is
any set of $\sigma$-equivalence class representatives, then $\Be$ is a basis for
the Macdonald representation $\macd{\Phi}{nA_1}$. 
Furthermore, if we order each such basis $\Be=\{\beta_1,\cdots,\beta_k\}$ in a
way compatible with the order $\le_\sigma$, i.e., in such a way that $i<j$
whenever $\beta_i<_\sigma \beta_j$, then the change of basis matrix
between any two such bases is unitriangular with integer entries. In particular,
this is true for the change of basis matrix between the nonnesting basis and the
noncrossing basis.
\end{theorem}

\begin{proof}
By Proposition \ref{prop:sumclass} (ii), each element $\gamma \in \Be$ is the
sum of the nonnesting element $\gamma'$ that is $\sigma$-equivalent to $\gamma$
and a $\Z$-linear combination of nonnesting elements with strictly lower sums.
The nonnesting elements form a basis for $\macd{\Phi}{nA_1}$ by Theorem
\ref{thm:indep} (i), from which it follows that the set $\Be$ is also a basis,
and that the change of basis from $\Be$ to the nonnesting basis is unitriangular
with integer entries. If $\Be_1$ and $\Be_2$ are two such bases, then the change
of basis matrix from $\Be_1$ to $\Be_2$ is unitriangular with integer entries
because it is the product of the matrix changing $\Be_1$ to $\Be'$ with the
inverse of the matrix changing $\Be_2$ to $\Be'$, both of which are
unitriangular with integer entries.  Finally, the last assertion follows because
Proposition \ref{prop:sumclass} (iii) implies that both the nonnesting and
noncrossing bases are sets of $\sigma$-equivalence class representatives.
\end{proof}

\begin{rmk}\label{rmk:shell6}
Recall that the {\it M\"obius function}, $\mu$, of a partially ordered set $P$
is defined to satisfy $\mu(x, x)=1$, $\mu(x, y)=0$ if $x \not\leq y$, and $$
\sum_{z : x \leq z \leq y} \mu(x, z) = 0 $$ if $x < y$. A poset is {\it
Eulerian} if we have $\mu(x, y) = (-1)^{\lambda(y) - \lambda(x)}$ whenever $x
\leq y$.  It can be shown that if $x, y \in X$ correspond to a nonnesting and
noncrossing element, respectively, then the interval $I = [x, y]$ corresponds to
a $\sigma$-equivalence class if and only if $I$ is Eulerian.
\end{rmk}

\begin{rmk}\label{rmk:posetcong}
Reading \cite{reading02} defines a {\it poset congruence} to be an equivalence
relation on a poset $X$ such that
\begin{enumerate}
    \item[(i)]{each equivalence class is an interval;}
    \item[(ii)]{the projection mapping $x \in X$ to the maximal element in its
           equivalence class is order preserving; and}
       \item[(iii)]{the projection mapping $x \in X$ to the minimal element in
        its equivalence class is order preserving.}
\end{enumerate} It can be shown using \cite[Proposition 42]{watson14} that, in
    type $D_n$, the equivalence relation induced on $X$ by $\sigma$ is a poset
    congruence.  It can also be shown (by direct computational verification, for
    example) that the same is true in types $E_7$ and $E_8$.
\end{rmk}

\section{Examples}\label{sec:examples}
In this section, we give type-specific details about the $n$-roots in types
$D_n$ for $n$ even, $E_7$, and $E_8$.  In all types, we explicitly describe the
maximally aligned, crossing, and nesting $n$-roots $\theta_A, \theta_C$ and
$\theta_N$. We find the nonnesting element $w_N$ (Definition \ref{def:wN}), and
we use $w_N$ and Theorem
\ref{thm:nonnest} (ii) to deduce the dimension of the Macdonald representation
$\macd{\Phi}{nA_1}$. We also discuss type-specific properties of the set $X_I$
of alignment-free positive $n$-roots for all types. In addition, we explain
precise connections between the Macdonald representation $\macd{\Phi}{nA_1}$ of
type $D_{2k}$ and a Specht module of the symmetric group $S_{2k}$ (Proposition
\ref{prop:specht}),

We note that by Lemma \ref{lem:dmax} and Remark \ref{rmk:cna}, the noncrossing
and nonnesting positive $n$-roots of type $D_{2k}$ can be easily recovered from
the well-studied noncrossing and nonnesting perfect matchings of $[2k]$. More
generally, in all types, the nonnesting positive $n$-roots can be computed
efficiently via the elements the elements $\theta_C$ and $w_N$ by Theorem
\ref{thm:nonnest} (ii), and it is possible to construct the noncrossing
$n$-roots using Fan's construction of monomial cells in \cite{fan97}.  In the
notation of \cite{fan97}, the maximally aligned $n$-root $\theta_A$ can be
identified with the element $b_1 b_3 \cdots b_{2k-1}$ in type $D_{2k}$, with
$b_2 b_4 b_6 b_7$ in type $E_7$ (with the labelling of Figure \ref{fig:ade}
(d)), and with $b_2 b_3 b_5 b_7$ in type $E_8$ (with the labelling of Figure
\ref{fig:ade} (e)). In types $E_7$ and $E_8$, it is also possible to use a
computer program to find all noncrossing and nonnesting $n$-roots by generating
all the (finitely many) positive $n$-roots and then removing all $n$-roots where
a crossing or nesting can be found.  For these reasons, and to save space, we
have chosen not to list the noncrossing and nonnesting bases in type $E_7$ or
$E_8$ in this paper (although the complete lists are available upon request).

\subsection{Type \texorpdfstring{$D_{2k}$}{D}}
\label{sec:D}
If $W$ has type $D_n$ for an even integer $n=2k$, then the positive $n$-roots
can be naturally identified with the perfect matchings of $[n]$, as explained in
Lemma \ref{lem:dmax} (ii). Under this identification, the actions of $W$ on the
$n$-roots and on the matching agree, and the alignments, crossings, and nestings
in the $n$-roots correspond to the alignments, crossings, and nestings in the
matchings in the obvious way by Remark \ref{rmk:cna}.  We also recall from
Section \ref{sec:constructions} and Remark \ref{rmk:Daction} that the reflection
$r=s_\al\in W$ acts as the transposition $(ij)$ on the $n$-roots for each root
$\al=\ep_i\pm\ep_j$ of $W$, so that the action of $W$ factors through the
homomorphism $\phi: W\ra S_{2k}$ of Equation \eqref{eq:phimap} to induce an
action of $S_{2k}=W(A_{2k-1})$ on the $n$-roots, giving the Macdonald
representation $\macd{\Phi}{nA_1}$ the structure of an $S_{2k}$-module (where
the elements of $S_{2k}$ permute the indices of the terms $\ep_i^2$).  The above
facts will allow us to connect the theory of $n$-roots in type $D_n$ to some
widely studied type-$A$ objects and results.

Recall that the number of coplanar quadruples in each positive $n$-root is
$M=\binom{k}{2}$, the number of pairs of 2-blocks, by Corollary
\ref{cor:mcount}.

Let $\nu_A, \nu_C, $ and $\nu_N$ be the positive $n$-roots corresponding to the
matchings $\{12,34,\cdots, (n-1)n\}, \{1(k+1), 2(k+2), \cdots, k(2k)\}$ and
$\{1n, 2(n-1), \cdots, k(k+1)\}$, respectively. Every pair of 2-blocks in the
first matching forms an alignment, so the matching contains $\binom{k}{2}=M$
alignments.  It then follows from  Proposition \ref{prop:minmax} (i) that
$\nu_A$ is the unique maximally aligned $n$-root $\theta_A$ in the set $X$.
Similar arguments show that $\nu_C=\theta_C$ and $\nu_N=\theta_N$ by Proposition
\ref{prop:neststab} (iii) and Proposition \ref{prop:minmax} (ii), respectively.
Note that we have $\sigma(\theta_N)=2\sum_{i = 1}^k \ep_i$.

Let $w$ be the element expressed by the word $$ \bfw= \bfw_{2, k-2} \bfw_{3,
k-3} \cdots \bfw_{k, 0} ,$$ where $\bfw_{i, j} := s_i s_{i+2} s_{i+4} \cdots
s_{i+2j}$. For example, in type $D_8$, we have $w = (s_2 s_4 s_6)(s_3
s_5)(s_4)$, and the heap of $w$ is shown in Figure \ref{fig:heaps} (a).  The
word $\bfw$ has $M$ letters, and it is straightforward to verify that
$w(\theta_C)=\theta_A$, so it follows from Theorem \ref{thm:nonnest} (i) that
$w$ is the fully commutative nonnesting element $w_N$ and $\bfw$ is a reduced
word for it.

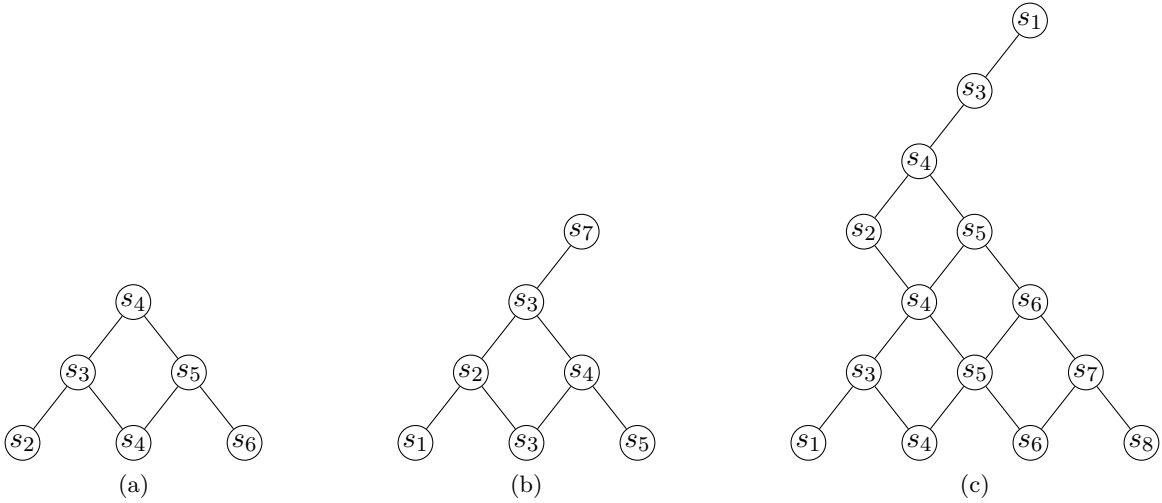
\begin{figure}[h!]
    \centering
    \subfloat[]{
\begin{tikzpicture}
\node[main node] (0) {$s_2$};
\node[main node] (1) [right=1cm of 0] {$s_4$};
\node[main node] (2) [right=1cm of 1] {$s_6$};
\node[main node] (3) [above right=0.6cm and 0.4cm of 0] {$s_3$};
\node[main node] (4) [above right=0.6cm and 0.4cm of 1] {$s_5$};
\node[main node] (5) [above right=0.6cm and 0.4cm of 3] {$s_4$};

\path[draw]
(0)--(3)--(5)--(4)--(2)
(4)--(1)--(3);
\end{tikzpicture}
    }\quad\quad\quad\quad
    \subfloat[]{
\begin{tikzpicture}
\node[main node] (0) {$s_1$};
\node[main node] (1) [right=1cm of 0] {$s_3$};
\node[main node] (2) [right=1cm of 1] {$s_5$};
\node[main node] (3) [above right=0.6cm and 0.4cm of 0] {$s_2$};
\node[main node] (4) [above right=0.6cm and 0.4cm of 1] {$s_4$};
\node[main node] (5) [above right=0.6cm and 0.4cm of 3] {$s_3$};
\node[main node] (6) [above right=0.6cm and 0.4cm of 5] {$s_7$};

\path[draw]
(0)--(3)--(5)--(4)--(2)
(4)--(1)--(3)
(5)--(6);

\end{tikzpicture}
    }\quad\quad\quad\quad
    \subfloat[]{
\begin{tikzpicture}
\node[main node] (0) {$s_1$};
\node[main node] (1) [right=1cm of 0] {$s_4$};
\node[main node] (2) [right=1cm of 1] {$s_6$};
\node[main node] (3) [right=1cm of 2] {$s_8$};
\node[main node] (4) [above right=0.6cm and 0.4cm of 0] {$s_3$};
\node[main node] (5) [above right=0.6cm and 0.4cm of 1] {$s_5$};
\node[main node] (6) [above right=0.6cm and 0.4cm of 2] {$s_7$};
\node[main node] (7) [above right=0.6cm and 0.4cm of 4] {$s_4$};
\node[main node] (8) [above right=0.6cm and 0.4cm of 5] {$s_6$};
\node[main node] (9) [above right=0.6cm and 0.4cm of 7] {$s_5$};
\node[main node] (10) [above left=0.6cm and 0.4cm of 7] {$s_2$};
\node[main node] (11) [above right=0.6cm and 0.4cm of 10] {$s_4$};
\node[main node] (12) [above right=0.6cm and 0.4cm of 11] {$s_3$};
\node[main node] (13) [above right=0.6cm and 0.4cm of 12] {$s_1$};

\path[draw]
(0)--(4)--(1)--(5)--(2)--(6)--(3)
(4)--(7)--(5)--(8)--(6)
(8)--(9)--(7)--(10)--(11)--(12)--(13)
(9)--(11);
\end{tikzpicture}
    }
    \caption{The heaps of the nonnesting elements of types $D_8$, $E_7$, and $E_8$}
    \label{fig:heaps}
\end{figure}

Since $w_N=w$ is fully commutative, the elements in the set $\{v\in W:v\le_L
w_N\}$ are in bijection with the order filters of the heap poset of $w_N$. (See
    \cite[Section 2.2]{FC} for the definition of the heap poset; an order filter
    of a poset $P$ is a subset of $P$ such that $y \in I$ whenever the
    conditions $y \in P$, $x \in I$, and $x \leq y$ hold.) These filters are in
    canonical correspondence with Dyck paths of order $k$, i.e., staircase walks
    from $(0, 0)$ to $(k, k)$ that lie strictly below (but may touch) the
    diagonal $y=x$. It is well known \cite[Theorem 1.5.1 (vi)]{stanley15} that
    the number of such paths is the $k$-th Catalan number,
    $C_k=\frac{1}{k+1}\binom{2k}{k}$. Theorem \ref{thm:nonnest} (ii) and Theorem
    \ref{thm:indep} imply that the number of nonnesting positive $n$-roots of
    $W$ is given by $C_k$, as are the number of noncrossing positive $n$-roots
    and the dimension of the Macdonald representation $\macd{\Phi}{nA_1}$.

The level function $\lambda$ in type $D_{2k}$ has a combinatorial interpretation
that is natural in the context of combinatorial game theory \cite{irie21}.  The
matching corresponding to an $n$-root $\be$ can be identified with a Steiner
system $S(1,2,2k)$, i.e., a collection of 2-blocks of $[2k]$ with the property
that any singleton lies in a unique 2-block. The level $\lambda(\be)$ then
counts the number of 2-element subsets $E$ of $[2k]$ with the property that the
matching corresponding to $\beta$ contains no 2-blocks of the form
$(E\setminus\{j\})\cup \{i\}$ where $i\le j$ and $j\in E$ (in particular, the
matching cannot contain $E$). With some more work, it can be shown that each
crossing gives rise to one such subset $E$, and each nesting gives rise to two
such subsets. This gives a combinatorial interpretation of the formula
$\lambda(x) = C(x)+2N(x)$, and also explains the appearance of the product of
odd quantum integers in \cite[Equation (4.2)]{irie21}.  In addition, the
quantity $C(m)+2N(m)$ associated to each matching $m$ appears as ``$C(m)+2U(m)$"
in the context of octabasis Laguerre polynomials in \cite[Sections 4 and
5]{simion96}, as the weight ``$\omega(m)$'' in the context of Gaussian
$q$-distributions in \cite[Theorem 4]{diaz09}, and as
``$\emph{cov}(m)-\emph{cro}(m)$'' in the
context of $q$-Bessel numbers in \cite[Section 4]{cheon15}.

The poset structure on the set $X_I$ in type $D_{2k}$ coincides with a familiar
one.

\begin{prop}\label{prop:abruhat}
Suppose $W$ has type $D_{2k}$. Then as a poset, the interval $X_I = [\theta_C,
\theta_N]$ in the quasiparabolic set $X$ is canonically isomorphic to the
symmetric group $S_k$ under the (strong) Bruhat order via the map $\varphi: S_k
\ra X_I$ sending each element $\tau\in S_n$ to the $n$-root
\begin{equation}\label{eq:bruhat}\varphi(\tau)=\prod_{i = 1}^k (\ep_i^2 -
      \ep_{\tau(i) + k}^2).
  \end{equation} Under this bijection, we have \[
    \lambda(\varphi(\tau))=M+\ell(\tau) \] for every $\tau\in S_k$, where $M$ is
    the number of coplanar quadruples in each $n$-root and $\ell$ denotes Coxeter
    length.
\end{prop}

\begin{proof}
By Proposition \ref{prop:neststab} (ii), the set $X_I$ is the
$\sigma$-equivalence class of the $n$-root $\theta_N$. We noted earlier that
$\sigma(\theta_N)=2\sum_{i = 1}^k \ep_i$, which implies that $n$-roots in $X_I$
are precisely the positive $n$-roots whose components are all of the form $\ep_i
\pm \ep_j$, where $1 \leq i \leq k < j \leq 2k$. These are  precisely the
$n$-roots listed in the theorem, so the map $\varphi$ is surjective.  It is
clear that $\varphi$ is also injective, so $\varphi$ is a bijection.

The Bruhat order on $S_k$ is generated by relations of the form $\tau< r\tau$
where we have $\tau\in S_k$ and $r$ is a reflection $r=(\tau(i),\tau(j))\in S_k$
for some $i<j$ such that $\tau(i)<\tau(j)$ \cite[Section 5.9, Example
2]{humphreys90}. In this case, the quadruple $\{\ep_i\pm\ep_{\tau(i)+k},
\ep_j\pm\ep_{\tau(j)+k}\}$ contained in $\varphi(\tau)$ is a crossing and is
moved to the nesting $\{\ep_i\pm\ep_{\tau(j)+k}, \ep_j\pm\ep_{\tau(i)+k}\}$, so
we have $\lambda(\varphi(\tau))<\lambda(r\varphi(\tau))$ in $X_I$ by Proposition
\ref{prop:miracle} (iii).  Conversely, if we have
$\lambda(\varphi(\tau))<\lambda(r\varphi(\tau))$ in $X_I$ for some $\tau\in S_k$
and some reflection $r\in W$, then since $\varphi(\tau)$ has no alignments, $r$
must move a crossing in $\varphi(\tau)$ to a nesting by Proposition
\ref{prop:miracle} (iii). The crossing moved must be of the form
$\{\ep_i\pm\ep_{\tau(i)+k}, \ep_j\pm\ep_{\tau(j)+k}\}$ for some $i,j\in [k]$
such that $i<j$ and $\tau(i)<\tau(j)$, and the only possibilities for $r$ are
$(ij)$ and $(\tau(i)+k,\tau(j)+k)$. In either case, we have
$r\varphi(\tau)=\varphi(r'\tau)$ for the reflection $r'=(\tau(i),\tau(j))\in
S_k$, so that we have
\[\varphi\inverse(\varphi(\tau))=\tau<r'\tau=\varphi\inverse(r\varphi(\tau))\]
where $<$ denotes the Bruhat order in $S_k$.  It now follows that $\varphi$ is a
poset isomorphism.

To prove the last assertion, we note that each inversion of a permutation
$\tau\in S_k$ corresponds to a nesting in the corresponding alignment-free
$n$-root $\gamma=\varphi(\tau)$, and we recall that $\ell(\tau)$ equals the number
of inversions in $\tau$. It follows that $N(\gamma)=\ell(\tau)$; therefore we have
\[ \lambda(\gamma)=C(\gamma)+2
N(\gamma)=(C(\gamma)+N(\gamma))+N(\gamma)=M+N(\gamma)=M+\ell(\tau).\qedhere \]
\end{proof}

We now discuss the structure of the space  $\macd{D_n}{nA_1}$ underlying the
Macdonald representation as an $S_{2k}$-module. As a vector space,
$\macd{D_n}{nA_1}$ is isomorphic to the free vector space on the noncrossing
perfect matchings of $[n]=[2k]$, which is denoted by $V(n,k,0)$ in the work of
Rhoades \cite{rhoades17}. Furthermore, given a simple reflection $s_i=(i,i+1)\in
S_{n}$ and a noncrossing perfect matching $m$ corresponding to an $n$-root
$\gamma$, the reflection $s_i$ acts on $m$ in one of the following ways:
\begin{enumerate}
    \item[\rm{(1)}] if $i(i+1)$ is a 2-block in $m$, then the $n$-root $\gamma$ contains
         $\ep_i^2-\ep_{i+1}^2$ as a factor, so $s_i(m)=-m$;
     \item[\rm{(2)}] if $i(i+1)$ if not a 2-block in $m$, then $m$ contains two blocks
         $ia$ and $(i+1)b$ which either form an alignment (if $a<i<i+1<b$) or a
         nesting (if $b<a<i$ or $i+1<b<a$). In all cases, we have $s_i(m)=m''$
         where $m''$ is the matching
         $(m\setminus\{ia,(i+1)b\})\cup\{(i+1)a,ib\}$. Here, the blocks $(i+1)a$
         and $ib$ form a crossing, and  the Ptolemy relation
         $\gamma_C=\gamma_N+\gamma_A$ from Theorem \ref{thm:cna} (vi) implies
         that
         \begin{equation}
             \label{eq:resolve} s_i(m)=m''=m+m'
         \end{equation} where $m'$ is the perfect matching $m'=(m\setminus
     \{ia,(i+1)b\}) \cup \{i(i+1),ab\}$. Here, the blocks $i(i+1)$ and $ab$ form
     the nesting in the Ptolemy relation if $\{ia, (i+1)b\}$ is an alignment and
     form the alignment in the Ptolemy relation if  $\{ia, (i+1)b\}$ is a
     nesting. The matching $m'$ is noncrossing by Theorem \ref{thm:positive}
     (iv).
 \end{enumerate}
 It follows from the above analysis that the action of $S_{2k}$ on
 $\macd{D_n}{nA_1}$ agrees with the action of $S_{2k}$ on the space $V(n,k,0)$
 defined by Rhoades. The precise formula in Equation \eqref{eq:resolve} appears
 in the work of Kim \cite[Equation (1.3)]{kim24}.  By \cite[Proposition
 5.2]{rhoades17}, as an $S_{2k}$ module $V(n,k,0)$ is isomorphic to the Specht
 module $S^{(k,k)}$ corresponding to the 2-row partition $(k,k)$, so we may
 summarize our discussion as follows:

\begin{prop}
\label{prop:specht}
If $W$ has type $D_{n}$ for $n=2k$ even, then the $W$-action on the Macdonald
representation $\macd{D_n}{nA_1}$ factors through the map $\phi$ defined by
Equation \eqref{eq:phimap} to induce an $S_n$-module structure on
$\macd{D_n}{nA_1}$.  The resulting $S_n$-module is isomorphic to the Specht
module $S^{(k,k)}\cong V(n,k,0)$. In particular, it is irreducible.
\end{prop}

\begin{rmk}\label{rmk:webgraph}
The nonnesting and noncrossing bases for the $S_n$-module $\macd{D_n}{nA_1}\cong
S^{(k,k)}$ are also studied extensively in the works of Russell--Tymoczko
\cite{russell19}, Im--Zhu \cite{im22}, Hwang--Jang--Oh \cite{hwang23}, and
Heard--Kujawa \cite{kujawa24}. In these papers, the noncrossing basis is called
the \emph{web basis}, and the nonnesting basis can be naturally identified with
the \emph{standard basis} (or the \emph{polytabloid} or \emph{Specht} basis) as
explained in \cite[Lemma 3.1]{im22} and \cite[Section 1]{hwang23}. Under this
identification, the isomorphism of \cite[Theorem 2.2]{russell19} associates each
nonnesting perfect matching with the unique noncrossing matching in the same
$\sigma$-equivalence class, and Theorem 5.5 of \cite{russell19} follows from
Theorem \ref{thm:nesttri} as a special case.  The restriction of the
quasiparabolic order to the noncrossing basis gives rise to the {\it web graph}
of \cite[Section 2.3]{russell19}, which therefore has the structure of a
distributive lattice by Theorem \ref{thm:nonnest} (ii).  Our definition of the
nesting number (Definition \ref{def:stats} (i)) agrees with the nesting number
of \cite{russell19} when restricted to noncrossing $n$-roots, and is inspired by
\cite{russell19}.  It also follows from \cite[Corollary 4.2]{hwang23} that if
$W$ has type $D_{2k}$ and we expand the maximally crossing $n$-root $\theta_C$
as a linear combination of the noncrossing basis, $\theta_C = \sum \lambda_\be
\be$, then the sum $\sum \lambda_\be$ of the nonnegative integers $\lambda_\be$
is given by the number $E_{k+1}$ in the family (1,1,1,2,5,16,272, \dots) of
\emph{Euler numbers}, which are characterized by the equation $$ \sec(x) +
\tan(x) = \sum_{i = 0}^\infty E_i \frac{x^i}{i!} .$$ Coefficients in the
expansion of the maximally crossing $2k$-root $\theta_C$ into the noncrossing
basis have a combinatorial interpretation in terms of the so-called web
permutations in  $S_k$ by \cite[Theorem 1.2]{hwang23}.
\end{rmk}

\subsection{Type \texorpdfstring{$E_7$}{E7}}
\label{sec:e7}
Suppose $W$ has type $E_7$.  We define $\nu_A$ to be the positive $7$-root with
the following components: \[ \al_2, \ \al_4, \ \al_6, \ \al_7, \ \al_2 +
2\al_3 + \al_4 + \al_7, \] \[ \al_2 + 2\al_3 + 2\al_4 + 2\al_5 + \al_6 + \al_7,
    \ 2\al_1 + 3\al_2 + 4\al_3 + 3\al_4 + 2\al_5 + \al_6 + 2\al_7 .\]

We define $\nu_C$ to be the positive $7$-root with the following components: \[
    \alpha_2 + \alpha_3 + \alpha_4 + \alpha_5 + \alpha_7,\quad \alpha_2 +
    2\alpha_3 + \alpha_4 + \alpha_7,\quad           \alpha_3 + \alpha_4 +
    \alpha_5 + \alpha_6 + \alpha_7,\quad        \alpha_1 + \alpha_2 + \alpha_3 +
\alpha_4 + \alpha_7 \] \[ \alpha_1 + \alpha_2 + 2\alpha_3 + \alpha_4 + \alpha_5
    + \alpha_7,\quad   \alpha_1 + 2\alpha_2 + 2\alpha_3 + \alpha_4 + \alpha_5 +
    \alpha_6 + \alpha_7,\quad    \alpha_1 + 2\alpha_2 + 3\alpha_3 + 3\alpha_4 +
2\alpha_5 + \alpha_6 + \alpha_7.  \]

We define $\nu_N$ to be the positive 7-root with the following components: $$
\alpha_7,\quad \alpha_2 + 2\alpha_3 + \alpha_4 + \alpha_7,\quad \alpha_1 +
\alpha_2 + 2\alpha_3 + \alpha_4 + \alpha_5 + \alpha_7,\quad \alpha_1 + \alpha_2
+ 2\alpha_3 + 2\alpha_4 + \alpha_5 + \alpha_6 + \alpha_7, $$ $$ \alpha_1 +
2\alpha_2 + 2\alpha_3 + \alpha_4 + \alpha_5 + \alpha_6 + \alpha_7, \quad
\alpha_1 + 2\alpha_2 + 2\alpha_3 + 2\alpha_4 + \alpha_5 + \alpha_7,\quad
\alpha_2 + 2\alpha_3 + 2\alpha_4 + 2\alpha_5 + \alpha_6 + \alpha_7 .$$

Finally, we define the element $w\in W$ to be the element expressed by the word
$$\bfw = (s_1s_3s_5)(s_2s_4)(s_3)(s_7).$$ The heap of $w$ is shown in Figure
\ref{fig:heaps} (b).

\begin{prop}\label{prop:e7answer}
If $W$ has type $E_7$, then the $7$-roots $\nu_A$, $\nu_C$ and $\nu_N$ given
above are  respectively the maximally aligned, maximally crossing, and maximally
nesting $7$-roots of $W$.  The element $w$ is the nonnesting element $w_N$, and $\bfw$
a reduced word for it.  The Macdonald representation $\macd{E_7}{7A_1}$ has
dimension $15$.
\end{prop}

\begin{proof}
Recall from Corollary \ref{cor:mcount} that the number of \coplanar quadruples
in any $7$-root is $M=7$.  Direct verification shows that
$\rootht(\sigma(\nu_C)) - \rootht(\sigma(\nu_A))=49-35 = 2M$ and that
$w(\nu_C)=\nu_A$, which implies the assertions about $\nu_A$, $\nu_C$,
$\bfw$ and  $w$ by Theorem \ref{thm:nonnest} (iii). The dimension of the
Macdonald representation $\macd{E_7}{7A_1}$ equals the cardinality of the set
$\{v\in W: v\le_L w\}$ by Theorem \ref{thm:indep} (i) and Theorem
\ref{thm:nonnest} (ii). As explained in Section \ref{sec:D}, this set is in
bijection with the order filters of the heap of the fully commutative $w_N$, and
direct computation shows that this heap has 15 filters, so the dimension of
$\macd{E_7}{7A_1}$ is 15.


It remains to show that $\nu_N$ is the maximally aligned $n$-root $\theta_A$.
By inspection, the heights of the components of $\nu_N$ are 1, 5, 7, 9, 9, 9, 9
when listed in increasing order. The sum of the first three terms of this
sequence is bigger than the largest term, so $\nu_N$ cannot contain any
alignments by Proposition \ref{prop:htseq} (ii).  If $\nu_N$ contains a
crossing, it follows from Proposition \ref{prop:htseq} (iv) that the crossing
cannot contain any component of height 1, and that the crossing can contain at
most one component of height 9. It then follows from the listed heights that
$\nu_N$ contains no crossing either.  It follows that all the $M$ coplanar
quadruples in $\nu_N$ are nestings, so that $\nu_N=\theta_N$ by Proposition
\ref{prop:minmax} (ii).
\end{proof}

The set $X_I$ of alignment-free positive $7$-roots in type $E_7$ is intimately
related to the combinatorics of the \emph{Fano plane} (Figure \ref{fig:fano}),
the finite projective plane of order 2 over the field $\mathbb{F}_2$ with two
elements. We recall that any two points in the Fano share a unique line that
contains them both, so that the vertex labellings of the Fano points using the
labels $1,2,..,7$ correspond precisely to the Steiner triple systems $S(2,3,7)$
via the bijection that associates each line in the Fano plane with the triple of
labels for the vertices in that line. For example, the labellings shown in
Figure \ref{fig:fano} (a) and (b) correspond respectively to the Steiner systems
$L_C$ and $L_N$ from Proposition \ref{prop:e7top}. It is well known that the
automorphism group of the Fano plane is the simple group $GL(3, 2)$ of order
$168$, so that the number of inequivalent vertex labellings is $7!/168=30$.

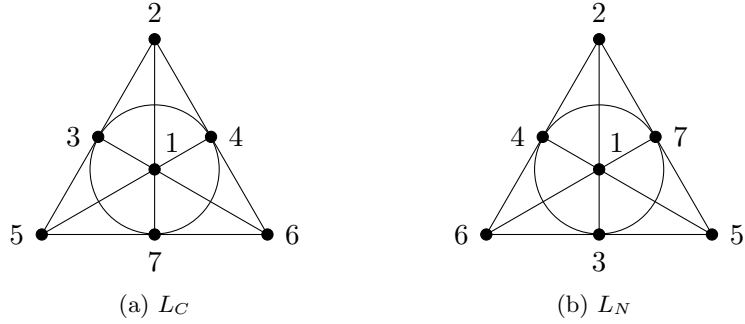
\begin{figure}[h!]
\centering
\subfloat[$L_C$]
{
\begin{tikzpicture}[
mydot/.style={
  draw,
  circle,
  fill=black,
  inner sep=1.5pt}
]
\draw
  (0,0) coordinate (A) --
  (3,0) coordinate (B) --
  ($ (A)!.5!(B) ! {sin(60)*2} ! 90:(B) $) coordinate (C) -- cycle;
\coordinate (O) at
  (barycentric cs:A=1,B=1,C=1);
\draw (O) circle [radius=3*1.717/6];
\draw (C) -- ($ (A)!.5!(B) $) coordinate (LC); 
\draw (A) -- ($ (B)!.5!(C) $) coordinate (LA); 
\draw (B) -- ($ (C)!.5!(A) $) coordinate (LB); 
\foreach \Nodo in {A,B,C,O,LC,LA,LB}
  \node[mydot] at (\Nodo) {};    
  \node [left=0.1cm of A] {$5$};
  \node [right=0.1cm of B] {$6$};
  \node [above=0.1cm of C] {$2$};
  \node [right=0.1cm of LA] {$4$};
  \node [left=0.1cm of LB] {$3$};
  \node [below=0.1cm of LC] {$7$};
  \node [above right=0.1cm and 0.001cm of O] {$1$};
\end{tikzpicture}
}
\quad\quad\quad\quad
\subfloat[$L_N$]
{
\begin{tikzpicture}[
mydot/.style={
  draw,
  circle,
  fill=black,
  inner sep=1.5pt}
]
\draw
  (0,0) coordinate (A) --
  (3,0) coordinate (B) --
  ($ (A)!.5!(B) ! {sin(60)*2} ! 90:(B) $) coordinate (C) -- cycle;
\coordinate (O) at
  (barycentric cs:A=1,B=1,C=1);
\draw (O) circle [radius=3*1.717/6];
\draw (C) -- ($ (A)!.5!(B) $) coordinate (LC); 
\draw (A) -- ($ (B)!.5!(C) $) coordinate (LA); 
\draw (B) -- ($ (C)!.5!(A) $) coordinate (LB); 
\foreach \Nodo in {A,B,C,O,LC,LA,LB}
  \node[mydot] at (\Nodo) {};    
  \node [left=0.1cm of A] {$6$};
  \node [right=0.1cm of B] {$5$};
  \node [above=0.1cm of C] {$2$};
  \node [right=0.1cm of LA] {$7$};
  \node [left=0.1cm of LB] {$4$};
  \node [below=0.1cm of LC] {$3$};
  \node [above right=0.1cm and 0.001cm of O] {$1$};
\end{tikzpicture}}
       \caption{The inequivalent labellings of the Fano plane corresponding to
        $\theta_C$ and $\theta_N$}
\label{fig:fano}
\end{figure}

\begin{prop}\label{prop:e7top}
If $W$ has type $E_7$, then every component of every 7-root $\gamma\in X_I$ has
the form \[ \eta_{abc}= \left( \sum_{i = 0}^7 \ep_i \right) - 2 (\ep_0 + \ep_a +
\ep_b + \ep_c) \] for a 3-element subset $abc:=\{a,b,c\}$ of the set [7], and
the map $\varphi$ sending each 7-root $\gamma\in X_I$ to the set
\[L_\gamma=\{abc: \eta_{abc} \,\vert\,\gamma\} \] gives a canonical bijection
from $X_I$ to the 30 inequivalent labellings of the Fano plane.  Under this
bijection, the  minimal element $\theta_C$ of $X_I$ corresponds to the labelling
$$ L_C=L_{\theta_C}=\{136, 145, 127, 235, 246, 347, 567\} ,$$ and the maximal
element $\theta_N$ corresponds to the labelling $$ L_N=L_{\theta_N}=\{123, 145,
246, 257, 347, 356, 167\} .$$
\end{prop}

\begin{proof}
In the Fano coordinates, the components of the maximally crossing and maximally
nesting $n$-roots $\theta_C$ and $\theta_N$ are given by the rows of the
following matrices $M_C$ and $M_N$, respectively, where each ``$+$'' stands for
1 and each ``$-$'' stands for $-1$ for brevity. 
\begin{equation} M_C = \left[ \begin{matrix} {\text{--}} & {\text{--}} &
      {\text{+}} & {\text{--}} & {\text{+}} & {\text{+}} & {\text{--}} &
      {\text{+}}\\ {\text{--}} & {\text{--}} & {\text{+}} & {\text{+}} &
      {\text{--}} & {\text{--}} & {\text{+}} & {\text{+}}\\ {\text{--}} &
      {\text{--}} & {\text{--}} & {\text{+}} & {\text{+}} & {\text{+}} &
      {\text{+}} & {\text{--}}\\ {\text{--}} & {\text{+}} & {\text{--}} &
      {\text{--}} & {\text{+}} & {\text{--}} & {\text{+}} & {\text{+}}\\
      {\text{--}} & {\text{+}} & {\text{--}} & {\text{+}} & {\text{--}} &
      {\text{+}} & {\text{--}} & {\text{+}}\\ {\text{--}} & {\text{+}} &
      {\text{+}} & {\text{--}} & {\text{--}} & {\text{+}} & {\text{+}} &
      {\text{--}}\\ {\text{--}} & {\text{+}} & {\text{+}} & {\text{+}} &
      {\text{+}} & {\text{--}} & {\text{--}} & {\text{--}}\\
\end{matrix} \right], \; M_N = \left[ \begin{matrix} {\text{--}} & {\text{--}}
    & {\text{--}} & {\text{--}} & {\text{+}} & {\text{+}} & {\text{+}} &
    {\text{+}}\\ {\text{--}} & {\text{--}} & {\text{+}} & {\text{+}} &
    {\text{--}} & {\text{--}} & {\text{+}} & {\text{+}}\\ {\text{--}} &
    {\text{+}} & {\text{--}} & {\text{+}} & {\text{--}} & {\text{+}} &
    {\text{--}} & {\text{+}}\\ {\text{--}} & {\text{+}} & {\text{--}} &
    {\text{+}} & {\text{+}} & {\text{--}} & {\text{+}} & {\text{--}}\\
    {\text{--}} & {\text{+}} & {\text{+}} & {\text{--}} & {\text{--}} &
    {\text{+}} & {\text{+}} & {\text{--}}\\ {\text{--}} & {\text{+}} &
    {\text{+}} & {\text{--}} & {\text{+}} & {\text{--}} & {\text{--}} &
    {\text{+}}\\ {\text{--}} & {\text{--}} & {\text{+}} & {\text{+}} &
    {\text{+}} & {\text{+}} & {\text{--}} & {\text{--}}\\
\end{matrix} \right]
\end{equation}
By inspection, the components of $\theta_C$ have the properties (a) each of them
is a root of the form $\eta_{abc}$ for some triple $abc\se [7]$, and (b) the
triples corresponding to components form a Steiner triple system.

By Proposition \ref{prop:e7answer}, the rightmost generator appearing in $w_N$
is $s_7$, which implies that $I = S \backslash \{s_7\}$. It follows from Section
\ref{sec:constructions} that $W_I$ is a Weyl group of type $A_6$, isomorphic to
$S_7$, and that $W_I$ acts on $X_I$ by permuting the Fano coordinates. Since all
elements of $X_I$ are conjugate to $\theta_C$ under the action of $W_I$ by
Proposition \ref{prop:neststab} (ii) and (iii), it now follows from the previous
paragraph that for every 7-root $\gamma\in X_I$, the components of $\gamma$
satisfy the properties (a) and (b) satisfied by the components of $\theta_C$.
This implies that the map $\varphi$ takes each element to a Steiner triple
system (and thus one of the 30 inequivalent labellings of the Fano plane).  The
map $\varphi$ is clearly injective, and it is surjective because all Steiner
triple systems are isomorphic via the permutation action of $S_7$ by Remark
\ref{rmk:steiner_unique}. This proves the first sentence of the proposition. The
second sentence holds by inspection of the matrices $M_C$ and $M_N$.
\end{proof}

\begin{rmk}\label{rmk:labellings}
The labellings canonically corresponding to $X_I$ have the following additional
properties.
\begin{enumerate}
    \item[{\rm (i)}] The triples $ijk$ in the labelling $L_C$ corresponding to
         the 7-root $\theta_C$ appear in \cite[Section IV]{talamini10} as the
         triples indexing the ``globally invariant linear forms'' $\pm x_i\pm
         x_j\pm x_k$ of type $E_7$.

\item [{\rm (ii)}] The labelling $L_N$ corresponding to $\theta_N$ is the unique
        labelling with the property that if the digits are written in binary, then the third
  digit of each triple is the bitwise exclusive or (XOR) of the other two.
  \item[{\rm (iii)}] Recall from Remark \ref{rmk:xht} that
       $X_I$ naturally splits into two equal-sized components that are
       interchanged by the action of a reflection in $W_I$. As discussed in
       \cite{saniga16}, any two distinct labellings in the same component have
       precisely one triple in common.
   \item[{\rm (iv)}] The level $\lambda(\gamma)$ of each 7-root $\gamma$ in
       $X_I$ equals $(14-d)$, where $d$ is the number of 3-element subsets $E$ of
        the set $[7]$ with the property that the labelling $L_\gamma$ contains
        no blocks of the form $(E\setminus\{j\})\cup \{i\}$ where $i\le j$ and
        $j\in E$. This fact can be verified computationally, and is similar to
        the interpretation of the level function in type $D_{2k}$ via Steiner
        systems $S(1,2,2k)$ given in Section \ref{sec:D}.
\end{enumerate}
\end{rmk}

\begin{rmk}\label{rmk:e7chen} The noncrossing basis in type $E_7$ is illustrated
      by the diagram labelled ${\mathfrak M}_6$ in \cite[Appendix]{chen98},
      where each rectangle can be identified with a noncrossing basis element
      $\be$. A label of $i$ on a rectangle indicates that $\al_i | \beta$, i.e.,
      that $\al_i$ is a component of $\be$. The edges connecting rectangles
      refer to {\it star operations} in the sense of \cite{kl79}, which can be
      interpreted directly in terms of $n$-roots as follows. If $i$ and $j$ are
      adjacent vertices of the Dynkin diagram, then two noncrossing basis
      elements $\be$ and $\be'$ such that $\al_i | \be$ and $\al_j | \be'$ are
      joined by an edge if we have $s_is_j(\be) = \be'$ or, equivalently,
      $s_js_i(\be') = \be$. (A similar construction appears in \cite[Lemma
      2.8]{gx3}.) Note that the Dynkin diagram of type $E_7$ in \cite{chen98}
      differs from the Dynkin diagram of type $E_7$ shown in Figure
      \ref{fig:ade} (d) in the labelling of vertices, but it can be obtained by
      removing the vertex ``8'' and its incident edge from the Dynkin diagram of
      type $E_8$ shown in Figure \ref{fig:ade} (e).
  \end{rmk}

\begin{rmk}\label{rmk:utterly}
Ren--Sam--Schrader--Sturmfels \cite[Theorem 4.1]{ren13} give an ``utterly
explicit" basis for the 15-dimensional Macdonald representation in type $E_7$
that is natural in the context of the G\"opel variety in algebraic geometry.
The elements of the nonnesting basis and the noncrossing basis in type $E_7$ all
factorize into linear factors in $\Sym(\fh^*)$ by construction, but not all the
basis elements of \cite[Theorem 4.1]{ren13} do, even after extending scalars to
$\C$. It follows that the basis of \cite{ren13} is not the same as either the
noncrossing basis or the nonnesting basis, even after applying a change of basis
of $\fh^*$.
\end{rmk}

\subsection{Type \texorpdfstring{$E_8$}{E8}}\label{sec:e8}
Suppose $W$ has type $E_8$.  We define $\nu_A$ to be the positive $8$-root with
the following components: $$ \al_2, \ \al_3, \ \al_5, \ \al_7,\ \al_2 + \al_3 +
2\al_4 + \al_5, $$ $$\al_2 + \al_3 + 2\al_4 + 2\al_5 + 2\al_6 + \al_7, \ 2\al_1
+ 2\al_2 + 3\al_3 + 4\al_4 + 3\al_5 + 2\al_6 + \al_7, $$ $$ 2\al_1 + 3\al_2 +
4\al_3 + 6\al_4 + 5\al_5 + 4\al_6 + 3\al_7 + 2\al_8 .$$

We define $\nu_C$ to be the positive $8$-root with the following components: $$
\alpha_1 + \alpha_2 + \alpha_3 + \alpha_4 + \alpha_5 + \alpha_6 + \alpha_7 +
\alpha_8,\quad      \alpha_1 + \alpha_2 + \alpha_3 + 2\alpha_4 + \alpha_5 +
\alpha_6 + \alpha_7, $$ $$ \alpha_1 + \alpha_2 + \alpha_3 + 2\alpha_4 +
2\alpha_5 + \alpha_6,\quad \alpha_1 + \alpha_2 + 2\alpha_3 + 2\alpha_4 +
\alpha_5 + \alpha_6, $$ $$ \alpha_1 + \alpha_2 + 2\alpha_3 + 2\alpha_4 +
2\alpha_5 + \alpha_6 + \alpha_7,\quad \alpha_1 + \alpha_2 + 2\alpha_3 +
3\alpha_4 + 2\alpha_5 + \alpha_6 + \alpha_7 + \alpha_8, $$ $$ \alpha_1 +
\alpha_2 + 2\alpha_3 + 3\alpha_4 + 3\alpha_5 + 3\alpha_6 + 2\alpha_7 +
\alpha_8,\quad \alpha_1 + 3\alpha_2 + 3\alpha_3 + 5\alpha_4 + 4\alpha_5 +
3\alpha_6 + 2\alpha_7 + \alpha_8.  $$

We define $\nu_N$ to be the positive $8$-root with the following components: $$
\alpha_1,\quad \alpha_1 + \alpha_2 + 2\alpha_3 + 2\alpha_4 + \alpha_5,\quad
\alpha_1 + \alpha_2 + 2\alpha_3 + 2\alpha_4 + 2\alpha_5 + 2\alpha_6 + \alpha_7,
$$ $$ \alpha_1 + \alpha_2 + 2\alpha_3 + 3\alpha_4 + 2\alpha_5 + 2\alpha_6 +
\alpha_7 + \alpha_8,\quad \alpha_1 + \alpha_2 + 2\alpha_3 + 3\alpha_4 +
3\alpha_5 + 2\alpha_6 + 2\alpha_7 + \alpha_8, $$ $$ \alpha_1 + 2\alpha_2 +
2\alpha_3 + 3\alpha_4 + 2\alpha_5 + 2\alpha_6 + 2\alpha_7 + \alpha_8,\quad
\alpha_1 + 2\alpha_2 + 2\alpha_3 + 3\alpha_4 + 3\alpha_5 + 2\alpha_6 + \alpha_7
+ \alpha_8, $$ $$ \alpha_1 + 2\alpha_2 + 2\alpha_3 + 4\alpha_4 + 3\alpha_5 +
2\alpha_6 + \alpha_7 .$$

Finally, we define the element $w\in W$ to be the element expressed by the word
$$\bfw = (s_1s_4s_6s_8)(s_3s_5s_7)(s_4s_6)(s_2s_5)(s_4)(s_3)(s_1).$$ The heap of
$w$ is shown in Figure \ref{fig:heaps} (c).

\begin{prop}\label{prop:e8answer}
If $W$ has type $E_8$, then the $8$-roots $\nu_A$, $\nu_C$ and $\nu_N$ given
above are respectively the maximally aligned, maximally crossing, and maximally
nesting $8$-roots of $W$.  The element $w$ is the nonnesting element $w_N$, and $\bfw$
a reduced word for it.  The Macdonald representation $\macd{E_8}{8A_1}$ has
dimension $50$.
\end{prop}

\begin{proof}
The statements can be proved using the same strategy used in the proof of
Proposition \ref{prop:e7answer} except for the following changes in numerical
details. The number $M$ of coplanar quadruples in an $n$-root is now 14, and the
number of order filters in the heap of the nonnesting element $w_N$ is 50. The
components of $\nu_N$ have heights 1, 7, 11, 13, 15, 15, 15,
and 15, which implies that $\nu_N$ has no alignments by Proposition
\ref{prop:htseq} (ii).  Furthermore, if $\nu_N$ had a crossing, then Proposition
\ref{prop:htseq} (iv) implies that the only possibility would be for the crossing
to contain roots of heights 7, 11, 13, and 15, but this is not possible
either because $11+13 > 7+15$.
\end{proof}

\begin{rmk}\label{rmk:schmidtv}
    \begin{enumerate}
        \item [{\rm (i)}]
Schmidt \cite[Lemma 3.4]{schmidt24} gives an explicit partition of the 120
positive-negative pairs of roots in type $E_8$ into 15 sets of size 8. The
components of $\theta_A$, $\theta_N$, and $\theta_C$ appear in Schmidt's list as
$V_1$, $V_{14}$, and $V_{15}$, respectively.
        \item [{\rm (ii)}] In standard coordinates, the components of the
             maximally crossing and maximally nesting $8$-roots $\theta_C=\nu_C$
             and $\theta_N=\nu_N$ are given by the rows of the following
             matrices $M'_C$ and $M'_N$, respectively, where each ``$+$'' stands for
1 and each ``$-$'' stands for $-1$ for brevity. 

         \[ M'_C = \left[
                     \begin{matrix} {\text{+}} & {\text{+}} & {\text{--}} &
                         {\text{--}} & {\text{--}} & {\text{--}} &
                         {\text{+}} & {\text{+}}\\ {\text{+}} & {\text{--}}
                         & {\text{+}} & {\text{--}} & {\text{--}} &
                         {\text{+}} & {\text{--}} & {\text{+}}\\ {\text{+}}
                         & {\text{--}} & {\text{--}} & {\text{+}} &
                         {\text{+}} & {\text{--}} & {\text{--}} &
                         {\text{+}}\\ {\text{--}} & {\text{+}} & {\text{+}}
                         & {\text{--}} & {\text{+}} & {\text{--}} &
                         {\text{--}} & {\text{+}}\\ {\text{--}} & {\text{+}}
                         & {\text{--}} & {\text{+}} & {\text{--}} &
                         {\text{+}} & {\text{--}} & {\text{+}}\\ {\text{--}}
                         & {\text{--}} & {\text{+}} & {\text{+}} &
                         {\text{--}} & {\text{--}} & {\text{+}} &
                         {\text{+}}\\ {\text{--}} & {\text{--}} &
                         {\text{--}} & {\text{--}} & {\text{+}} & {\text{+}}
                         & {\text{+}} & {\text{+}}\\ {\text{+}} & {\text{+}}
                         & {\text{+}} & {\text{+}} & {\text{+}} & {\text{+}}
                         & {\text{+}} & {\text{+}}\\
                 \end{matrix} \right], \; M'_N = \left[ \begin{matrix}
    {\text{+}} & {\text{--}} & {\text{--}} & {\text{--}} & {\text{--}} &
    {\text{--}} & {\text{--}} & {\text{+}}\\ {\text{--}} & {\text{+}} &
    {\text{+}} & {\text{+}} & {\text{--}} & {\text{--}} & {\text{--}} &
    {\text{+}}\\ {\text{--}} & {\text{+}} & {\text{--}} & {\text{--}} &
    {\text{+}} & {\text{+}} & {\text{--}} & {\text{+}}\\ {\text{--}} &
    {\text{--}} & {\text{+}} & {\text{--}} & {\text{+}} & {\text{--}} &
    {\text{+}} & {\text{+}}\\ {\text{--}} & {\text{--}} & {\text{--}} &
    {\text{+}} & {\text{--}} & {\text{+}} & {\text{+}} & {\text{+}}\\
    {\text{+}} & {\text{+}} & {\text{+}} & {\text{--}} & {\text{--}} &
    {\text{+}} & {\text{+}} & {\text{+}}\\ {\text{+}} & {\text{+}} &
    {\text{--}} & {\text{+}} & {\text{+}} & {\text{--}} & {\text{+}} &
    {\text{+}}\\ {\text{+}} & {\text{--}} & {\text{+}} & {\text{+}} &
    {\text{+}} & {\text{+}} & {\text{--}} & {\text{+}}\\
\end{matrix} \right]  \] All rows in $M'_C$ other than the bottom row contain
    four ``$+$'' and four ``$-$'', and the 14 quadruples recording the column
    numbers of the positive and negative entries in these rows form a  Steiner
    quadruple system. Furthermore, these 14 quadruples are precisely the ones
    indexing the ``globally invariant linear forms''  of type $E_8$ in
    \cite[Section V]{talamini10}.
\item [{\rm (iii)}] The matrix $M'_C$ above is a {\it Hadamard matrix}, meaning
     a matrix with entries in $\{+1, -1\}$ that has orthogonal rows (and,
     therefore, orthogonal columns). By rearranging the rows, the matrix can be
     expressed more simply as the Kronecker product $H \otimes H \otimes H$,
     where $$H = \left[
         \begin{matrix} \text{+1} & \text{+1} \\ \text{--1} & \text{+1}
     \end{matrix} \right].$$
    \end{enumerate}
\end{rmk}

\begin{rmk}\label{rmk:e8chen}
The noncrossing basis in type $E_8$ is illustrated by the diagram labelled
${\mathfrak M}_{50}$ in \cite[Appendix]{chen00}.
\end{rmk}

The set $X_I$ of alignment-free positive $8$-roots in type $E_8$ can be used to
give a convenient realization of the graph $\bar\Gamma_1$ studied by Schmidt in
\cite{schmidt24}. The graph $\bar\Gamma_1$, which is the complement of another
graph $\Gamma_1$, has the property that it is quantum isomorphic (in the sense
of \cite{atserias19}) but not isomorphic to the orthogonality graph $G_{E_8}$ of
the roots of type $E_8$. The vertices of $G_{E_8}$ are the 120 positive roots of
type $E_8$, and two roots are adjacent in $G_{E_8}$ if and only if they are
orthogonal.

To realize the graph $\bar\Gamma_1$ via $X_I$, recall from Remark \ref{rmk:xht}
that $X_I$ naturally splits into two equal-sized components, $X_I^e$ and
$X_I^o$, which consist of all the elements in $X_I$ with even levels and odd
levels, respectively. The components each have $240/2=120$ elements since
$\abs{X_I}=240$ by Proposition \ref{prop:poincare} (ii).
\begin{defn}
    \label{def:Gamma} We define $\Gamma$ to be the following graph: the vertex
          set is $X_I^e$, the set of all alignment-free positive $8$-roots of
          even parabolic level, and two vertices are adjacent if and only if
          they have no common components.
      \end{defn}
\noindent The next four paragraphs recall Schmidt's construction of the graph
         $\Gamma_1$ and explain why $\Gamma$ is isomorphic to $\bar\Gamma_1$.
         (The isomorphism will also hold if we replace $X_I^e$ with $X_I^o$ in
         Definition \ref{def:Gamma}.)

Start with the folded halved 8-cube, where the vertices are the 64 pairs of the
form $\{x,\mathbf{1}+x\}$ for all length-8 binary strings $x\in \mathbb{F}_2^8$
with an even number of 1s (where $\mathbf{1}$ is the string with all entries
equal to $1$). We can naturally identify these vertices with the 64 positive
roots of the form $\al=(\sum_{i=1}^7\pm\ep_i)+\ep_8$, via the bijection sending
$\al$ to the pair $\{x_\al,\mathbf{1}+x_\al\}$ where $x_\al$ is the string
$(x_i)_{i=1}^8$ such that $x_i=1$ if and only if $\ep_i$ appears with
coefficient $-1$ in $\al$ for all $i\in[8]$. The 64 positive roots of the form
$(\sum_{i=1}^7\pm\ep_i)+\ep_8$  are precisely the positive roots of $x$-height 1
in the sense of Remark \ref{rmk:xht}, so they are also precisely the roots that
can appear as a component of an 8-root in $X_I$ by Remark \ref{rmk:xht}.

By definition, two vertices $\{x,\mathbf{1}+x\}$ and $\{y,\mathbf{1}+y\}$ in the
folded halved cube are adjacent if and only if $x$ and $y$ differ in 2 or 6
entries. It follows that in the complement of the folded halved cube, two
distinct vertices $\{x,\mathbf{1}+x\}$ and $\{y,\mathbf{1}+y\}$ are adjacent if
and only if $x$ and $y$ differ in 4 positions. This complement is denoted
$VO_6^+(2)$. The condition that $x$ and $y$ differ in 4 entries holds if and
only if the positive roots corresponding to $\{x,\mathbf{1}+x\}$ and
$\{y,\mathbf{1}+y\}$  are orthogonal; therefore each clique of size 8 in
$VO_6^+(2)$ corresponds to an 8-root in $X_I$.

Schmidt defines the vertex set of $\Gamma_1$ to be any orbit of cliques of size
8 under the group $\Z_2^6\rtimes A_8$ in $VO_6^+(2)$, where $A_8$ is the
alternating subgroup of $S_8$. There are two such orbits, both of size 120, and
the choice of the orbit does not matter, so we may assume that the orbit
contains a clique corresponding to an 8-root with even level, i.e., to a vertex,
$\beta$, of $\Gamma$. The vertices of $\Gamma_1$ then match precisely the
vertices of $\Gamma$ for the reasons sketched below. We have
$I=S\setminus\{s_1\}$ by Proposition \ref{prop:e8answer}; therefore we have
$W_I\cong W(D_7)$. The action of $W_I = W(D_7)$ on $X_I$ can be extended to an
action of a larger subgroup $G \leq W(E_8)$, generated by $W_I$ together with
the reflection $s_\theta$ corresponding to the highest root
$\theta=2(\ep_7+\ep_8)$.  We have $G \cong W(D_8) \cong N \rtimes S_8$, where $N
\cong \Z_2^7$ is the elementary abelian group of order $2^7$. By considerations
involving $x$-heights (in the sense of Remark \ref{rmk:xht}), each reflection in
$G$ changes every 8-root in $X_I$, and when it does so it changes the parabolic
level by an odd number because it moves a $C$ to an $N$ or vice versa.  It
follows that the commutator subgroup $G'\cong \Z_2^7\rtimes A_8$ of $G$ acts on
$X_I$ with $X_I^{e}$ and $X_I^{o}$ as its orbits.  This action induces a
transitive action of $G'/Z(G) \cong \Z_2^6\rtimes A_8$ on $\Gamma$ that matches
the action of $\Z_2^6\rtimes A_8$ on $\Gamma_1$.

Two vertices in $\Gamma_1$ are defined to be adjacent if and only if they are
cliques that intersect in exactly two elements from $VO_6^+(2)$. This occurs if
and only if their corresponding $8$-roots have two components in common. With
some more work, or using computation, one can show that two distinct 8-roots
whose levels have the same parity either have disjoint components or have
exactly two components in common. It follows that $\Gamma$ is isomorphic to
$\bar\Gamma_1$. To summarize, we have the following result.

\begin{rmk}
\label{rmk:quantumiso}
The graph $\Gamma$ from Definition \ref{def:Gamma}, i.e., the graph whose
vertices are the alignment-free positive $8$-roots of even parabolic level and
where two vertices are adjacent if and only if they have no common component, is
isomorphic to the graph $\bar\Gamma_1$ from \cite{schmidt24}. As a consequence,
the graph $\Gamma$ is quantum isomorphic but not isomorphic to the orthogonality
graph $G_{E_8}$ of the $E_8$ root system.
\end{rmk}

The graphs $\Gamma_{E_8}$ and $\bar\Gamma_1$ are known to be strongly
regular graphs with parameters $(120, 63, 30, 36)$. It follows that $\Gamma$ is
also such a graph.  Mathon and Street \cite[Table 2.2]{mathon97} mention that
the graphs $\Gamma_{E_8}$ and $\bar\Gamma_1$ each have 2025 $8$-cliques. The
8-cliques of $\Gamma_{E_8}$ are the positive $8$-roots of $E_8$, and the 8-cliques of $\Gamma\cong \bar\Gamma_1$
are classified by Fitz in \cite[Theorem 7.6]{fitz}. Finally, we
note that, as \cite{schmidt24} points out, there are other constructions of
$\Gamma_1$ in the literature \cite{brouwer89, mathon97}. However, the
construction in terms of $8$-roots has the advantages of being concise, and
being clearly related to the $E_8$ root system.

\begin{rmk}
    \label{rmk:rank4}
The group $\Aut(\Gamma_1)$ acts as a permutation group of rank 4 on $\Gamma_1$
\cite{brouwer89}.  It follows that if two $n$-roots $x,y\in X_I$ both have even
or odd levels, then $x$ and $y$ can be in one of four relative positions. These
can be shown to be the following (where $N$ is the elementary abelian group of
order $2^7$ mentioned above):
\begin{enumerate}
    \item[(i)]{$x=y$;}
    \item[(ii)]{$x$ and $y$ have precisely two common components;}
    \item[(iii)]{$x$ and $y$ have disjoint components,  and $y = n.x$ for some
             $n \in N$;}
         \item[(iv)]{$x$ and $y$ have disjoint components, and $y \ne n.x$ for
      any $n \in N$.}
\end{enumerate} The situations in (iii) and (iv) correspond to the edges in the
    graph $\Gamma$, and they show that the edges of $\Gamma$ naturally split into
    two types. This is not the case for the graph $G_{E_8}$: the automorphism
    group of $G_{E_8}$ has rank 3, and two vertices can only be in three
    relative positions: equality, adjacency, and non-adjacency.
\end{rmk}

\section{Concluding remarks}
\label{sec:conclusions}
\subsection{Poincar\'e polynomials}
Rains and Vazirani \cite[Section 8]{rains13} define the {\it Poincar\'e series}
of a quasiparabolic set $\mathcal{X}$ to be $PS(q) = \sum_{x \in \mathcal{X}}q^{\lambda(x)}$. They point
out that in many cases, the Poincar\'e series factorizes in a very simple way,
and the factors are often {\it quantum integers} $$ [d]_q := \frac{q^d-1}{q-1} =
1 + q + q^2 + \cdots + q^{d-1} .$$ These quantum integers often behave as if
they were the degrees of polynomial invariants of a Coxeter group, and in some
cases, the integers can be interpreted in terms of degrees of invariants in
characteristic 2 (see \cite[Section 8, Example 9.4]{rains13}).

\begin{prop}\label{prop:poincare}
Let $W$ be a Weyl group of type $E_7, E_8$, or $D_n$ for $n$ even.
\begin{enumerate}
    \item[{\rm (i)}] For the quasiparabolic set $X$ for $W$ consisting of all
     positive $n$-roots of $W$, equipped with the level function such that
     $\lambda(\theta_A)=0$,  we have $$ PS_X(q) =
     \begin{cases} \prod_{i = 2}^k [2i-1]_q & \text{\ if \ } W \text{\ is\ of\
      type\ } D_{2k},\\ [3]_q[5]_q[9]_q & \text{\ if \ } W \text{\ is\ of\ type\
      } E_7, \\ [3]_q[5]_q[9]_q[15]_q & \text{\ if \ } W \text{\ is\ of\ type\ }
      E_8.
  \end{cases} $$
\item[{\rm (ii)}] For the quasiparabolic set $X_I\se X$ for $W_I$ consisting of
     the alignment-free positive $n$-roots of $W$ (with its level function
     inherited from $X$), we have $$ PS_{X_I}(q) =
     \begin{cases} q^M\prod_{i = 2}^k [i]_q & \text{\ if \ } W \text{\ is\ of\
      type\ } D_{2k},\\ q^M[2]_q[3]_q[5]_q & \text{\ if \ } W \text{\ is\ of\
      type\ } E_7, \\ q^M[2]_q[3]_q[5]_q[8]_q & \text{\ if \ } W \text{\ is\ of\
      type\ } E_8,
  \end{cases} $$ where $M$ is the level $\lambda(\theta_C)$ of the minimal
    element of $X_I$ (which is also the number of coplanar quadruples in each
    $n$-root).  In particular, each of  the factors $[i]_q$ of $PS_{X_I}(q)$
    corresponds to a factor $[2i-1]_q$ of $PS_X(q)$ in (i).
    \end{enumerate}
\end{prop}

\begin{proof}
We have verified both (i) and (ii) in types $E_7$ and $E_8$ computationally. (We
do not have conceptual proofs at the moment.) For type $D_n$, part (i) follows
from \cite[Equation (5.4)]{simion96}, or \cite[Theorem 4]{diaz09}, or
\cite[Corollary 3.3]{cheon15}, after we identify $X$ with the perfect matchings
of the set $[n]$ as usual. Finally, part (ii) for type $D_n$ follows from
Proposition \ref{prop:abruhat} and the well-known form $\prod_{i=2}^k [i]_q$ for
the Poincar\'e series $\sum_{\tau\in S_k}q^{\ell(\tau)}$ of the symmetric group
$S_k$.
\end{proof}

\begin{rmk}
\label{rmk:cohomology}
The exponents $3, 5, 9$ and $3, 5, 9, 15$ that respectively appear in the
Poincar\'e series $PS_X(q)$ of types $E_7$ and $E_8$ show up as the degrees of
generators in the cohomology modulo 2 of compact exceptional Lie groups
\cite{araki61}, and as the codimensions of generators of Chow rings associated
to linear algebraic groups in characteristic 2 \cite[Section 4]{petrov08}.
\end{rmk}

\begin{rmk}\label{rmk:xs}
Recall from Proposition \ref{prop:sumclass} that the set $X_I$ is the top
$\sigma$-equivalence class with respect to the order $\le_\sigma$.  It turns out
that for every $\sigma$-equivalence class $C$ in $X$, the polynomial
$PS_C(q)=\sum_{x\in C}q^{\lambda(x)}$ has the form \begin{equation}
    \label{eq:everyclass} PS_C(q)=\prod_{d \in D} q^{d-1} [d]_q
\end{equation} for some set of nonnegative integers $D$.  This is particularly
    remarkable because in general, there is no obvious way to turn a
    $\sigma$-equivalence class into a $W$-set for a suitable Weyl group $W$.  In
    type $D_n$, the integers from $D$ have an interpretation in terms of rook
    placements \cite[Theorem 1]{watson14}. Summing over all $\sigma$-equivalence
    classes gives rise to the expression for $PS_X(q)$ that appears in the
    abstract of \cite{billera15}. In types $E_7$ and $E_8$, we verified Equation
    \eqref{eq:everyclass} by computation.
\end{rmk}

\subsection{Coxeter elements}
\label{sec:coxeter_elements}
Let $d_1, d_2, \ldots, d_r$ be the numbers that appear in the factorization
$\prod_{i=1}^r [d_i]_q$ of the Poincare series $PS_X(q)$ in Proposition
\ref{prop:poincare} (i).  It follows easily from the definitions that
$\prod_{i=1}^r d_i$ is the number of positive $n$-roots, and that $\sum_{i=1}^r
(d_i-1) = 2M$, where $M$ is the number of \coplanar quadruples in an $n$-root.
It also turns out that the largest integer $d_r$ in each case (which is $n-1$ in
type $D_n$, is $9$ in type $E_7$, and is $15$ in type $E_8$) is equal to $h/2$,
where $h$ is the Coxeter number.

We recall that, by definition, a Coxeter element is a product of all the simple
reflections in some order, and the Coxeter number is the order of any Coxeter
element $c$. All such elements are conjugate, and therefore have the same order.
It turns out that $c^{h/2}$ acts as $-1$ in the reflection representation, and
therefore acts trivially on the set $X$. If $C$ is the cyclic group of order
$h/2$ generated by the action of $c$ on the positive $n$-roots, then it
can be shown that the nonidentity elements of $C$ act without fixed points on
the positive $n$-roots. The factor of $[h/2]_q$ in $PS_X(q)$ then implies that
the triple $(X, PS_X(q), C)$ satisfies the {cyclic sieving phenomenon} of
Reiner, Stanton, and White \cite{reiner04}: the number of fixed points of $c^d$
is equal to $PS_X(e^{2\pi id/m})$, where $m=h/2$.

It is possible, by choosing a suitable Coxeter element $c$ and $n$-root $\be$,
to find an orbit of $n$-roots $$ O=\{c^d(\be) : 0 \leq d < h/2\} $$ that
contains every positive root exactly once as one of its components. This can be
achieved in type $D_n$ by taking $\be = \theta_N$ and $c = s_1 s_2 \cdots
s_{n-1}s_n$. We also verified that such an orbit can also be found in types
$E_7$ and $E_8$, although it is necessary to make a different choice of $\be$.
The existence of such an orbit $O$ in type $E_8$ is related to the
Kochen--Specker theorem in quantum mechanics \cite{waegell15}.

\subsection{Feature-avoidance via quasiparabolic structure}
\label{sec:abstract_fa}
It can be shown that each of the three types of feature-avoiding $n$-roots in
the set $X$ can be characterized using only the quasiparabolic structure of $X$,
without reference to the combinatorics of $n$-roots.  Specifically, the
following holds for all $n$-roots $x\in X$:
\begin{enumerate}
    \item [{\rm (i)}] $x$ is alignment-free if and only if there does not exist
         a reflection $r$ such that $\lambda(r(x))-\lambda(x)$ is a strictly
         positive even number;
     \item [{\rm (ii)}] $x$ is noncrossing if and only if there is a sequence $$
          x <_Q r_1(x) <_Q r_2r_1(x) <_Q \cdots <_Q x_1 $$ where $x_1$ is the
          unique maximal element of $X$ and the level increases by $2$ at each
          step;
      \item [{\rm (iii)}] $x$ is nonnesting if and only if there does not exist
           a reflection $r$ such that $\lambda(r(x)) - \lambda(x)$ is a strictly
           negative even number.
   \end{enumerate}
In addition, Remark \ref{rmk:shell6} shows that the $\sigma$-equivalence classes
can be characterized as the Eulerian intervals between nonnesting and
noncrossing elements.  It may be interesting to use these characterizations to
extend the notions of feature-avoiding elements and $\sigma$-equivalence to more
general quasiparabolic sets.

\section*{Acknowledgements}
We are grateful to the referees for reading the paper carefully and suggesting
many improvements. We also thank Dana Ernst, Emily King, Heather Russell, and
Nathaniel Thiem for helpful conversations.

\end{document}